\documentclass[11pt]{amsart}

\usepackage{epigamath}

\usepackage[english]{babel}

\numberwithin{equation}{section}

\usepackage[shortlabels]{enumitem}

\usepackage[utf8]{inputenc}
\usepackage[T1]{fontenc}
\usepackage{color,soul}

\usepackage{tikz}
\usetikzlibrary{arrows,positioning,decorations.markings}
\tikzset{>=angle 90}

\usepackage{tikz-cd}
\usepackage{stmaryrd}

\usepackage{amsmath,amssymb,amsthm,mathrsfs,stmaryrd,color}
\usepackage{amsfonts}
\usepackage[all]{xy}
\usepackage{mathtools}

\usepackage{mymath}

\RequirePackage{ifthen}
\RequirePackage{verbatim}

\usepackage{relsize}
\usepackage[bbgreekl]{mathbbol}

\theoremstyle{plain}
\newtheorem{theorem}{Theorem}[section]
\newtheorem{thm}[theorem]{Theorem}
\newtheorem{lem}[theorem]{Lemma}
\newtheorem{cor}[theorem]{Corollary}
\newtheorem{prop}[theorem]{Proposition}
\newtheorem{propdef}[theorem]{Proposition and Definition}
\newtheorem{defprop}[theorem]{Definition and Proposition}

\theoremstyle{definition}
\newtheorem{notation}[theorem]{Notation}

\newtheorem{defi}[theorem]{Definition}
\newtheorem{construction}[theorem]{Construction}

\theoremstyle{remark}
\newtheorem{remark}[theorem]{Remark}
\newtheorem{rem}[theorem]{Remark}
\newtheorem{example}[theorem]{Example}
\newtheorem{remdef}[theorem]{Remark and Definition}
\newtheorem{defrem}[theorem]{Definition and Remark}

\newcommand{\PPerf}{\textup{Perf}}
\DeclareMathOperator{\MMod}{Mod}
\DeclareMathOperator{\QQCoh}{QCoh}
\newcommand{\cn}{\textup{cn}}
\DeclareMathOperator{\Equiv}{Equiv}
\DeclareMathOperator{\cofib}{cofib}
\newcommand{\cl}{\textup{cl}}
\DeclareMathOperator{\fib}{fib}
\DeclareMathOperator{\Poly}{Poly}

\newcommand{\dSt}{\textup{dSt}}
\newcommand{\Shv}{\textup{Shv}}
\newcommand{\pPerf}{\textup{perf}}
\DeclareMathOperator{\codom}{codom}

\DeclareMathOperator{\dSch}{dSch}

\DeclareMathOperator{\Ch}{Ch}

\DeclareMathOperator{\Ani}{\textup{Ani}}
\newcommand{\AniAlg}[1]{\textup{AR}_{#1}}

\DeclareMathOperator{\length}{length}
\newcommand{\triv}{\textup{triv}}

\newcommand{\HDG}{\textup{\textbf{HDG}}}
\newcommand{\conj}{\textup{\textbf{conj}}}

\DeclareSymbolFontAlphabet{\mathbb}{AMSb}
\DeclareSymbolFontAlphabet{\mathbbl}{bbold}

\def\lowsim{\vbox to 0pt{\vss\hbox{$\scriptstyle\sim$}\vskip-2pt}}

\setlist[enumerate,1]{label={\rm(\arabic*)}, ref={\rm\arabic*}}

\newcommand{\supth}[1]{\ensuremath{#1^{\mathrm{th}}}}

\EpigaVolumeYear{8}{2024} \EpigaArticleNr{5} \ReceivedOn{November 28, 2022}
\InFinalFormOn{November 6, 2023}
\AcceptedOn{November 21, 2023}

\title{Derived $\boldsymbol{F}$-zips}
\titlemark{Derived $\boldsymbol{F}$-zips}

\author{Can Yaylali}
\address{Technische Universit\"at Darmstadt, Schlossgartenstra\ss{}e 7, 
64289 Darmstadt, Germany}
\email{yaylali@mathematik.tu-darmstadt.de}

\authormark{C.~Yaylali}

\AbstractInEnglish{We define derived versions of $F$-zips and associate a derived $F$-zip to any proper smooth morphism of schemes in positive characteristic. We analyze the stack of derived $F$-zips and certain substacks. We make a connection to the classical theory and look at problems that arise when trying to generalize the theory to derived $G$-zips and derived $F$-zips associated to lci morphisms.

  As an application, we look at Enriques surfaces and analyze the geometry of the moduli stack of Enriques surfaces via the associated derived $F$-zips. As there are Enriques surfaces in characteristic~$2$ with non-degenerate Hodge--de Rham spectral sequence, this gives a new approach, which could previously not be obtained by the classical theory of $F$-zips.}

\MSCclass{14J10, 14F08, 14A30 (primary); 14F40, 14J28 (secondary)}
\KeyWords{Derived algebraic geometry, moduli of perfect complexes, de Rham cohomology, Enriques surfaces}

\acknowledgement{This project was funded by the Deutsche Forschungsgemeinschaft (DFG, German Research Foundation) TRR 326 \textit{Geometry and Arithmetic of Uniformized Structures}, project number 444845124 and by the LOEWE grant ``Uniformized Structures in Algebra and Geometry''.}

\begin{document}


\maketitle

\begin{prelims}

\DisplayAbstractInEnglish

\bigskip

\DisplayKeyWords

\medskip

\DisplayMSCclass

\end{prelims}


\newpage

\setcounter{tocdepth}{1}

\tableofcontents


\section{Introduction}

The notion of $F$-zips was introduced by Moonen and Wedhorn in \cite{moon-wed}. Before starting with the positive characteristic case discussed in \cite{moon-wed}, let us first look at the characteristic zero case, after \cite{ZipIntro}.

Let $X\rightarrow \Spec(\CC)$ be a proper smooth morphism of schemes. By general GAGA principles, we can associate a compact complex manifold $X^{\textup{an}}$  to $X$ in a ``universal'' way (we do not make this explicit here). Important for us is that the algebraic de Rham cohomology $H^{n}_{\dR}(X/\CC)$ of $X$ is isomorphic to the complex de Rham cohomology $H^{n}_{\dR}(X^{\textup{an}})$ of $X^{\textup{an}}$. Complex de Rham cohomology computes the singular cohomology of $X^{\textup{an}}$ with complex coefficients, \textit{i.e.}
$H^{n}_{\dR}(X^{\textup{an}})\cong H^{n}_{\textup{sing}}(X^{\textup{an}},\CC)\cong H^{n}_{\textup{sing}}(X^{\textup{an}},\ZZ)\otimes_{\ZZ}\CC$. Thus, complex de Rham cohomology comes equipped with an integral structure. The $\CC$-vector space given by $H^{1}_{\dR}$ with its integral structure characterizes, for example, abelian varieties over $\CC$ (this is called the global Torelli property of abelian varieties).

We also have a descending filtration $C^{\bullet}$ on this complex vector space $H^{n}_{\dR}(X^{\textup{an}})$ which is induced by the Hodge spectral sequence. It is known that this spectral sequence degenerates. Therefore, successive quotients are computed by the Hodge cohomologies $C^{i}/C^{i+1}=H^{n-i}(X^{\textup{an}},\Omega^{i}_{X^{\textup{an}}})$. The real structure on the singular cohomology together with the complex conjugation on $\CC$ induce an $\RR$-linear endomorphism on $H^{n}_{\textup{sing}}(X^{\textup{an}},\CC)$. The images of the $C^{i}$ under this map are again complex vector spaces and induce an ascending filtration on $H^{n}_{\dR}(X^{\textup{an}})$ by $D_{i}\coloneqq \overline{C^{n-i}}$. One can show that $D_{i-1}\oplus C^{i}= H^{n}_{\dR}(X^{\textup{an}})$ for all $i\in\ZZ$. These data, together with the integral structure obtained via the comparison with the singular cohomology, endow $H^{n}_{\dR}(X/\CC)$ with an \textit{integral Hodge structure}. The study of integral Hodge structures and their moduli can be found in \cite{Bertin} and leads to the notion of \textit{Griffiths' period domains}. The study of these data enables us to analyze the moduli of geometric objects via linear-algebra data.

These results can be extended to arbitrary smooth proper families in characteristic zero. In characteristic $p>0$, however, we do not have a complex conjugation. But still, we have an analogous structure on the de Rham cohomology.

Let us fix an $\FF_{p}$-algebra $A$ and a smooth proper $A$-scheme $X$. Contrary to the characteristic zero case, we have a second spectral sequence on the de Rham cohomology, the \textit{conjugate spectral sequence}. This spectral sequence endows the de Rham cohomology $H^{n}_{\dR}(X/A)$ with a second filtration $D_{\bullet}$. We also have an analogue of the Poincar\'e lemma, the Cartier isomorphism. The Cartier isomorphism links the graded pieces of the Hodge filtration $C^{\bullet}$ with $D_{\bullet}$. If we assume that the Hodge cohomologies are finite projective and the Hodge--de Rham spectral sequence is degenerate, then the successive quotients are \textit{isomorphic up to Frobenius twist}; \textit{i.e.}\ we have $C^{i}/C^{i+1}\otimes_{A,\Frob} A\cong D_{i}/D_{i-1}$, where $\Frob\colon A\rightarrow A$ denotes the Frobenius endomorphism $a\mapsto a^{p}$. Putting all of these data together, we have a finite projective $A$-module $H^{n}_{\dR}(X/A)$ equipped with two filtrations $C^{\bullet}$ and $D_{\bullet}$ and isomorphisms $\varphi_{i}\colon C^{i}/C^{i+1}\otimes_{A,\Frob} A\cong D_{i}/D_{i-1}$. In good cases, generalizing this by replacing the de Rham cohomology with an arbitrary finite projective module, we get the following definition.

An\textit{ $F$-zip over a scheme $S$} of characteristic $p>0$ is a tuple $(M,C^\bullet,D_\bullet,\varphi_\bullet)$, where $M$ is a finite locally free $\Ocal_S$-module, $C^\bullet$ is a descending filtration on $M$, $D_\bullet$ is an ascending filtration on $M$ and $\varphi_\bullet\colon(\gr^\bullet_C M )^{(1)}\xrightarrow{\lowsim}\gr_D^\bullet M$ are isomorphisms, where the index $(1)$ denotes the Frobenius twist. (For an $\Ocal_{S}$-module $\Fcal$, we set $\Fcal^{(1)}\coloneqq \Fcal\otimes_{\Ocal_{S},\Frob_{S}}\Ocal_{S}$.)
The associated stack of $F$-zips is then a rather combinatorical object. Associating an $F$-zip to a geometric object, such as an abelian scheme, one can analyze its moduli by analyzing the moduli of $F$-zips.

With this method, Moonen--Wedhorn extended results of Ekedahl and Oort on the stratification of families of abelian schemes in positive characteristic. They defined new stratifications on families of proper smooth morphisms satisfying certain conditions, which generalized the known results for abelian varieties in positive characteristic (see \cite{moon-wed}). By attaching extra structure to $F$-zips, we can also generalize the theory of $F$-zips to the theory of so-called \textit{$G$-zips}, defined in \cite{PWZ}. Goldring and Koskivirta used the theory of $G$-zips to construct group-theoretical Hasse invariants on Ekedahl--Oort stratum closures of a general Hodge-type Shimura variety (see \cite{GK}). They applied this in different cases to the Langlands program and proved, for example, a conjecture of Oort.

One major way to associate an $F$-zip to a geometric object is through its de Rham cohomology. Namely, the Hodge and conjugate spectral sequences induce two filtrations on the $\supth{k}$ de Rham cohomology. If the Hodge--de Rham spectral sequence degenerates and the Hodge cohomologies are finite projective (and thus the conjugate spectral sequence also degenerates\footnote{See \cite[Proposition~(2.3.2)]{Katz}.}), the graded pieces are isomorphic up to Frobenius twist, via the Cartier isomorphism. The hypothesis on the Hodge--de Rham spectral sequence restricts us to a certain class of geometric objects, which for example include abelian schemes, K3-surfaces, smooth proper curves and smooth complete intersections in the projective space. But since, for example, the Hodge--de Rham spectral sequence does not degenerate for supersingular Enriques surfaces in characteristic $2$ (see \cite[Theorem 2]{Enriques2}),
we cannot use the theory of $F$-zips to analyze their moduli stack.

One possible solution to this problem is going to the derived world. The idea is straightforward. If we replace the $\supth{k}$ de Rham cohomology with its hypercohomology, we get a perfect complex with two filtrations, and the Cartier isomorphism still applies to the graded pieces.\footnote{See Section~\ref{sec.filtration} for the notions of filtrations in the derived category and of graded pieces.} But taking the derived category naively leads us to problems since we cannot glue in the ordinary derived category. To apply geometric methods, we want descent on the derived category. This problem is solved by introducing the language of $\infty$-categories. So, in particular, the idea of this article is to use homotopy-theoretic methods to analyze derived versions of $F$-zips.

A homotopy-theoretic version of algebraic geometry was developed in \cite{TV2}. In the reference, To\"en--Vezzosi worked in the model-categorical setting. They used simplicial commutative rings as a replacement for commutative rings and presheaves of spaces as a replacement for presheaves of sets (or groupoids). They defined model structures on those (actually in a more general setting, see \cite{TV1}) and used Grothendieck topologies (in their setting) to define derived versions of stacks, schemes and affine schemes as fibrant objects in the corresponding model category. They analyzed certain properties such as geometricity and smoothness, as well as the cotangent complex. In this context, a derived stack is $n$-geometric if it has an $(n-1)$-geometric atlas by a coproduct of derived affine schemes. This notion allows us to define notions like smooth, flat and \'etale by using the atlas and defining it on the level of animated rings. The notion of higher geometricity comes into play if one wants to work with stacks that take values in higher groupoids. In \cite{TVaq}, To\"en and Vaqui\'e gave an important example of a geometric stack, namely the derived stack of perfect complexes. In \cite{AG}, Antieau and Gepner recalled this fact in detail in the setting of spectral algebraic geometry. This shows that one can glue perfect complexes and can cover the stack of perfect complexes 
by affine derived schemes (in a suitable sense), and therefore the notion of derived $F$-zips indicated above should also behave in a similar fashion.

The translation to the world of $\infty$-categories is rather straightforward using Lurie's works \cite{HTT,HA,SAG}. Since To\"en--Vezzosi defined their version of derived algebraic geometry using fibrant objects in model categories, we get analogous notions if we look at the $\infty$-categories associated to the corresponding model categories. Nevertheless, we will recall the definitions and theorems needed in this article from \cite{TV2} and \cite{TVaq} without using much of the model structure and will prove the results purely in the world of $\infty$-categories. This shows that the definitions and results obtained this way  rely not on the chosen model structure but on the underlying $\infty$-category.

We will not do this in detail and will be very brief. A detailed discussion about derived algebraic geometry with proofs and references can be found in the author's notes \cite{NotesDAG}.

\subsubsection*{Derived algebraic geometry}
The first section of this paper focuses on reformulating some results of \cite{TV2,TVaq,SAG} in the language of animated rings. The $\infty$-category of \textit{animated rings} $\AniAlg{\ZZ}$ is given by freely adjoining sifted colimits to polynomial algebras. Looking at overcategories for any animated ring $A$, we can define the $\infty$-category of \textit{animated $A$-algebras} $\AniAlg{A}\coloneqq (\AniAlg{\ZZ})_{A/}$. The benefit of this definition is that many questions about functors from $\AniAlg{\ZZ}$ to $\infty$-categories with sifted colimits can be reduced to polynomial algebras. Animated rings should be thought of as connective spectral commutative rings (\textit{i.e.}\ $E_\infty$-rings) with extra structure. In particular, after forgetting this extra structure, we can also define modules over animated rings (as modules over the underlying $E_{\infty}$-ring). One important example of such a module is the cotangent complex. This module arises naturally if we want to define an analogue of the module of differentials as the module that represents the space of derivations.

The underlying $E_{\infty}$-ring of an animated ring is a commutative algebra object in spectra. Thus, we can define homotopy groups of animated rings and automatically see (with the theory developed in \cite{HA}) that we can associate to every animated ring $A$ an $\NN_{0}$-graded ring $\pi_{*}A$. Using this, we reduce definitions, for example,  smoothness of a morphism $A\rightarrow B\in\AniAlg{\ZZ}$ to smoothness of the ordinary rings $\pi_0A\rightarrow \pi_0B$ together with compatibility of the graded ring structure, \textit{i.e.}\ $\pi_*A\otimes_{\pi_0A} \pi_0B\cong \pi_*B$. Analogously to the classical case, we have that for a smooth morphism of animated rings, its cotangent complex (the module representing the space of derivations) is finite projective. We can upgrade this to an ``if and only if'' if we assume that on $\pi_{0}$ the ring homomorphism is finitely presented. This does not hold in the classical world; \textit{i.e.}\ a ring homomorphism with a finite projective module of differentials may not be smooth; see \textit{e.g.}\ non-smooth regular closed immersions.

Defining \textit{derived stacks} is rather straightforward now. Let $\SS$ denote the $\infty$-category of spaces (also called $\infty$-groupoids). We define derived stacks to be presheaves (of spaces) on $\AniAlg{\ZZ}$ which satisfy \'etale descent. One important class of examples consist of \textit{affine derived schemes}, which we define as representable presheaves on $\AniAlg{\ZZ}$. We can also define relative versions, where we replace $\ZZ$ with an animated ring. We will see that they naturally satisfy fpqc descent. For affine derived schemes, it is easy to define properties by using their underlying animated rings. To do the same for derived stacks, we will need the notion of \textit{$n$-geometric} morphisms. This notion is defined inductively, where we say that a morphism $f\colon F\rightarrow G$ of derived stacks is $(-1)$-geometric if the base change with an affine derived scheme is representable by an affine derived schemes. A $(-1)$-geometric morphism is smooth if it is so after base change to any affine. The morphism $f$ is $n$-geometric if for any affine derived scheme $\Spec(A)$ with morphism $\Spec(A)\rightarrow G$, the base change $F\times_G \Spec(A)$ has a smooth $(n-1)$-geometric effective epimorphism $\coprod \Spec(T_i)\twoheadrightarrow F\times_G \Spec(A) $, where an $n$-geometric morphism is smooth if after affine base change, the induced maps of the atlas to the base are $(-1)$-geometric and smooth. For a good class\footnote{By ``good class'' we mean stable under base change, composition and equivalences, as well as smooth local on the source and target.} of properties \textbf{P} of affine derived schemes, \textit{e.g.} smooth, flat,\dots,\footnote{Note that the property \'etale is not smooth local on the source. We have to be careful if we want to define \'etale morphisms of $n$-geometric stacks.} we can now say that a morphism of derived stacks has property $\pbf\in\Pbf$ if it is $n$-geometric for some $n$ and after base change with an affine derived scheme, the atlas over the affine base has property $\pbf$. As in the affine case, we can relate deformation theory of derived stacks to geometric properties.

We will finish the discussion about derived algebraic geometry with an example of a geometric stack, the derived stack of perfect complexes. This will become important later on, when we want to look at families of perfect complexes together with extra structure, \textit{i.e.}\ derived $F$-zips.

\begin{theorem}[\textit{cf.} Theorem~\ref{main thm}]
	The derived stack
	\begin{align*}
	\PPerf\colon \AniAlg{R} &\longrightarrow \SS\\
	A&\longmapsto \left(\MMod_A^{\textup{perf}}\right)^{\simeq}
	\end{align*}
	is locally geometric and locally of finite presentation.
\end{theorem}

\subsubsection*{Derived $\boldsymbol{F}$-zips}
The second part of this article focuses on derived $F$-zips. To be more specific, we first define derived $F$-zips and then analyze the geometry of their moduli spaces.
The definition of derived $F$-zips is influenced by the natural structure that arises on the de Rham hypercohomology for some proper smooth scheme morphism $X\rightarrow\Spec(A)$. We see, for example, that the Hodge filtration is just a functor $\ZZ^{\op}\rightarrow \Dcal(A)^{\perf}$ that is bounded. This is what we call a \textit{descending filtration}. With this notion, we define a \textit{derived $F$-zip over $A\in \AniAlg{\FF_p}$}, for some positive prime $p$, to be a tuple $(C^\bullet,D_\bullet,\phi,\varphi_\bullet)$, where $C^\bullet$ is a bounded descending filtration, $D_\bullet$ is a bounded ascending filtration of $A$-modules, $\phi\colon \colim_{\ZZ^{\op}}C^{\bullet}\xrightarrow{\lowsim}\colim_{\ZZ}D_{\bullet}$ is an equivalence and  $\varphi_\bullet\colon (\gr^\bullet C)^{(1)}\xrightarrow{\lowsim} \gr^\bullet D$ are equivalences.  As we vary $A$, this construction induces a derived stack (even a hypercomplete fpqc sheaf). One of our main results is that this stack is locally geometric.

\begin{theorem}[\textit{cf.} Theorem~\ref{fzip local geometric plus}]
	The derived stack
		\begin{align*}
			\FZip\colon  \AniAlg{\FF_p}&\longrightarrow \SS,\\
			A&\longmapsto \infty\textup{-groupoid of derived $F$-zips over }A
		\end{align*}
		is locally geometric.
\end{theorem}

The idea of the proof is straightforward. We first look at the derived substack $\FZip^{[a,b],S}$, for a finite subset $S\subseteq \ZZ$ and $a\leq b\in\ZZ$, classifying those derived $F$-zips $(C^{\bullet},D_{\bullet},\phi,\varphi_{\bullet})$ where we fix the Tor-amplitude $[a,b]$ of all filtered pieces and $\gr^{i}C\simeq 0$ for $i\not\in S$. In this way we only have to look at the stacks classifying two chains of morphisms of modules with fixed Tor-amplitude, that have a connecting equivalence at the last entry, such that the graded pieces are equivalent after Frobenius twist. Since perfect modules with fixed Tor-amplitude, morphisms of those and equivalences of those are geometric, we conclude the geometricity of $\FZip^{[a,b],S}$.

We can also look at derived substacks $\FZip^{\leq \tau}$, for a function $\tau\colon \ZZ\times\ZZ\rightarrow \NN_{0}$ with finite support, classifying those derived $F$-zips $\Fline\coloneqq (C^{\bullet},D_{\bullet},\phi,\varphi_{\bullet})$ where the fiberwise dimensions of the $\pi_{i}(\gr^{j}C)$ are at most $\tau(i,j)$. If we have equality and the $\pi_{i}(\gr^{j}C)$ are finite projective, we call $\Fline$ homotopy finite projective of type $\tau$. By the upper semi-continuity of the dimension of fiberwise cohomology of perfect complexes, we see that $\FZip^{\leq\tau}$ is an open substack of $\FZip$ and that it is also geometric, as it is in fact open in some $\FZip^{[a,b],S}$. Writing $\FZip$ as the filtered colimit of the $\FZip^{\leq \tau}$, we deduce the theorem.

Since derived $F$-zips satisfy descent, we can glue this definition to any derived scheme $S$. There is also an \textit{ad hoc} definition in the derived scheme case, but we can show that both definitions agree.

The definition is constructed in such a way  that every proper smooth morphism $f\colon X\rightarrow S$ induces a derived $F$-zip $\underline{Rf_{*}\Omega^{\bullet}_{X/S}}$ over $S$.

\subsubsection*{Modification of filtrations and examples of derived $\boldsymbol{F}$-zips}
For the term \textit{filtration} above, we do not enforce something like a monomorphism condition on the filtration. Even though it seems natural, it actually leads to another definition of derived $F$-zips, which we call \textit{strong derived $F$-zips}. The difference between these two becomes apparent if we look at the corresponding spectral sequences.

\begin{thm}[\textit{cf.} Theorem~\ref{de Rham strong degen}]
	Let $f\colon X\rightarrow S$ be a smooth proper morphism of schemes. Let us consider the Hodge--de Rham spectral sequence
	$$
	E_1^{p,q} = R^qf_*\Omega_{X/S}^p\Longrightarrow R^{p+q}f_*\Omega_{X/S}^\bullet.
	$$
	 The derived $F$-zip $\underline{Rf_*\Omega^\bullet_{X/S}}$ is strong if and only if the Hodge--de Rham spectral sequence degenerates and $R^if_*\Omega_{X/S}^j$ is finite locally free for all $i,j\in\ZZ$.
\end{thm}

So, we do not expect the theory of strong derived $F$-zips to give us any new information if we want to consider geometric objects that do not induce classical $F$-zips. But, we can show that the derived stack of strong derived $F$-zips is open in the derived stack of derived $F$-zips. Further,  looking at very specific types of strong derived $F$-zips, we can even make a connection to classical $F$-zips.

This connection can be generalized by looking at the full sub-$\infty$-category of derived $F$-zips with degenerate spectral sequences\footnote{For any animated ring $A$, by \cite[Proposition 1.2.2.14]{HA}, a functor $F\in\Fun(\ZZ,\MMod_{A})$ with $F(n)\simeq 0$ for $n\ll 0$ induces  a spectral sequence of the form $E_{1}^{p,q}=\pi_{p+q}(\gr^{p}F)\Rightarrow \pi_{p+q}(\colim_{\ZZ} F)$. A derived $F$-zip over $A$ comes equipped with two filtrations and thus induces two such spectral sequences, which we call the spectral sequences attached to the derived $F$-zip.} such that the graded pieces attached to the filtrations have finite projective homotopy groups of type $\tau$, denoted by $\Xcal^\tau$. It is not hard to see that $\Xcal^\tau$ is equivalent to the product of classical $F$-zips of type corresponding to the components of $\tau$ (see Section~\ref{classical theory} for more details).

We note that we formulate all the results more generally for derived $F$-zips over arbitrary derived schemes of positive characteristic.

The above connection to classical $F$-zips also shows that in the case of K3-surfaces or proper smooth curves and abelian schemes, the theory of strong derived $F$-zips gives no new information. In the K3-surfaces and proper smooth curve cases, we can be more specific. Every derived $F$-zip of K3-type or proper smooth curve type is induced by a classical $F$-zip. This is because in both cases, there is only one cohomology group with a non-trivial filtration. This does not hold for abelian schemes (since they have a more complicated type), but as remarked earlier, the derived $F$-zip associated to an abelian scheme $X/A$ is completely determined by the classical $F$-zips associated to $H^{1}_{\dR}(X/A)$ (since the Hodge--de Rham spectral sequence of abelian schemes is degenerate, the Hodge cohomologies are finite projective, and we have $H^{n}_{\dR}(X/A)\cong\wedge^{n}H^{1}_{\dR}(X/A)$). Also, one can look at the moduli stack of Enriques surfaces in characteristic $2$. In this case, the Hodge--de Rham spectral sequence does not degenerate in general. Hence, we cannot directly use the theory of $F$-zips by associating to an Enriques surface its de Rham cohomology but have to use derived $F$-zips for this approach. Using the upper semi-continuity of cohomology, we can see with the theory of derived $F$-zips that the substacks classifying Enriques surfaces of type $\ZZ/2$ or $\mu_2$ are open in the moduli of Enriques surfaces and the substack classifying Enriques surfaces of type $\alpha_2$ is closed. These results give a new proof of the results of Liedtke \cite{Lied}, who does not use the derived theory. In the future, we want to use the theory of derived $F$-zips to analyze discrete invariants of morphisms with non-degenerate Hodge--de Rham spectral sequence. Further, this approach should make it easier to understand the deformation theory of such morphisms as it is naturally part of derived algebraic geometry.

\subsubsection*{Derived $\boldsymbol{F}$-zips with cup product}
Let $f\colon X\rightarrow S$ be a proper smooth morphism of schemes in positive characteristic with geometrically connected fibers of fixed dimension $n$. Further, assume the Hodge--de Rham spectral sequence associated to $f$ degenerates and the Hodge cohomologies are finite locally free. As in the classical case, there is extra structure on the de Rham hypercohomology coming from the cup product, namely a perfect pairing. For classical $F$-zips, this induces a $G$-zip structure on the $F$-zip associated to $H^{n}_{\dR}(X/S)$, for certain reductive groups over a field of characteristic $p>0$. One could try to define a derived $G$-zip, for a reductive group $G$ over a field $k$ of characteristic $p>0$, in such a way such that the cup product induces a derived $\widetilde{G}$-zip structure on the de Rham hypercohomology for some reductive group $\widetilde{G}$ over $k$. As explained in Section~\ref{derived G-zip}, the most obvious ways to generalize the theory of $G$-zips to derived $G$-zips are not quite right. The problem here is that we do not have a ``good'' way of defining derived group schemes. For example, we would like to have that the derived analogue of $\GL_n$-torsors is given by perfect complexes of \textit{Euler characteristic} $\pm n$. But, as far as we know, there is no such analogue.

Alternatively, we show that the symmetric monoidal category of classical $F$-zips over a scheme $S$ in characteristic $p>0$ is equivalent to the symmetric monoidal category of vector bundles over a certain algebraic stack $\Xfr_{S}$. The stack $\Xfr_{S}$ is given by pinching the projective line at $0$ and $\infty$ via the Frobenius morphism and then taking the quotient by the induced $\Gm_{,\FF_{p}}$-action. The category of $G$-zips over $S$ is then equivalent to the stack of $G$-torsors on $\Xfr_{S}$. To apply this construction to derived $F$-zips, we will show that perfect complexes over $\Xfr_{S}$ are precisely derived $F$-zips.
\begin{thm}[\textit{cf.}~Corollary~\ref{cor.fzip.perf}]
	Let $R$ be an $\FF_{p}$-algebra and $S$ an $R$-scheme. Then we have $$\FZip_{R}(S)\simeq \PPerf(\Xfr_{S}).$$
\end{thm}

But, we lack a definition of derived groups and torsors attaching extra structure to perfect complexes. So again, we did not follow this approach further.

For completeness, we naively put the cup product structure into the definition of derived $F$-zips leading to the definition of \textit{$\dR$-zips}. This again is a sheaf, and we can explicitly analyze the projection to derived $F$-zips.

\begin{prop}[\textit{cf.}~Proposition~\ref{dR-zips}]
	The induced morphism via forgetting the pairing
	$$
	p\colon\dRZip\longrightarrow\FZip
	$$
	is smooth and locally of finite presentation; in particular, $\dRZip$ is locally geometric and locally of finite presentation.
\end{prop} 
Depending on the Tor-amplitude of the $\dR$-zips, we can specify the properties of the above forgetful functor.
 
The derived world has another benefit. Usually, we can extend results for smooth objects to objects that are only \textit{lci} (in fact to any animated algebra via left Kan extension). In the case of the de Rham hypercohomology, we know that its lci analogue is given by the derived de Rham complex (here lci is needed to ensure perfectness of the cotangent complex). This seems like a good generalization of the de Rham hypercohomology since it comes equipped with two filtrations with Frobenius-equivalent graded pieces. But, one can show that these filtrations are not bounded in any way (see Section~\ref{de rham lci} for more details). So we would need a notion of derived $F$-zips with unbounded filtrations. But, then the obvious problem becomes the geometricity since we would have to cover an infinite amount of information with the atlas, which is not clear at all~-- geometricity is \textit{a priori} not preserved under arbitrary limits (and may not even be for cofiltered limits).

\subsection*{Structure of this paper} 

We start with a quick summary of derived algebraic geometry (see Section~\ref{sec:DAG}), \textit{i.e.}\ the theory of  \'etale sheaves on animated rings with values in spaces. Mainly, we introduce the notions of derived stacks, geometricity of morphisms and derived schemes. We also look at quasi-coherent modules over derived stacks. We end this section with a quick look at the derived stack of perfect complexes.

Next, we talk about filtrations on the derived category and introduce derived $F$-zips (see Section~\ref{sec:F-Zip}). We show that the presheaf which assigns to an animated ring the $\infty$-category of derived $F$-zips is in fact a sheaf, so a derived stack, and is even locally geometric. After the geometricity, we discuss some important substacks and try to generalize the notion of derived $F$-zips to derived schemes. We look at certain substacks that come naturally by looking at derived $F$-zips of certain type. Also, we look at the substack classifying those filtrations that are termwise monomorphisms. In particular, we show that  under some assumptions, this condition is equivalent to the degeneracy of the Hodge--de Rham spectral sequence. Finally, analogously to classical $F$-zips, we relate derived $F$-zips to perfect complexes over the (Frobenius) pinched projective line.

We finish the study of derived $F$-zips by trying to connect classical $F$-zips with derived $F$-zips (see Section~\ref{classical theory}). We show that in the case of a degenerating Hodge--de Rham spectral sequence, there is no new information coming from derived $F$-zips. Lastly, we apply our theory to the moduli of Enriques surfaces (see Section~\ref{sec:Enriques}).

We finish this paper by elaborating on the problems that appeared while trying to generalize the theory of derived $F$-zips to the case of proper lci morphisms and trying to define derived $G$-zips (see Section~\ref{sec:general}). For completeness, we also naively equip derived $F$-zips with extra structure.

In the appendix, we discuss the connection between classical $G$-zips and $G$-torsors on the (Frobenius) pinched projective line modulo $\Gm_{,\FF_{p}}$-action $\Xfr$.

\subsection*{Assumptions}
All rings are commutative with unit.

We work with the Zermelo--Frenkel axioms of set theory with the axiom of choice, and we assume the existence of inaccessible regular cardinals.

Throughout this paper, we fix some uncountable inaccessible regular cardinal $\kappa$ and the collection $U(\kappa)$ of all sets having cardinality less than $\kappa$, which is a Grothendieck universe (and as a Grothendieck universe is uniquely determined by $\kappa$) and hence satisfies the usual axioms of set theory (see \cite{universe}). When we talk about small, we mean $\Ucal(\kappa)$-small. In the following, we will use some theorems which assume smallness of the respective ($\infty$-)categories. When needed, without further mentioning it, we assume that the corresponding ($\infty$-)categories are contained in $\Ucal(\kappa)$.

If we work with families of objects that are indexed by some object, we will assume, if not further mentioned, that the indexing object is a $\Ucal(\kappa)$-small set.

\subsection*{Notation}
We work in the setting of $(\infty,1)$-categories (see~\cite{HTT}). By abuse of notation, for any $1$-category $C$, we will always denote its nerve again by $C$, unless otherwise specified.

A \textit{subcategory} $\Ccal'$ of an $\infty$-category $\Ccal$ is a simplicial subset $\Ccal'\subseteq\Ccal$ such that the inclusion is an inner fibration. In particular, any subcategory of an $\infty$-category is itself an $\infty$-category, and we will not mention this fact.
\begin{itemize}
\item $\Delta$ denotes the simplex category (see \cite[000A]{kerodon}), 
  \textit{i.e.}\ the category of finite non-empty linearly ordered sets, and $\Delta_{+}$ denotes the category of (possibly empty) finite linearly ordered sets. We denote by $\Delta_{s}$ those finite non-empty
  linearly ordered sets whose morphisms are strictly increasing functions and by $\Delta_{s,+}$ those (possibly empty) finite linearly ordered sets whose morphisms are strictly increasing functions.
	\item By an $\infty$-category, we always mean an $(\infty,1)$-category.
	\item $\SS$ denotes the $\infty$-category of small spaces (also called $\infty$-groupoids or anima).
	\item $\ICat$ denotes the $\infty$-category of small $\infty$-categories.
	\item $\Sp$ denotes the $\infty$-category of spectra.
	\item For an $E_{\infty}$-ring $A$, we denote the $\infty$-category of $A$-modules in spectra, \textit{i.e.}\ $\MMod_{A}(\Sp)$ in the notation of \cite{HA}, by $\MMod_{A}$.
	\item For any ordered set $(S,\leq)$, we denote its corresponding $\infty$-category again by $S$, where the corresponding $\infty$-category of an ordered set is given by the nerve of $(S,\leq)$ seen as a $1$-category (the objects are given by the elements of $S$ and $\Hom_S(a,b)$ equals $\ast$ if and only if $a\leq b$ and is otherwise empty). 
	\item For any set $S$, the $\infty$-category $S^{\disc}$ is the nerve of the set $S$ seen as a discrete $1$-category (the objects are given by the elements of $S$, and $\Hom_S(a,a)$ equals $\ast$ for any $a\in S$ and is otherwise empty).
	\item For any morphism $f\colon X\rightarrow Y$  in an $\infty$-category $\Ccal$ with finite limits, if it exists, we denote the functor from $\Delta_{+}$ to $\Ccal$ that is given by the \v{C}ech nerve of $f$ (see \cite[Section 6.1.2]{HTT}) by $\Cv(Y/X)_{\bullet}$.
	\item Let $\Ccal$ be an $\infty$-category with final object $\ast$. For morphisms $f\colon \ast\rightarrow X$ and $g\colon \ast\rightarrow X$,  if it exists, we denote the homotopy pullback $\ast\times_{f,X,g}\ast$ by $\Omega_{f,g}X$. If $\Ccal$ has an initial object $0$, then we denote the pullback $0\times_{X}0$ by $\Omega X$.
	\item Let $f\colon X\rightarrow Y$ be a morphism in $\SS$, and let $y\in Y$. We write $\fib_{y}(X\rightarrow Y)$ or $\fib_{y}(f)$ for the pullback $X\times_{Y}\ast$, where $\ast$ is the final object in $\SS$ (up to homotopy) and the morphism $\ast\rightarrow Y$ is induced by the element $y$, which by abuse of notation we also denote by $y$.
	\item For a morphism $f\colon M\rightarrow N$ in $\MMod_{A}$, where $A$ is some $E_{\infty}$-ring, we define $\fib(f)=\fib(M\rightarrow N)$ (resp.\ $\cofib(f)=\cofib(M\rightarrow N)$) as the pullback (resp.\ pushout) of $f$ with the essentially unique zero morphism $0\rightarrow N$ (resp.\ $M\rightarrow 0$). 
	\item  When we say that a square diagram in an $\infty$-category $\Ccal$ of the form 
	$$
	\begin{tikzcd}
		W\arrow[r,""]\arrow[d,""]& X\arrow[d,""]\\
		Y\arrow[r,""] &Z
	\end{tikzcd}
	$$
	 is commutative, we always mean that we can find a morphism $\Delta^{1}\times\Delta^{1}\rightarrow \Ccal$ of $\infty$-categories that extends the above diagram.
\end{itemize}

\subsection*{Acknowledgements}
This article is part of my PhD thesis \cite{thesis} under the supervision of Torsten Wedhorn (another extract, with more detailed notes in derived algebraic geometry, is \cite{NotesDAG}). I want to thank him for his constant help and for all his patience while listening to my questions during our countless discussions. He suggested this topic and encouraged me to learn a lot about higher algebra and derived algebraic geometry. I would also like to thank Anton G\"uthge, Manuel Hoff and Zhouhang Mao for their feedback concerning mistakes in the earlier versions, Benjamin Antieu, Adeel Khan and Jonathan Weinberger for their help concerning derived algebraic geometry and higher topos theory,
Christian Liedtke for discussions concerning Enriques surfaces and Timo Richarz for helpful discussions. Lastly, I want to thank R{\i}zacan \c{C}ilo\u{g}lu, Catrin Mair, Simone Steilberg and Thibaud van den Hove.

\section{Overview of derived algebraic geometry}
\label{sec:DAG}
The first thing that one might ask is, \textit{why derived algebraic geometry?} As explained at the beginning of Section~\ref{sec:F-Zip}, we want to define the analogue of $F$-zips, where we work with perfect complexes instead of vector bundles. One downside to the derived category is that we cannot glue morphisms in it.  This problem is resolved if we keep track of all the higher homotopies, \textit{i.e.}\ pass to the derived $\infty$-category. So, it is very natural to work with $\infty$-categories and sheaves in $\infty$-groupoids instead of groupoids.

But we can go a bit further. Instead of functors from $\textup{(Ring)}$ to $\SS$, we can work with functors from animated rings (the $\infty$-category associated to simplicial commutative rings) to $\SS$. One of the benefits of working with animated rings is that deformation theory comes very naturally. To be more specific, the cotangent complex is the complex representing derivations. This allows us to link properties of morphisms like \textit{smooth}, \textit{\'etale} and \textit{locally of finite presentation} to properties of the cotangent complex. Hence, we can use geometric properties of sheaves on animated rings with values in $\infty$-groupoids to analyze its deformation theory.

Still the question remains, \textit{what does geometry mean in this context?} This section is dedicated to this question, and we want to summarize important aspects of derived algebraic geometry following the works of To\"en--Vezzosi \cite{TV2}, Antieu--Gepner \cite{AG}, To\"en--Vaqui\'e \cite{TV2} and Lurie \cite{DAG}. We will only briefly showcase the theory of derived algebraic geometry needed for this paper. A detailed discussion, with proofs and references, can be found in the author's notes \cite{NotesDAG}. 

\subsection{Derived commutative algebra}
\label{sec:derived commutative algebra}
In the following, $R$ will be a ring.

In this subsection, we want to give a quick summary about animated rings and look at the deformation theory of animated rings.  A detailed overview with proofs and references can be found in \cite[Section~3.1]{NotesDAG}.

\subsubsection{Animated rings}
\label{sec:simplicial commutative algebras}

By $\Poly_R$ we denote the category of polynomial $R$-algebras in finitely many variables. Then the category of $R$-algebras is naturally equivalent to the category of functors from $\Poly_R^{\op}$ to $\Sets$ which preserve finite products. Applying this construction to the $\infty$-categorical case, we obtain $\AniAlg{R}$, the $\infty$-category of animated $R$-algebras.\footnote{The name ``animated $R$-algebras'' is due to Cesnavicius--Scholze \cite{CS}, who give a construction for any cocomplete $\infty$-category, called \textit{animation}. The animation of the (nerve of the) category of $R$-algebras leads to an equivalent $\infty$-category.}

\begin{defi}
	We define the $\infty$-category of \textit{animated $R$-algebras} as
	$$
		\AniAlg{R} \coloneqq \Fun_{\pi}(\Poly_R^{\op},\SS),
	$$
	        where the subscript $\pi$ denotes the full subcategory of $\Fun(\Poly_R^{\op},\SS)$ of functors that preserve finite products.
                
	If $R=\ZZ$, we say \textit{animated ring} instead of animated $R$-algebra. If $A$ is an animated ring, we also denote the \textit{$\infty$-category of animated $A$-algebras} by $\AniAlg{A}\coloneqq (\AniAlg{\ZZ})_{A/}$.
\end{defi}

As shown in \cite{SAG}, the definition of animated rings yields a functor $\theta\colon\AniAlg{R}\rightarrow\Einftycn_{R}$. Here $\Einftycn_{R}$ denotes the $\infty$-category of connective $E_{\infty}$-R-algebras, \textit{i.e.}\ connective commutative $R$-algebra objects in spectra (see \cite{HA} for more details). This functor can be seen as a ``forgetful functor'' in the sense that it forgets the strictness of the associativity in $\AniAlg{R}$. In particular, any animated $R$-algebra has an underlying connective spectrum. In this way, we are able to define homotopy groups of animated rings.

\begin{defi}
	Let $A$ be an animated ring, and let $i\in \ZZ$. Then we define the \textit{$\supth{i}$ homotopy group} as $\pi_{i}A\coloneqq \pi_{i}\theta(A)$.
\end{defi}

\begin{rem}
	We want to note that all of the above can be upgraded to the case where $R$ is an animated ring, via passage to undercategories.
\end{rem}

We conclude this subsection by starting with a bit of geometry, namely the localization of animated rings.

\begin{propdef}
\label{localization}
	For any element $f\in\pi_{0}A$,\footnote{Recall that by the construction of the fundamental group, as explained above, $\pi_{0}A$ is a commutative ring.} there is an animated ring $A[f^{-1}]$ with the property that for all $B\in\AniAlg{A}$, the simplicial set $\Hom_{\AniAlg{A}}(A[f^{-1}],B)$ is non-empty if and only if the image of $f$ under $\pi_0(A)\rightarrow\pi_0(B)$ is invertible.\footnote{In fact, we can even localize at any subset $F\in \pi_{0}A$ (see \cite{NotesDAG}).} This localization is also compatible with taking homotopy groups; \textit{i.e.}\ $\pi_{i}(A[f^{-1}])\simeq (\pi_{i}A)_{f}$. 
\end{propdef}

\begin{proof}
	The proof follows that of \cite[Proposition 1.2.9.1]{TV2}. The idea is to localize $\AniAlg{A}$ at the multiplication map by $f$ (and apply $\Sym_{A}$) and define $A[f^{-1}]$ as the image under the localization functor. For further details and the general case, where we localize at a subset of $\pi_{0}A$, see \cite[Proposition 3.13]{NotesDAG}.
\end{proof}

The localization map of animated rings should be an open immersion. But, we still need to define properties of morphisms between animated rings.

\begin{defi}
	Let $f\colon A\rightarrow B$ be a morphism between animated rings. Then $f$ is called
	\begin{enumerate}
		\item \textit{locally of finite presentation} if $B$ is a compact animated $A$-algebra through $f$,
		\item	\textit{flat} if $\pi_{0}f$ is flat and the natural map $\pi_{i}A\otimes_{\pi_{0}A}\pi_{0}B\rightarrow \pi_{i}B$ is an isomorphism for all $i\in\ZZ$, and
		\item \textit{smooth} (resp.\ \textit{\'etale}) if $f$ is flat and $\pi_{0}f$ is smooth (resp.\ \'etale).
	\end{enumerate}
\end{defi}

Now using the definition of the localization, it is not hard to see that the natural map $A\rightarrow A[f^{-1}]$ is a monomorphism and is locally of finite presentation. One can even show that it is flat. So, in particular, $A\rightarrow A[f^{-1}]$ is an \textit{open immersion},\footnote{An \textit{open immersion} of animated rings is per definition a flat, finitely presented monomorphism.} as expected.

\subsubsection{Modules over animated rings}
Let us quickly recall the notion of an $A$-module for an animated ring $A$ and list some facts that we will need later on.

The forgetful functor $\theta$ from animated rings to $E_{\infty}$-rings allows us to see any animated ring as a ring object in spectra. In particular, we can look at modules in spectra over the underlying animated ring. 

\begin{defi}
	Let $A$ be an animated ring. Then we define the \textit{$\infty$-category of $A$-modules} as $\MMod_{A}\coloneqq \MMod_{\theta(A)}(\Sp)$.
\end{defi}

Note that we could also look at the animation of modules, but as explained in \cite{CS}, this only gives the $\infty$-category of \textit{connective} modules. In particular, for a discrete animated ring $A$, we have $\MMod_{A}\simeq \Dcal(A)$, whereas the animation of $A$-modules only recovers $\Dcal(A)^{\leq 0}$. 

Further, the functor $\theta$ has a left adjoint, which induces a left adjoint to the forgetful functor $\AniAlg{R}\rightarrow \MMod^{\cn}_{R}$.

\begin{defi}
	The left adjoint to the forgetful functor $\AniAlg{R}\rightarrow \MMod^{\cn}_{R}$ is denoted by $\Sym_{R}$.
\end{defi}

The equivalence of the derived $\infty$-category with spectral modules gives us an idea how to define perfect modules and the Tor-amplitude. 

\begin{defi}
  Let $A$ be an animated ring. An $A$-module $M$ is called \textit{perfect} if $M$ is compact in $\MMod_{A}$.
  
	The Tor-amplitude of a perfect $A$-module $M$ is defined as the Tor-amplitude of $M\otimes_{A}\pi_{0}A\in\Dcal(\pi_{0}A)$.
\end{defi}

As modules over animated rings are defined as modules over the underlying ring spectrum, we can use the results of \cite{HA} to define and understand notions like \textit{projective} and \textit{flat} modules. We also have the Tor-spectral sequence relating the Tor-groups to the homotopy groups of the tensor product of modules. We do not want to go into detail and refer to \cite{HA} or \cite[Section 2.2]{NotesDAG}.

But, we want to end this section with a quick proposition found in \cite{AG} relating the notion of perfectness and Tor-amplitude, as in the classical case.

\begin{lem}
	\label{general props of Tor}
	Let $A$ be an animated $R$-algebra. Let $P$ and $Q$ be $A$-modules.
	\begin{enumerate}
		\item If\, $P$ is perfect, then $P$ has finite Tor-amplitude.
		\item If\, $B$ is an $A$-algebra and $P$ has Tor-amplitude in $[a,b]$, then the $B$-module $P\otimes_A B$ has Tor-amplitude in $[a,b]$.
		\item If\, $P$ has Tor-amplitude in $[a,b]$ and $Q$ has Tor-amplitude in $[c,d]$, then $P\otimes_A Q$ has Tor-amplitude in $[a+c,b+d]$.
		\item If\, $P,Q$ have Tor-amplitude in $[a,b]$, then for any morphism $f\colon P\rightarrow Q$, the fiber of $f$ has Tor-amplitude in $[a-1,b]$ and the cofiber of $f$ has Tor-amplitude in $[a,b+1]$.
		\item If\, $P$ is a perfect $A$-module with Tor-amplitude in $[0,b]$, with $0\leq b$, then $P$ is connective and $\pi_0P\simeq \pi_0(P\otimes_A \pi_0A)$.
		\item The $A$-module $P$ is perfect and has Tor-amplitude in $[a,a]$ if and only if\, $P$ is equivalent to $M[a]$ for some finite projective $A$-module.
		\item If\, $P$ is perfect and has Tor-amplitude in $[a,b]$, then there exists a morphism
		$$
		M[a]\longrightarrow P
		$$
		such that $M$ is a finite projective $A$-module and the cofiber is perfect with Tor-amplitude in $[a+1,b]$.
	\end{enumerate}
\end{lem}

\begin{proof}
	Since modules over animated rings are defined as modules over their underlying $E_\infty$-ring spectrum, this is \cite[Proposition 2.13]{AG}.
\end{proof}

\subsubsection{Deformation theory of animated ring}

As mentioned at the beginning of this section, one benefit of working with animated rings is the naturally arising deformation theory. Let us be more precise.

For any animated $R$-algebra $A$ and any connective $A$-module $M$, we can define the \textit{square zero extension of $A$ by M}, denoted by $A\oplus M$.\footnote{This is induced by the left Kan extension of the classical square zero extension; \textit{i.e.}\ for a discrete ring $A$ and an $A$-module $M$, we have that $A\oplus M$ is the direct sum with multiplication given by $(a,m)(a',m')\coloneqq (aa',am'+a'm)$.} This allows us to define derivations.

\begin{defi}
	An  \textit{$R$-linear derivation of $A$ into $M$} is a morphism $A\rightarrow A\oplus M$ over $A$. The space of derivations is denoted by $\Der_{R}(A,M)\coloneqq \Hom_{(\AniAlg{R})_{/A}}(A,A\oplus M)$.
\end{defi}

Classically, meaning if we do not work with animated rings, the space of derivations is represented by the K\"ahler differentials $\Omega^{1}_{A/R}$. In the case of animated rings, $\Der_{R}(A,M)$ is represented by a connective $A$-module $L_{A/R}$ (which is unique up to homotopy), the \textit{cotangent complex of $A$ over $R$}.

Since this construction is a natural generalization of K\"ahler differentials, it gives a more natural approach to deformation theory. In particular, if $f\colon A\rightarrow B$ is a morphism of animated rings such that $\pi_{0}f$ is finitely presented, the properties \textit{finitely presented}, \textit{smooth} and \textit{\'etale} of $f$ correspond to the properties \textit{perfect}, \textit{finite projective} and \textit{trivial} of $L_{B/A}$, respectively.

\begin{prop}
\label{prop-cotangent char}
	Let $f\colon A\rightarrow B$ be a morphism of animated rings. Assume $\pi_{0}f$ is finitely presented. Then
	\begin{enumerate}
		\item $f$ is locally of finite presentation if and only if\, $L_{B/A}$ is perfect; 
		\item $f$ is smooth if and only if\, $L_{B/A}$ is finite projective; 
		\item $f$ is \'etale if and only if\, $L_{B/A}\simeq 0$.
	\end{enumerate}
\end{prop}

\begin{proof}
	 The proof is the same as in the model-categorical case, presented in \cite{TV2}. A detailed proof of the second statement can be found in \cite[proof of Proposition 2.56]{NotesDAG}.	
\end{proof}

Further, we want to mention that one can show that the natural truncation morphisms $A_{\leq n}\rightarrow A_{\leq n-1}$ are square zero extensions (in a suitable sense). In particular, \'etale and thus also Zariski coverings of animated rings only depend on the underlying covering on the commutative ring $\pi_{0}A$, and any such covering on $\pi_{0}A$ gives rise to a covering on $A$. 

\begin{prop}
	\label{lift etale}
	Let $A$ be an animated $R$-algebra. Then the base change under the natural morphism $A\rightarrow \pi_0A$ induces an equivalence of $\infty$-categories between \'etale $A$-algebras and \'etale $\pi_0A$-algebras.
\end{prop}

\begin{proof}
See \cite[Proposition 5.2.3]{CS}.
\end{proof}

\subsection{Derived algebraic geometry}
\label{sec:der.alg.geo}
In this section, we want to give a quick overview of derived algebraic geometry. It is important to note that we are following the definitions of \cite[Sections~1 and~2.2]{TV2} as there are different approaches to derived algebraic geometry. For example, there are different results depending on the points of the stacks. The theory differs if we replace animated rings with DG-rings or (connective) $E_{\infty}$-rings. We are interested in animated rings since we want to work in positive characteristic and that setting seems to be the one most suited for this application. The main reason is the Frobenius, which naturally exists for animated rings in positive characteristic. For $E_{\infty}$-rings, there is no clear notion of a Frobenius. This is because we would need to find a Frobenius map that is homotopy coherent on all levels, and since animated rings have ``more'' structure, the ``on the nose'' definition on simplicial commutative rings gives us a Frobenius morphism on animated rings. The difference between animated rings and DG-rings only becomes visible in positive characteristic. Via the Dold--Kan
correspondence, one gets a functor from animated rings to DG-rings, but one can show that the image of this functor takes values in DG-rings that naturally have a PD-structure.

It is important to note that To\"en--Vezzosi develop the theory of derived algebraic geometry for animated rings in terms of model categories. But, using the forgetful map between animated rings and connective $E_{\infty}$-rings, we can use the results of Lurie in \cite{SAG} to understand the theory purely in $\infty$-categorical terms, without much effort (as the difficult part was done by Lurie).

For this section, we will follow \cite[Section~2]{TV2}, \cite{AG} and the lecture notes of Adeel Khan \cite{Khan}. We will not prove any of the assertions in this section and refer to the notes of the author \cite[Section 4.2]{NotesDAG} for details.

\subsubsection{Affine derived schemes}

\label{sec:affine derived schemes}
In the following, $R$ will be a ring and $A$ an animated $R$-algebra.

Let us define the \'etale and fpqc topology.

\begin{propdef}
	Let $B$ an animated $A$-algebra. 
	\begin{enumerate}
	\item There exists a Grothendieck topology on $\AniAlg{A}^{\op}$,
          called the \textup{fpqc topology}, which can be described as follows:  A sieve $($see \cite[Definition 6.2.2.1]{HTT}$)$ $\Ccal\subseteq (\AniAlg{A}^{\op})_{/B}\simeq\AniAlg{B}^{\op}$ is a covering sieve if and only if it contains a finite family $(B\rightarrow B_{i})_{i\in I}$ for which the induced map $B\rightarrow \prod_{i\in I}B_{i}$ is faithfully flat.
		\item There exists a Grothendieck topology on the full subcategory $(\AniAlg{A}^{\et})^{\op}$ of \'etale $A$-algebras, called the \textup{\'etale topology}, which can be described as follows:  A sieve $\Ccal\subseteq (\AniAlg{A}^{\et})^{\op}_{/B}\simeq(\AniAlg{B}^{\et})^{\op}$ is a covering sieve if and only if it contains a finite family $(B\rightarrow B_{i})_{i\in I}$ for which the induced map $A\rightarrow \prod_{i\in I}B_{i}$ is faithfully flat $($and \'etale, which is automatic$)$.
	\end{enumerate}
\end{propdef}

\begin{proof}
	This follows using the results from \cite{SAG} and the forgetful functor between animated rings and $E_{\infty}$-rings. For further details, see \cite[Proposition 4.1]{NotesDAG}.
\end{proof}

We use the functorial view of schemes to define derived affine schemes. As expected, derived affine schemes are sheaves for the \'etale topology.

\begin{defrem}
  An \textit{affine derived scheme over $A$} is a functor from $\AniAlg{A}$ to spaces (\textit{i.e.}\ a presheaf on $\AniAlg{A}^{\op}$) which is equivalent to $\Spec(B)\coloneqq \Hom_{\AniAlg{A}}(B,-)$ for some $B\in\AniAlg{A}$.
  
  The presheaf $\Spec(B)$ is in fact a sheaf for the fpqc topology, as we can see using that the fpqc topology on $E_{\infty}$-algebras is subcanonical (see \cite[Theorem~D.6.3.5]{SAG}).
\end{defrem}

We can immediately define properties of affine derived schemes using their underlying animated ring morphisms.

\begin{defi}
	Let $\Pbf$ be one of the following properties of a morphism of animated rings: \textit{flat, smooth, \'etale, locally of finite presentation}. We say that a morphisms of affine derived schemes $\Spec(B)\rightarrow \Spec(C)$ has property $\Pbf$ if the underlying homomorphism $C\rightarrow B$ has $\Pbf$.
\end{defi}

\begin{rem}
	Let us note that the above properties of morphisms of affine derived schemes are stable under equivalences, compositions and pullbacks and are \'etale local on the source and target. 
\end{rem}

\begin{defi}
  For a discrete ring $A$, we define
  $$
  \Spec(A)_{\cl}\coloneqq \Hom_{\Ring}(A,-)\colon \Ring\longrightarrow \Sets
  $$
  to be its underlying classical scheme. We will abuse notation and denote the underlying locally ringed space of $\Spec(A)_{\cl}$ the same.
\end{defi}

\begin{remark}
The notation $(-)_{\cl}$ is introduced since even for a discrete ring $A$, the corresponding derived stack $\Spec(A)$ is a sheaf with values in spaces. Thus, for a (possibly non-discrete) animated ring $B$, the space $\Hom_{\Ani}(A,B)$ need not to be discrete; \textit{e.g.} we have that $\Hom_{\Ani}(\ZZ[X],B)\simeq \Omega^\infty B$. But, for example if we restrict ourselves to discrete rings $C$, we have $\Hom_{\Ani}(A,C)\simeq \Hom_{\Ring}(\pi_0A,C)$ by adjunction, even when $A$ is not discrete.
\end{remark}

\subsubsection{Geometric stacks}\label{sec:geometric-stacks}

In this section, we mostly follow \cite{AG} and \cite{TV2}. 

\begin{defi}
	Let $A$ be an animated ring. A \textit{derived stack over $A$} is a sheaf of spaces on $(\AniAlg{A}^{\et})^{\op}$.
	We denote the $\infty$-category of derived stacks over $A$ by $\dSt_A$. 
	If $A=\ZZ$, we simply say derived stack and denote the $\infty$-category by $\dSt$.
\end{defi}

Next we define what a ``geometric'' derived stack is. This terminology is due to Simpson, who himself tried to understand sheaves of rings taking values in $n$-groupoids. For this, he defined inductively the term ``n-(algebraic) geometric'', which describes \'etale sheaves that have an atlas that is less truncated than before and  again has an atlas by sheaves that are less truncated; \textit{i.e.}\ an $n$-geometric stack is a sheaf with a smooth atlas given by $(n-1)$-geometric sheaves. So in the end, one can understand $n$-geometric sheaves by understanding inductively sheaves that are less truncated (or more affine).

\begin{defi}[\textit{cf.} \protect{\cite[Definition 1.3.3.1]{TV2}}]
	We will define a geometric morphism inductively.
	\begin{enumerate}
	\item A derived stack is \textit{$(-1)$-geometric} or \textit{affine} if it is equivalent to an affine derived scheme.
          
	  A morphism of derived stacks $X\rightarrow Y$ is \textit{$(-1)$-geometric} or \textit{affine} if for all affine schemes $\Spec(A)$ and all $\Spec(A)\rightarrow Y$, the base change $X\times_{Y}\Spec(A)$ is affine.
          
	  A $(-1)$-geometric morphism of derived stacks $X\rightarrow Y$ is \textit{smooth} if for all affine derived schemes $\Spec(A)$ and all morphisms $\Spec(A)\rightarrow Y$, the base change morphism $\Spec(B)\simeq X\times_Y\Spec(A)\rightarrow \Spec(A)$ corresponds to a smooth morphism of animated rings.
        \end{enumerate}
        
		Now let $n\geq 0$.
        \begin{enumerate}[resume]
	      \item An \textit{$n$-atlas} of a derived stack $X$ is a family $(\Spec(A_i)\rightarrow X)_{i\in I}$ of morphisms of derived stacks such that
		\begin{enumerate}
			\item each $\Spec(A_i)\rightarrow X$ is $(n-1)$-geometric and smooth, and
			\item the induced morphism $\coprod \Spec(A_i)\rightarrow X$ is an effective epimorphism.
		\end{enumerate}	
		A derived stack is called \textit{$n$-geometric}, if 
		\begin{enumerate}
			\item it has an $n$-atlas, and
			\item the diagonal $X\xrightarrow{\Delta}X\times X$ is $(n-1)$-geometric.
		\end{enumerate}
	      \item A morphism $X\rightarrow Y$ of derived stacks is called \textit{$n$-geometric} if for all affine derived schemes $\Spec(A)$ and all morphisms $\Spec(A)\rightarrow Y$, the base change $X\times_Y \Spec(A)$ is $n$-geometric.
                
		An $n$-geometric morphism $X\rightarrow Y$ of derived stacks is called \textit{smooth} if for all affine derived schemes $\Spec(A)$ and all morphisms $\Spec(A)\rightarrow Y$, the base change $X\times_Y \Spec(A)$ has an $n$-atlas given by a family of affine derived schemes $(\Spec(A_i))_{i\in I}$ such that the induced morphisms $A\rightarrow A_i$ are smooth.
	\end{enumerate}
	We call a morphism of derived stacks \textit{geometric} if it is $n$-geometric for some $n\geq -1$.
\end{defi}

\begin{defi}
	\label{def P-geometric}
	Let $\Pbf$ be a property of affine derived schemes that is stable under equivalences, pullbacks and compositions and is smooth local on the source and target. Then we say that a morphism of derived stacks $X\rightarrow Y$ has $\Pbf$ if it is geometric and for an affine $(n-1)$-atlas $(U_i)_{i\in I}$ of the pullback along an affine derived scheme $\Spec(B)$, the corresponding morphisms $U_i\rightarrow \Spec(B)$ of affine schemes have $\Pbf$.
\end{defi}

\begin{lem}
	The properties ``locally of finite presentation'', ``flat'' and ``smooth'' of morphisms of affine derived schemes satisfy the conditions of Definition~\ref{def P-geometric}.
\end{lem}
\begin{proof}
	This follows from the characterizations in Proposition~\ref{prop-cotangent char} and the definitions. See \cite[Lem\-ma~4.14]{NotesDAG} for further details.
\end{proof}

\begin{rem}
  We want to note that the property \textit{\'etale} is not smooth local on the base since if it were smooth local in our context, then it would be smooth local in the classical theory of schemes, which it is not. This, in particular, means that one has to be careful when defining \'etale morphisms of derived stacks. A discussion and definition can be found in \cite[Definition 4.11]{NotesDAG} since this was not treated carefully in \cite{TV2}, which  led to some errors in the definition given there. 
  
	The basic idea is to define \'etale morphisms similarly to how it is done for algebraic stacks, where we only consider morphisms that are ``DM'' (again, we refer to \cite{NotesDAG}).
\end{rem}

We next define open and closed immersions of derived stacks. These seem natural, except that for closed immersions, we do not impose any monomorphism condition. This is due to the fact that monomorphisms have a vanishing cotangent complex. Since many interesting closed immersions (such as regular immersions) have non-vanishing cotangent complex, we omit this property in the definition.

\begin{defi}
	A morphism of derived stacks $X\rightarrow Y$ is
	\begin{enumerate}
		\item an \textit{open immersion} if it is  flat and locally of finite presentation and is a monomorphism, where flat is in the sense of Definition~\ref{def P-geometric} and monomorphism means $(-1)$-truncated in the sense of \cite{HTT}; \textit{i.e.}\ the homotopy fibers of $X\rightarrow Y$ are either empty or contractible;  
		\item	a \textit{closed immersion} if it is affine and for any $\Spec(B)\rightarrow Y$, the corresponding morphism $X\times_Y\Spec(B)\simeq \Spec(C)\rightarrow \Spec(B)$ induces a surjection $\pi_0B\rightarrow\pi_0C$ of rings.
	\end{enumerate}
\end{defi}

Let us give an important example of an open immersion of derived stacks.

\begin{lem}
	Let $A$ be an animated $R$-algebra, and let $f\in \pi_0A$ be an element. Then the inclusion $j\colon \Spec(A[f^{-1}])\hookrightarrow \Spec(A)$ is an open immersion.
\end{lem}

\begin{defi}
  A derived stack $X$ is \textit{locally geometric} if we can write $X$ as the filtered colimit of geometric derived stacks $X_i$ with open immersions $X_i\hookrightarrow X$.
  
	We say that a locally geometric stack $X\simeq \colim_{i\in I} X_{i}$ is \textit{locally of finite presentation} if each $X_{i}$ is locally of finite presentation. 
\end{defi}

\begin{defprop}
	For a morphism of derived stacks $f\colon X\rightarrow Y$, we define $\Im(f)$ as an epi-mono factorisation $X\twoheadrightarrow \Im(f)\hookrightarrow Y$ of $f$ $($here ``epi'' means ``effective epimorphism''\,$)$. This factorisation is unique up to homotopy.
\end{defprop}
\begin{proof}
  The existence of such a factorisation follows from \cite[Example~5.2.8.16]{HTT}.
  The uniqueness up to homotopy follows from \cite[Proposition 5.2.8.17]{HTT}.
\end{proof}

\begin{rem}
	\label{lift along affine} 
	Let $A$ be an animated ring.
	Using the above, there is a way to lift open immersions $\widetilde{U}\hookrightarrow \Spec(\pi_{0}A)$ to open immersions $U\rightarrow \Spec(A)$. The idea is to cover $\widetilde{U}$ by standard affines and lift these using Proposition~\ref{localization} (for more details, see \cite[Remark 4.42]{NotesDAG}). 
\end{rem}

We can also define derived schemes using open immersions.

\begin{defi}
	Let $X$ be a derived stack. Then $X$ is a \textit{derived scheme} if it admits a cover $(\Spec(A_i)\hookrightarrow X)_{i\in I}$ such that each $\Spec(A_i)\hookrightarrow X$ is an open immersion (in particular, $X$ is $1$-geometric\footnote{It is not hard to see that any geometric monomorphism is $0$-geometric (a more detailed discussion can be found in \cite{NotesDAG}).}).
\end{defi}

\begin{remdef}[Truncation]
	\label{truncation}
	There is an adjunction between $\Shv_{\et}(\Alg{R})$, \'etale sheaves on $\Alg{R}$ with values in $\SS$, and $\Shv_{\et}(\AniAlg{R})$, \'etale sheaves on $\AniAlg{R}$ with values in $\SS$,
	$$
		\begin{tikzcd}[column sep=1.44em] 
			 \iota\colon \Shv_{\et}(\Alg{R})\arrow[r,"", shift left =0.7]&\arrow[l,"", shift left =0.7]\Shv_{\et}(\AniAlg{R})\colon t_{0}\rlap{.}
		\end{tikzcd}
	        $$
	Let us state some facts about $\iota$ and $t_{0}$.
	\begin{enumerate}
		\item	The functor $t_0$ has a right and a left adjoint. 
		\item	The functor $t_0$ preserves geometricity (here geometricity of sheaves in $\Alg{R}$ is defined as for derived stacks; see \cite[Section 2.1.1]{TV2} for further information) and the properties flat, smooth and \'etale along geometric morphisms. 
		\item	The functor $\iota$ preserves geometricity and homotopy pullbacks of $n$-geometric stacks along flat morphisms and sends flat (resp.\ smooth, \'etale) morphisms of $n$-geometric stacks to flat (resp.\ smooth, \'etale) morphisms of $n$-geometric stacks.
	\end{enumerate}
	A proof for these statements is given in \cite[Proposition 2.2.4.4]{TV2}. A more detailed discussion can be found in \cite{NotesDAG}.
\end{remdef}

We list some properties of geometric morphisms of derived stacks. All of these statements can also be found in the $E_{\infty}$-case in \cite{AG}, and the proofs are similar. Proofs in this case can be found in \cite{NotesDAG}.

\begin{lem}
	\label{geometric base change}
	Let $X\rightarrow Z$ and $Y\rightarrow Z$ be morphisms of derived stacks. If\, $X\rightarrow Z$ is $n$-geometric, then so is $X\times_Z Y\rightarrow Y$.
\end{lem}

\begin{lem}
	\label{geometric local}
	A morphism of derived stacks $X\rightarrow Y$ is $n$-geometric if and only if the base change under $\Spec(A)\rightarrow Y$ for any $A\in \AniAlg{R}$ is $n$-geometric.
\end{lem}

\begin{lem}
	\label{geometric comp}
	Let $f\colon X\rightarrow Y$ and $g\colon Y\rightarrow Z$ be morphisms of derived stacks. If\, $f$ and $g$ are $n$-geometric, then so is $g\circ f$.
\end{lem}

\begin{prop}
	\label{diag geometric}
	Let $f\colon X\rightarrow Y$ be a morphism of derived stacks. Assume $X$ is $n$-geometric and the diagonal $Y\rightarrow Y\times Y$ is $n$-geometric. Then $f$ is $n$-geometric.
\end{prop}

\begin{cor}
	Let $X$ and $Y$ be $n$-geometric stacks. Then any morphism $X\rightarrow Y$ is $n$-geometric.
\end{cor}

\begin{proof}
	This follows immediately from the definitions and Proposition~\ref{diag geometric}
\end{proof}

\begin{prop}
	\label{diag + proj geometric}
	Let $X\rightarrow Y$ be an effective epimorphism of derived stacks, and suppose that $X$ and $X\times_YX$ are $n$-geometric. Further, assume that the projections $X\times_Y X\rightarrow X$ are $n$-geometric and smooth. Then $Y$ is an $(n+1)$-geometric stack. If in addition $X$ is quasi-compact and $X\rightarrow Y$ is a quasi-compact morphism, then $Y$ is quasi-compact. Finally, if\, $X$ is locally of finite presentation, then so is $Y$.
\end{prop}

\subsection{Quasi-coherent modules over derived stacks}

In this section, we will briefly look at quasi-coherent modules over derived stacks. Again, we will not give detailed proofs of any assertion and refer to \cite[Section 4.3]{NotesDAG} for the details.

The definition of quasi-coherent modules is straightforward. For any derived stack, we want to glue the $\infty$-categories of modules along an atlas. In general, this is achieved by right Kan extension.

\begin{defi}
  Let $X$ be a presheaf on $\AniAlg{R}^{\op}$. We define the \textit{$\infty$-category of quasi-coherent modules over $X$} to be
        $$
		\QQCoh(X)\coloneqq \lim_{\Spec(A)\rightarrow X}\MMod_A.
	$$
	        An element $\Fcal\in\QQCoh(X)$ is called a \textit{quasi-coherent module over $X$} or \textit{$\Ocal_{X}$-module}. For any affine derived scheme $\Spec(A)$ and any morphism $f\colon\Spec(A)\rightarrow X$, we denote the image of a quasi-coherent module $\Fcal$ under the projection $\QQCoh(X)\rightarrow\MMod_{A}$ by $f^{*}\Fcal$.
                
	We define the $\infty$-category of \textit{perfect quasi-coherent modules over $X$} to be
	$$
		\QQCoh_{\perf}(X)\coloneqq \lim_{\Spec(A)\rightarrow X}\MMod^{\perf}_A.
	$$
	We say that a perfect quasi-coherent module $\Fcal$ over $X$ has \textit{Tor-amplitude in $[a,b]$} if for every derived affine scheme $\Spec(A)$ and any morphism $f\colon \Spec(A)\rightarrow X$, the $A$-module $f^{*}\Fcal$ has Tor-amplitude in $[a,b]$. 
\end{defi}

\begin{rem}
\label{qcoh stable}
	As limits of stable $\infty$-categories with finite limit-preserving transition maps are stable, we know that for any $X\in\Pcal(\AniAlg{R}^{\op})$, the $\infty$-category $\QQCoh(X)$ is stable.
\end{rem}

Let us quickly note that the right Kan extension of sheaves is again a sheaf (in a suitable sense). In particular, any $\Fcal\in\QQCoh(X)$, for a derived scheme $X$, can be glued from modules on an affine open cover.

\begin{prop}
	\label{right kan of sheaf}
	Let $\Ccal$ be a presentable $\infty$-category, and let $F\colon \AniAlg{R}\rightarrow \Ccal$ be a $($hypercomplete$)$ sheaf with respect to the Grothendieck topology $\tau\in\lbrace \textup{fpqc, \'etale}\rbrace$ on $\AniAlg{R}$. Let $RF$ denote the right Kan extension of $F$ along the Yoneda embedding $\AniAlg{R}\hookrightarrow\Pcal(\AniAlg{R}^{\op})^{\op}$. Further, let us denote the corresponding $\infty$-topos of $($hypercomplete$)$ $\tau$-sheaves on $\AniAlg{R}$ by $\Shv_{\tau}$.  Then for any diagram $p\colon K\rightarrow \Pcal(\AniAlg{R}^{\op})$, where $K$ is a simplicial set, and morphism $\colim_{K} X_{k}\rightarrow Y$ that becomes an equivalence in $\Shv_{\tau}$ after sheafification,\footnote{We can describe $\Shv_{\tau}$ as a localization of $\Pcal(\AniAlg{R}^{\op})$ (as seen in the proof), so we get a functor $L\colon \Pcal(\AniAlg{R}^{\op})\rightarrow \Shv_{\tau}$ left adjoint to the inclusion, which we call \textit{sheafification}.} we have that the natural map $RF(Y)\rightarrow \lim_{K} RF(X_{k})$ is an equivalence.
\end{prop}

\begin{proof}
  The idea of the proof is to first reduce to $\Ccal\simeq \SS$ using the Yoneda lemma. Then $F$ itself is a derived stack and is thus local with respect to effective epimorphisms of derived stacks.
  
	A detailed proof of this can be found in \cite[Proposition 4.48]{NotesDAG}. 
\end{proof}

\begin{defi}
Let $\Ccal$ be a presentable $\infty$-category, and let  $\tau$ be the fpqc or \'etale topology on $\AniAlg{R}$.
A functor $F\colon \Pcal(\AniAlg{R}^{\op})^{\op}\rightarrow \Ccal$ is a (\textit{hypercomplete\,}) \textit{sheaf or satisfies $\tau$-descent} if for any effective epimorphism $X\rightarrow Y$ (resp.\ a hypercover $X^{\bullet}\rightarrow Y$), we have
$$
RF(Y)\simeq \lim_{\Delta}RF(\Cv(X/Y)_{\bullet})\quad\textup{(resp.\ }RF(Y)\simeq \lim_{\Delta_{s}}RF(X^{\bullet})).
$$
\end{defi}

\begin{rem}
\label{rem.sheaf.kan}
	In the setting of Proposition~\ref{right kan of sheaf}, we see that if $F$ is a (hypercomplete) sheaf, then so is its right Kan extension $RF$.
\end{rem}

\begin{prop}
	\label{right kan ex derived scheme}
	Let $\Ccal$ be a presentable $\infty$-category, and let $F\colon \AniAlg{R}\rightarrow \Ccal$ be an \'etale sheaf. Let $RF$ denote a right Kan extension of\, $F$ along the Yoneda embedding $\AniAlg{R}\hookrightarrow \Pcal(\AniAlg{R})^{\op}$. Then for any derived scheme $X$ over $R$, the natural morphism
	$$
	RF(X)\longrightarrow \lim_{\substack{U\hookrightarrow X \\ \textup{affine open}}} F(U)
	$$ 
	is an equivalence.
\end{prop}

\begin{proof}
	This is a generalization of \cite[Lecture 1, Proposition 3.5]{Khan} using Proposition~\ref{right kan of sheaf} (see \cite[Proposition 4.52]{NotesDAG}). 
\end{proof}

\begin{rem}
	\label{descent qcoh}
	Let us note that the functors $A\mapsto \MMod_{A}$ and $A\mapsto \MMod_A^{\textup{perf}}$ are hypercomplete sheaves for the fpqc topology (see \cite[Corollary~D.6.3.3]{SAG}). Thus, using the definitions of the functors $\QQCoh$ and $\QQCoh_{\perf}$, we see, using Remark~\ref{rem.sheaf.kan}, that these functors are hypercomplete sheaves for the fpqc topology.
\end{rem}

\begin{prop}
\label{derived cat is kan ext}
	Let $X$ be a scheme. Then we have an equivalence of $\infty$-categories $\Dcal_{\textup{qc}}(X)\simeq \QQCoh(X)$, where $\Dcal_{\textup{qc}}(X)$ denotes the derived $\infty$-category of\, $\Ocal_{X}$-modules with quasi-coherent cohomologies.
\end{prop}

\begin{proof}
	This is shown in the spectral setting in \cite{SAG}. In our setting, one can deduce this from Lurie's PhD thesis \cite{DAG} and the spectral setting. For more details, see \cite[Proposition 4.57]{NotesDAG}.
\end{proof}

\subsection{Deformation theory of a derived stack}

This section is derived from \cite[Section 1.4]{TV2}, \cite[Section 4.2]{AG} and \cite[Lecture 5]{Khan}. For detailed properties and the existence of the cotangent complex, we refer to \cite[Section 4.4]{NotesDAG}.

We let $R$ be a ring and assume every derived stack is a derived stack over $R$.

Let $f\colon X\rightarrow Y$ be a morphism of derived stacks. Let $x\colon \Spec(A)\rightarrow X$ be an $A$-point, where $A$ is an animated $R$-algebra. Let $M\in\MMod^\cn_A$, and let us look at the commutative square
$$
\begin{tikzcd}
X(A\oplus M)\arrow[r,""]\arrow[d,""]&X(A)\arrow[d,"f"]\\
Y(A\oplus M)\arrow[r,""]&Y(A)\rlap{,}
\end{tikzcd}
$$ 
where the horizontal arrows are given by the canonical projection $A\oplus M\rightarrow A$. We define the \textit{derivations at the point $x$}
as
$$
\Der_x(X/Y,M)\coloneqq \fib_x(X(A\oplus M)\longrightarrow X(A)\times_{Y(A)}Y(A\oplus M)),
$$
where we see $x$ as a point in the target via the natural map induced by $
\Spec(A\oplus M)\rightarrow\Spec(A) \xrightarrow{x}X\xrightarrow{f}Y.
$

\begin{defi}[\textit{cf.} \protect{\cite[Definition 1.4.1.5]{TV2}}]
\label{defi.cotangent.global}
We say that $L_{f,x}\in\MMod_A$ is a \textit{cotangent complex for $f\colon X\rightarrow Y$ at the point $x\colon \Spec(A)\rightarrow X$} if it is $(-n)$-connective for some $n\geq 0$ and if for all $M\in\MMod_A^\cn$, there is a functorial equivalence
$$
\Hom_{\MMod_A}(L_{f,x},M)\simeq\Der_x(X/Y,M).
$$

	When such a complex $L_{f,x}$ exists, we say that \textit{$f$ admits a cotangent complex at the point $x$}. If there is no possibility of confusion, we also write $L_{X/Y,x}$ for $L_{f,x}$. We also write $L_X$ if $Y\simeq \Spec(R)$.
\end{defi}

\begin{defi}
  We say that $L_f\in \QQCoh(X)$ is a cotangent complex for $f\colon X\rightarrow Y$ if for all points $x\colon \Spec(A)\rightarrow X$, the $A$-module $x^*L_f$ is a cotangent complex for $f$ at the point $x$.
  
	If $L_f$ exists, we say that \textit{$f$ admits a cotangent complex}. We will write $L_{f,x}$ instead of $x^*L_f$ if $f$ admits a cotangent complex.
\end{defi}

\begin{rem}
	Since the cotangent complex is by definition $(-n)$-connective on points for suitable $n$, we can use a version of the Yoneda lemma to see that it is unique up to equivalence (see \cite[Proposition~1.2.11.3]{TV2}). 
\end{rem}

For affine derived schemes, the construction of the cotangent complex shows that it agrees with the cotangent complex of animated rings. The points are then given by base change. 

We list some facts concerning the cotangent complex that can be found in \cite[Section 4.4]{NotesDAG}.

\begin{lem}
	\label{cotangent monomorphism}
	Let $j\colon X\hookrightarrow Y$ be a monomorphism of derived stacks over $A$; then $j$ admits a cotangent complex and $L_j\simeq 0$.
\end{lem}

\begin{proof}
	This is a straightforward calculation. Details can be found in the proofs of \cite[Lecture~5, Proposition 5.9]{Khan} or \cite[Lemma~4.64]{NotesDAG}).
\end{proof}

The vanishing of the cotangent complex shows that it is indeed not useful to impose a monomorphism condition on closed immersions of derived schemes.

\begin{prop}
	\label{cotangent implies smooth}
	Let $f\colon X\rightarrow Y$ be an $n$-geometric morphism of derived stacks. Then $f$ is smooth if and only if $t_{0}f$ is locally of finite presentation and $L_f$ exists, is perfect and has Tor-amplitude in $[-n-1,0]$.
\end{prop}

\begin{proof}
  The proof follows those of \cite[Propositions~4.45 and~4.46]{AG}.
  Let us give a very brief sketch and refer to \cite[Corollary 4.77]{NotesDAG} for details.  
	The idea, again, is induction on $n$. We first reduce to the case where $Y$ is an affine derived scheme. Then we show the assertion for $(-1)$-geometric derived stacks, \textit{i.e.}\ affine derived schemes, which is more or less clear. The ``only if'' direction follows using an atlas of $X$, inducing a diagram of the form $U\xrightarrow{p} X\xrightarrow{f} Y$, with $p$ being $(n-1)$-geometric. For the ``if'' direction, one has to work more. 
 
	We first define $L_{f}$ as the  fiber of $L_{f\circ p}\rightarrow L_{p}$, which exists by induction. Now we need to understand the space of derivations and show that $L_{f}$ indeed represents it. The proof of this is too involved to sketch here, and we refer to \textit{loc.~cit.}~for the details. The fact that $L_{f}$ is perfect and has the right Tor-amplitude then follows immediately by a 2-out-of-3 argument.
\end{proof}

\begin{cor}
	\label{global lfp cotangent}
	Let $f\colon X\rightarrow Y$ be an $n$-geometric morphism of derived stacks locally of finite presentation. Then $L_{f}$ is perfect. 
\end{cor}

\begin{proof}
	Let us sketch the proof of \cite[Corollary 4.78]{NotesDAG}.
 
	We first reduce to the case where $Y$ is affine. Then we use that $X$ has a smooth atlas $p\colon U\rightarrow X$. As perfect modules satisfy fpqc descent, we can check perfectness of $L_{f}$ after pullback to $U$. Now we can use the cofiber sequence of the cotangent complexes with respect to $f$ and $p$ to see that a 2-out-of-3 argument, together with Proposition~\ref{cotangent implies smooth}, implies the result.
\end{proof}

\begin{lem}
\label{geometric hypercomplete}
	Let $X$ be an $n$-geometric stack; then $X$ is hypercomplete.
\end{lem}

\begin{proof}
	The proof is analogous to that of \cite[Corollary 5.3.9]{DAG}. The basic idea is to show that for any geometric stack $X$, we have $X\simeq \lim_{n} X\circ \tau_{n}$. Then we use that $ X\circ \tau_{n}$ takes values in $(n+k)$-groupoids, for some sufficiently large $k$, and the fact that truncated sheaves are automatically hypercomplete (see \cite[Lemma 6.2.9]{HTT}). 
 
	A detailed proof is given in \cite[Lemma 4.81]{NotesDAG}.
\end{proof}

\subsection{The stack of perfect modules}
\label{sec:perf}

In this section, we want to recall that the derived stack of perfect modules is locally geometric. This was already proven in \cite{TVaq} in the model-categorical setting, and a detailed proof can be found in \cite{NotesDAG}. We recall some lemmas needed for the proof, as they will become important later on when we analyze the substacks of derived $F$-zips.

\begin{lem}
\label{upper semi-cont qc}
	Let $A$ be a commutative ring and $P$ be a perfect complex of $A$-modules, and let $n\in \NN_{0}$. Further, for $k\in \ZZ$, let $\beta_{k}\colon \Spec(A)_{\cl}\rightarrow \NN_{0}$ be the function given by $s\mapsto \dim_{\kappa(s)}\pi_{k}(P\otimes_{A}\kappa(s))$. Then $\beta_{k}^{-1}([0,n])$ is quasi-compact open.
\end{lem}

\begin{proof}
	This follows from \cite[0BDI]{stacks-project}.
\end{proof}

\begin{rem}
\label{tor and type}
	Let $A$ be a commutative ring and $P$ be a perfect complex of $A$-modules. Let $I\subseteq \ZZ$ be a finite subset, and for $k\in \ZZ$, let $\beta_{k}$ be as in Lemma~\ref{upper semi-cont qc}. Assume that $\beta_i$ is non-zero for $i\in I$ and zero everywhere else. Then using \cite[0BCD,066N]{stacks-project}, we see that $P$ has Tor-amplitude in $[\min(I),\max(I)]$.
\end{rem}

\begin{lem}
\label{lem perfect zero}
	Let $A$ be a commutative ring and $P$ be a perfect complex of $A$-modules. Then there exists a quasi-compact open subscheme $U\subseteq\Spec(A)_{\cl}$ with the following property: 
	\begin{itemize}
		\item an affine scheme morphism $\Spec(B)_{\cl}\rightarrow \Spec(A)_{\cl}$ factors through $U$ if and only if $P\otimes_{A} B\simeq 0$.
	\end{itemize}
\end{lem}

\begin{proof}
	This follows from the upper semi-continuity of the Betti numbers (see \cite[0BDI]{stacks-project}). For a detailed proof, we refer to \cite[Lemma 5.3]{NotesDAG}.
\end{proof}

The next lemma shows that the vanishing locus of perfect complexes is quasi-compact open. This will be applied to the cofiber of morphisms of perfect complexes. In particular, the locus classifying equivalences between fixed perfect modules is therefore quasi-compact open.

\begin{lem}
\label{equiv zariski open}
Let $A\in\AniAlg{R}$ and $P$ be a perfect $A$-module. Define the derived stack $V_P$ by letting $V_P(B)$ be the full sub-$\infty$-category of\, $\Hom_{\AniAlg{R}}(A,B)$ consisting of morphisms $u\colon A\rightarrow B$ such that $P\otimes_{A,u}B\simeq 0$. This is a quasi-compact open substack of\, $\Spec(A)$. 
\end{lem}

\begin{proof}
	The proof follows that of \cite[Proposition 2.23]{TVaq}. The idea of the proof is to use the upper semi-continuity of the Betti numbers of a perfect complex to see the openness on $\pi_{0}A$ and then lift the quasi-compact openness to $A$. For details, see \cite[Lemma 5.4]{NotesDAG}.	
\end{proof}

Next, we want to remark that the stack classifying morphisms between perfect modules is actually geometric and in good cases smooth. Since derived $F$-zips will come with two bounded perfect filtrations (\textit{i.e.}\ finite chains of morphisms of perfect modules), this lemma is crucial for the geometricity of derived $F$-zips.

\begin{lem}
	\label{spec sym geometric}
	Let $A$ be an animated $R$-algebra. Let $P$ be a perfect $A$-module with Tor-amplitude concentrated in $[a,b]$ with $a\leq 0$.
	Then the derived stack 
	\begin{align*}
		F^{A}_P\colon \AniAlg{A}&\longrightarrow \SS\\
		B&\longmapsto \Hom_{\MMod_A}(P,B)
	\end{align*}
	is $(-a-1)$-geometric and locally of finite presentation over $\Spec(A)$. $($We can view $F_{P}^{A}$ as a derived stack over $R$ with a morphism to $\Spec(A)$. So for any animated $R$-algebra $C$ that does not come with a morphism $A\rightarrow C$, the value of $F_{P}^{A}$ is empty.$)$ Further, the cotangent complex of $F^{A}_P$ at a point $x\colon \Spec(B)\rightarrow F^{A}_P$ is given by 
	$$
	L_{F_P^{A},x}\simeq P\otimes_A B.
	$$

	In particular, if $b\leq0$, then $F_P^{A}$ is smooth.
\end{lem}

\begin{proof}
  The proof follows that of \cite[Theorem 5.2]{AG}. First one computes the space of derivations associated to $F_{P}^{A}$. Then we use that any perfect complex of Tor-amplitude in $[a,b]$ sits in a fiber sequence of perfect modules, where the remaining modules have Tor-amplitude in $[a+1,b]$, resp.\ $[a,a]$ (see Lemma~\ref{general props of Tor}).
  Lastly, an inductive argument concludes the proof (see \cite[Lemma 5.5]{NotesDAG} for further details).	
\end{proof}

\begin{theorem}
	\label{main thm}
	The derived stack
	\begin{align*}
	\PPerf_R\colon \AniAlg{R} &\longrightarrow \SS\\
	A&\longmapsto (\MMod_A^{\textup{perf}})^{\simeq}
	\end{align*}
	is locally geometric and locally of finite presentation.

	To be more specific, we can write $\PPerf_R = \colim_{a\leq b} \PPerf_R^{\,[a,b]}$, where $\PPerf_R^{\,[a,b]}$ is the moduli space consisting of perfect modules which have Tor-amplitude concentrated in degree $[a,b]$, each $\PPerf_R^{\,[a,b]}$ is $(b-a+1)$-geometric and locally of finite presentation and the inclusion $\PPerf_{R}^{\,[a,b]}\hookrightarrow \PPerf_{R}$ is a quasi-compact open immersion. If\, $b-a\leq 1$, then $\PPerf_R^{\,[a,b]}$ is in fact smooth. 
\end{theorem}

\begin{proof}
	The proof is analogous to those of \cite[Theorem 5.6]{AG}, \cite[Proposition 3.7]{TVaq}, and a full proof in this setting can be found in \cite[proof of Theorem~5.14]{NotesDAG}. The basic idea is to first prove that $\PPerf^{\,[a,a]}\simeq \BGL$ is geometric, then proceed by induction, using that any perfect module sits in a cofiber sequence of perfect modules with smaller Tor-amplitude. An atlas is then given by taking fibers of morphisms of perfect modules with smaller Tor-amplitude. In particular, Lemma~\ref{spec sym geometric} gives us the geometricity and smoothness of the atlas. The surjectivity is reduced to classical statements, but we refer to \cite[Theorem~5.14]{NotesDAG} for details as this reduction is too involved. The openness follows again from the upper semi-continuity of the Betti numbers of a perfect complex.
\end{proof}

\section{Derived \texorpdfstring{$\boldsymbol{F}$}{F}-zips}
\label{sec:F-Zip}

In the following, we fix a prime $p$ and an $\FF_{p}$-algebra $R$. Starting from here, it is important that we have chosen the $\infty$-category of animated rings for our study of derived algebraic geometry. The main reason is that we want to have a Frobenius in characteristic $p>0$. For $E_{\infty}$-rings, it is not clear how to define a Frobenius morphism. But for animated rings, we  naturally have a Frobenius. Namely, if we see an animated ring $A$ over $\FF_{p}$ as a contravariant functor from $\Poly_{\FF_{p}}$ to $\SS$, then the Frobenius morphism induces a natural transformation of the animated ring to itself, which we denote by $\Frob\colon A\rightarrow A$. For any animated $R$-algebra $A$ and any $A$-module $M$, we denote the base change of $M$ by the Frobenius of $A$ by $M^{(1)}\coloneqq M\otimes_{A,\Frob}A$. If $A$ is discrete, we can see an $M$-module as an element in the derived category via the equivalence $\MMod_{A}\simeq \Dcal(A)$, and we have $M^{(1)}\simeq M\otimes^{L}_{A,\Frob}A$ (here we abuse notation and identify $A$ with $\pi_{0}A$ if $A$ is discrete).
 
We want to define derived versions of $F$-zips presented in \cite{moon-wed}. In the reference, Moonen--Wedhorn defined $F$-zips over schemes of characteristic $p>0$ and analyzed the corresponding classifying stack. One application is the $F$-zip associated to a scheme with degenerate Hodge--de Rham spectral sequence. Examples of those are abelian schemes and K3-surfaces. The degeneracy of the spectral sequence is used to get two filtrations (note that  the conjugate spectral sequence also degenerates) on the $\supth{i}$ de Rham cohomology. Our goal is to eliminate the extra information given by the degeneracy of the spectral sequences. This information seems unnecessary since the two spectral sequences are induced by filtrations on the de Rham hypercohomology, and thus if we pass to the derived categories, we can use the perfectness of the de Rham hypercohomology, the two filtrations and the Cartier isomorphism to get derived $F$-zips, as explained in the following example.

\begin{example}
\label{ex fzip de rham}
	Let $f\colon X\rightarrow S$ be a proper smooth morphism of schemes, where $S$ is an $R$-scheme. The complex $Rf_*\Omega^\bullet_{X/S}$ is perfect and commutes with arbitrary base change (see \cite[0FM0]{stacks-project}). The conjugate and Hodge filtrations on the de Rham complex induce functors $\conj\colon\ZZ\rightarrow \Dcal(S)$ and $\HDG\colon\ZZ^{\op}\rightarrow \Dcal(S)$ given by\footnote{Here $\tau_{\leq n}$ denotes the canonical truncation and $\sigma_{\geq n}$ the stupid truncation in the sense of \cite[0118]{stacks-project}.} $\conj(n)= Rf_*\tau_{\leq n}\Omega^\bullet_{X/S}$ and $ \HDG(n)= Rf_*\sigma_{\geq n}\Omega^\bullet_{X/S}$ (recall that we see the ordered set $\ZZ$ as a $1$-category (and thus via the nerve functor as an $\infty$-category), where we have a unique map between $a,b\in\ZZ$ if and only if $a\leq b$). The associated colimits are naturally equivalent as we have $$\colim_{\ZZ}\conj\simeq Rf_*\Omega^\bullet_{X/S}\simeq \colim_{\ZZ^{\op}}\HDG.$$
 
	For $n\geq 0$, we have the following exact sequence of complexes of $f^{-1}\Ocal_{S}$-modules: 
      \begin{gather*}
	0\longrightarrow \tau_{\leq n-1}\Omega^\bullet_{X/S}\overset{\del^{n}}\longrightarrow \tau_{\leq n}\Omega^\bullet_{X/S}\longrightarrow \Coker(\del^{n})\longrightarrow 0,\\[2ex]      
        0\longrightarrow \sigma_{\geq n+1}\Omega^\bullet_{X/S}\longrightarrow \sigma_{\geq n}\Omega^\bullet_{X/S}\longrightarrow \Omega^n_{X/S}[-n]\longrightarrow 0.\end{gather*}
	Note that there is a quasi-isomorphism $\Coker(\del^{n})\xrightarrow{\lowsim}\Hcal^n(\Omega^\bullet_{X/S})[-n]$ in $\Dcal(S)$.
	These induce fiber sequences in $\Dcal(S)$ of the form
        \begin{gather*}
        \conj(n-1)\longrightarrow \conj(n)\longrightarrow Rf_*\Hcal^n(\Omega^\bullet_{X/S})[-n], \\[2ex]
		 \HDG(n+1)\longrightarrow \HDG(n)\longrightarrow Rf_*\Omega^n_{X/S}[-n].
	\end{gather*}
	It makes sense to think of $Rf_*\Hcal^n(\Omega^\bullet_{X/S})[-n]$  and $Rf_*\Omega^n_{X/S}[-n]$ as \textit{``cokernels''} of the respective maps in the distinguished triangles (as they are the cofibers in the stable $\infty$-category $\Dcal(S)$).

	The notation $\conj$ and $\HDG$ is chosen to indicate their influence on the classical conjugate and Hodge filtrations. Using these functors, one can naturally associate converging spectral sequences (as explained for example in \cite[Definition 1.2.2.9, Proposition 1.2.2.14]{HA} or \cite[0FM7]{stacks-project} for the Hodge filtration) on $R^{i}f_{*}\Omega^{\bullet}_{X/S}$. The filtration on the $\supth{i}$ cohomology of the colimit of $\HDG$ ($\simeq Rf_{*}\Omega^{\bullet}_{X/S}$), for example, is given by
        $$
        F^{n}R^{i}f_{*}\Omega^{\bullet}_{X/S} \coloneqq \im (H^{i}(\HDG(n))\longrightarrow R^{i}f_{*}\Omega^{\bullet}_{X/S}).
        $$
        The spectral sequence associated to the Hodge functor is given by 
	$$
	E_{1}^{p,q}=H^{q}(X,Rf_*\Omega^p_{X/S})=H^{q+p}(X,Rf_*\Omega^p_{X/S}[-p])\Longrightarrow R^{p+q}f_{*}\Omega^{\bullet}_{X/S}.
	$$
	Therefore, it seems reasonable to think of $\conj$, resp.\ $\HDG$, as an \textit{ascending}, resp.\ \textit{descending}, \textit{filtration} (see Definition~\ref{defi fil} below) with \textit{graded pieces}
        $$
        \gr^{n}\conj \coloneqq Rf_{*}\Hcal^{n}(\Omega^{\bullet}_{X/S})[-n]\quad \textup{resp.}\ \gr^{n}\HDG\coloneqq Rf_{*}\Omega^{n}_{X/S}[-n]
          $$
          (see Definition~\ref{defi gr} below).

	The Cartier isomorphism gives an equivalence $Rf_*\Hcal^{n}(\Omega^{\bullet}_{X/S})\simeq Rf^{(1)}_*\Omega^{n}_{X^{(1)}/S}$.
	Again by \cite[0FM0]{stacks-project}, $Rf_*\Omega^{n}_{X/S}$ commutes with arbitrary base change, and therefore $$(\gr^{n}\HDG)^{(1)}\simeq (Rf_*\Omega^{n}_{X/S})^{(1)}[-n]\simeq Rf^{(1)}_*\Omega^{n}_{X^{(1)}/S}[-n]\simeq Rf_*\Hcal^{n}(\Omega^{\bullet}_{X/S})[-n]\simeq \gr^{n}\conj .$$
	We claim that $\conj$ and $\HDG$ take values in perfect complexes of $\Ocal_{S}$-modules and their respective graded pieces are perfect.
 
	Indeed, first note that we can check this Zariski locally, so we may assume that $S$ is affine and in particular quasi-compact. Then for any $n\in\ZZ$, the complexes $\gr^{n}\HDG$  and $Rf_{*}\Omega^{\bullet}_{X/S}$ are perfect, and their formation commutes with arbitrary base change  (see \cite[0FM0]{stacks-project}). Since $\sigma_{\geq 0}\Omega^{\bullet}_{X/S}=\Omega^{\bullet}_{X/S}$, we see inductively using the distinguished triangles above that for all $n\in\ZZ$, the complex $\HDG(n)$ is perfect. Now certainly the base change of perfect complexes is perfect, and therefore the Cartier isomorphism shows that the graded pieces of $\conj$ are also perfect. The quasi-compactness of $S$ implies that there is an $n\in\NN_{0}$ such that $\tau_{\leq n}\Omega_{X/S}^{\bullet}=\Omega_{X/S}^{\bullet}$,  and thus again inductively with the distinguished triangles above, we see that $\conj(n)$ is perfect for all $n\in\ZZ$.
 
	Note that we heavily used that Zariski locally, there is an $n\gg 0$ such that for any $k\leq 0$ and $j\geq 0$, we have $\conj(k-1)\simeq 0\simeq \HDG(n+j)$ and $\conj(n+j)\simeq Rf_{*}\Omega_{X/S}^{\bullet}\simeq \HDG(k)$, and $\conj(n+j)\rightarrow\conj(n+j+1)$ and $\HDG(k+1)\rightarrow\HDG(k)$ are equivalent to the identity.
\end{example}

The example above gives us an idea for the definition of \textit{derived $F$-zips} (see Definition~\ref{defi derived F-zip}). Namely, a derived $F$-zip should consist of \textit{two filtrations} (one descending and one ascending) with perfect values that are locally determined by a finite chain of morphisms, \textit{i.e.}\ functors
$$
C^{\bullet}\in\Fun(\ZZ^{\op},\PPerf(S))\quad \textup{and}\quad D_{\bullet}\in\Fun(\ZZ,\PPerf(S)),
$$
such that their colimits are equivalent and that on affine opens, they are up to equivalence determined by their values on a finite ordered subset of $\ZZ$, together with \textit{equivalences} $\varphi_{\bullet}$ of their graded pieces up to Frobenius twist, \textit{i.e.}\ for $\gr^{n}C\coloneqq \cofib(C^{n+1}\rightarrow C^{n})$ and $\gr^{n}D\coloneqq \cofib(D_{n-1}\rightarrow D_{n})$, equivalences of the form
$$
\varphi_{n}\coloneqq (\gr^{n}C)^{(1)}\xrightarrow{\;\lowsim\;}\gr^{n}D.
$$
The $\infty$-category of derived $F$-zips should then be defined as the $\infty$-category of such triples $(C^{\bullet},D_{\bullet},\varphi_{\bullet}).$

\subsection{Filtrations}
\label{sec.filtration}
In the following, $A$ will denote an animated ring.

We will now define the notions of a \textit{filtration} and \textit{graded pieces} and look at properties of filtrations. These definitions are highly influenced by the work of Gwilliam--Pavlov \cite{GW} and Example~\ref{ex fzip de rham}. 

\begin{defi}
	A morphism $f\colon M\rightarrow N$ of $A$-modules is called a \textit{monomorphism} if for all $i\in \ZZ$, the morphism $\pi_{i}f$ is injective. We call $f$ a \textit{split monomorphism} if $f$ admits a retraction.
\end{defi}

\begin{rem}
Note that a monomorphism in our sense is not equivalently an $\infty$-categorical monomorphism; \textit{i.e.}\ if $f$ is a monomorphism of $A$-modules in our sense, then the diagonal $f$ may not be an equivalence.\footnote{Note that in stable $\infty$-categories, pullback diagrams are equivalently pushout diagrams. So if $f\colon M\rightarrow N$ is a morphism of $A$-modules such that the diagonal is an equivalence, then $f$ is an equivalence.}
 
Further, let us observe that any split monomorphism is automatically a monomorphism but the other way is not necessarily true.
For this let us note that a split monomorphism $f\colon M\rightarrow N$ is equivalently a splitting of the fiber sequence $M\xrightarrow{f}N\rightarrow\cofib(f)$; \textit{i.e.}~we have an equivalence of cofiber sequences
	$$
	\begin{tikzcd}
		M\arrow[r,"f"]\arrow[d,"\id"]&N\arrow[r,""]\arrow[d,"\simeq"]& \cofib(f)\arrow[d,"\id"]\\
		M\arrow[r]&M\oplus\cofib(f)\arrow[r,""]&\cofib(f)\rlap{.}
	\end{tikzcd}
	$$
	In particular, the natural morphism $\ZZ\xrightarrow{\cdot 2} \ZZ$ is a monomorphism in our sense but is not split (otherwise, the short exact sequence 
	$$
	0\longrightarrow \ZZ\overset{\cdot 2}\longrightarrow\ZZ\longrightarrow \ZZ/2\ZZ\longrightarrow 0
	$$
	would be split).
\end{rem}

\begin{defi}
\label{defi fil}
	An \textit{ascending} (\textit{resp.\ descending}) \textit{filtration of $A$-modules} is an element $F\in \Fun(\ZZ,\MMod_A)$ (resp.\ $F\in \Fun(\ZZ^{\op},\MMod_A)$).  
 
	We call an ascending (resp.\ descending) filtration $F$
	\begin{enumerate}
		\item \textit{right bounded} if there exists an $i\in \ZZ$ such that the natural map $F(k)\rightarrow \colim_{\ZZ} F$ (resp.\ $F(k)\rightarrow \colim_{\ZZ^{\op}} F$) is an equivalence for all $i\leq k$ (resp.\ $i\geq k$),
		\item \textit{left bounded} if there exists an $i\in \ZZ$ such that the natural map $0\rightarrow F(k)$ is an equivalence for all $k\leq i$ (resp.\ $k\geq i$), 
		\item \textit{bounded} if it is left and right bounded,
		\item \textit{perfect} if $F$ takes values in $\MMod_A^{\perf}$,
		\item \textit{strong} if for all $i\leq j$ (resp.\ $j\leq i$), we have that $F(i)\rightarrow F(j)$ is a monomorphism.
	\end{enumerate}
\end{defi}

\begin{rem}
  The definition of a \textit{strong filtration} seems natural since for a discrete module $M$ over a discrete ring $A$, a filtration is usually defined as a filtered chain of submodules
  $$
  \dots\subseteq M_{i}\subseteq M_{i+1} \subseteq \dots \subseteq M
  $$
  (for simplicity, we only consider ascending filtrations).  But we can show that the Hodge filtration $\HDG$ of Example~\ref{ex fzip de rham} is strong \textbf{if and only if} the Hodge--de Rham spectral sequence is degenerate and the modules of the $E_{1}$-page are finite locally free (see Theorem~\ref{de Rham strong degen}). Since we are particularly interested in the cases where the Hodge--de Rham spectral sequence is non-degenerate, strong filtrations are not used in the definition of derived $F$-zips. 
\end{rem}

The $\infty$-category of $A$-modules is stable. Thinking of stable $\infty$-categories as analogues of abelian categories, we may think of cofibers as cokernels. This allows for a definition of graded pieces of a filtration, that was used in Example~\ref{ex fzip de rham}.

\begin{defi}
\label{defi gr}
	Let $F$ be an ascending (resp.\ descending) filtration of $A$-modules. For any $i\in \ZZ$, we define the \textit{$\supth{i}$ graded piece of\, $F$} as $\gr^i F\coloneqq \cofib (F(i-1)\rightarrow F(i))$ (resp.\ $\gr^i F\coloneqq \cofib (F(i+1)\rightarrow F(i))$).
 
\end{defi}

\begin{rem}
	By the construction of the category $\Fun(\ZZ,\MMod_A)$, one sees that two filtrations $F$ and $G$ are equivalent if and only if there is a morphism $F\rightarrow G$ such that for all $n\in \ZZ$, the induced morphism $F(n)\rightarrow G(n)$ is an equivalence of $A$-modules. However, one can show that a morphism of bounded filtrations is an equivalence if and only if it induces an equivalence on the graded pieces (this is an easy consequence using induction or \cite[Remark 3.21]{GW}).
\end{rem}

\begin{remark}
	Note that for a perfect filtration $F$ of $A$-modules, the graded pieces $\gr^i F$ are again perfect (since the $\infty$-category of perfect  modules is per definition stable; see \cite[Section 7.2.4]{HA}).

\end{remark}

\begin{remark}
	\label{Day convolution}
	We want to attach a monoidal structure to filtrations of $A$-modules (we will only consider ascending filtrations, but the arguments work analogously for descending filtrations as explained at the end of this remark).
        First, note that for any (symmetric) monoidal $\infty$-category $\Ccal$, the $\infty$-category $\Fun(\ZZ,\Ccal)$ has two monoidal structures. The first one is simply given by termwise tensor product (see \cite[Remark 2.1.3.4]{HA}); the other one is given by the Day convolution (see \cite[Example 2.2.6.17]{HA}).
        We will not use the monoidal structure given by termwise tensor product since we want to consider bounded filtrations, and for such we do not have a unit element with respect to the termwise tensor product. Having this in mind, we will look closely into the monoidal structure induced by the Day convolution, which we will explain in the following.
 
	For the Day convolution, we first need a (symmetric) monoidal structure on $\ZZ$. For this, we simply take $\ZZ$ with the usual addition, seen as a symmetric monoidal structure on $\ZZ$. Then the Day convolution of two elements $F,G\in \Fun(\ZZ,\MMod_A)$, denoted by $F\otimes G$, is given by the formula
	$$
        (F\otimes G)(k)\simeq \colim_{n+m\leq k} F(n)\otimes_AG(m),
        $$
        where we take the colimit over the category of triples $(a,b,a+b\rightarrow k)$ where $a,b\in \ZZ$ and $a+b\rightarrow k$ is a morphism in $\ZZ$ (recall that this simply means $a+b\leq k$) and the morphisms are given componentwise, \textit{i.e.}\ a morphism
        $$
        (a,b,a+b\rightarrow k)\longrightarrow (a',b',a'+b'\rightarrow k')
        $$
        is given by the relations $a\leq a'$, $b\leq b'$ and $k\leq k'$. A unit element for this tensor product is given by the bounded perfect filtration $A_\bullet^{\triv}$ on $A$, where $A_i^{\triv}\simeq A$ for $i\geq 0$ and $A_i^{\triv}=0$ otherwise, the maps $A_m^{\triv}\rightarrow A_n^\triv$ for $0\leq m\leq n$ are given by the identity and the maps $A_m^{\triv}\rightarrow A_n^\triv$ for $m\leq n\leq 0$ are given by~$0$.
 
	If we replace ascending filtrations with descending ones, the relations above get opposed; \textit{i.e.}\ we have unique morphisms $a\rightarrow b$ in $\ZZ^{\op}$ if and only if $b\leq a$. Taking this to account, we can dually define the Day convolution for descending filtrations similarly.
\end{remark}

\begin{notation}
	In the definition of derived $F$-zips, we will have an ascending and a descending filtration. For clarity,  for an ascending filtration $F\in\Fun(\ZZ,\MMod_{A})$, we denote its values by $F_{n}\coloneqq F(n)$ for any $n\in\ZZ$, and for a descending filtration $G\in\Fun(\ZZ^{\op},\MMod_{A})$, we denote its values by $G^{n}\coloneqq G(n)$. We also denote the filtrations by $F_{\bullet}\coloneqq F$ and $G^{\bullet}\coloneqq G$.
 For the gradings,  we omit the $\bullet$; \textit{i.e.}\ we write $\gr^{i}F\coloneqq\gr^{i}F_{\bullet}$ and $\gr^{i}G\coloneqq\gr^{i}G^{\bullet}$
\end{notation}

\begin{remark}
\label{zigzag}
	Let us visualize the Day convolution using an easy example. Let $M$ and $N$ be $A$-modules, and let $C\rightarrow M$ and $D\rightarrow N$ be morphisms of $A$-modules. Now let us look at the filtrations $C_\bullet$ and $D_\bullet$ given 
	\begin{align*}
		C_{\bullet}&\colon\cdots\xrightarrow{\;\;0\;\;} 0\xrightarrow{\;\;0\;\;}C\longrightarrow M\xrightarrow{\;\id\;} M\xrightarrow{\;\id\;} \cdots\\
		D_{\bullet}&\colon \cdots\xrightarrow{\;\;0\;\;} 0\xrightarrow{\;\;0\;\;}D\longrightarrow N\xrightarrow{\;\id\;} N\xrightarrow{\;\id\;}\cdots\rlap{,}
	\end{align*}
	where we set $C_{0}=C$ and $D_{0} = D$. Then we have $(C_\bullet\otimes D_\bullet)_0\simeq C\otimes_A D$, and the $A$-module $(C_\bullet\otimes D_\bullet)_1$ is given by the pushout of the  diagram
	$$
	\begin{tikzcd}
	&C\otimes_A D\arrow[dr,""]\arrow[dl,""]&\\
	M\otimes_A D& & C\otimes_A N\rlap{.}
	\end{tikzcd}
	$$
	The $A$-module $(C_\bullet\otimes D_\bullet)_2$ is given by the colimit of the diagram
	$$
	\begin{tikzcd}
	& &C\otimes_A D\arrow[dr,""]\arrow[dl,""]\arrow[dd,""]& &\\
	&M\otimes_A D\arrow[dr,""]\arrow[dl,"\id", swap]& & C\otimes_A N\arrow[dr,"\id"]\arrow[dl,""]&\\
	M\otimes_A D &  & M\otimes_A N & &C\otimes_A N\rlap{;}
	\end{tikzcd}
	$$
	in particular, we may forget about the topmost module and only look at the colimit of the bottom zigzag (in the homotopy category). This diagram makes clear that $(C_\bullet\otimes D_\bullet)_2\simeq M\otimes_A N$. The same visualization works for higher degrees of the filtration $(C_\bullet\otimes D_\bullet)_\bullet$, and we will prove in the following proposition that the Day convolution descends to perfect bounded filtrations having this tree structure in mind.
\end{remark}

\begin{prop}
	\label{Day descend}
	The Day convolution on $\Fun(\ZZ,\MMod_A)$ descends to a symmetric monoidal structure on the full subcategory of perfect bounded ascending filtrations.
 The same holds for perfect bounded descending filtrations.
\end{prop}
\begin{proof}
	That the unit element for the Day convolution is a bounded perfect filtration on an $A$-module is shown in Remark~\ref{Day convolution}.
 
	Let $C_\bullet$ and $D_\bullet$ be bounded ascending filtrations of $A$-modules. We claim that $(C_\bullet\otimes D_\bullet)_{\bullet}$ is a bounded ascending filtration.
	That  $(C_\bullet\otimes D_\bullet)_{\bullet}$ defines a left-bounded filtration is clear.  To see that it is also right bounded, fix some integers $k,k'\in \ZZ$ such that the natural morphisms
        $$
        C_i\xrightarrow{\;\lowsim\;} \colim_{\ZZ}  C_{\bullet}\quad\textup{and}\quad D_j\xrightarrow{\;\lowsim\;} \colim_{\ZZ} D_{\bullet}
        $$
        are equivalences for all $i\geq k$ and $j\geq k'$. For simplicity, let us set $M\coloneqq \colim_{\ZZ}C_{\bullet}$ and $N\coloneqq \colim_{\ZZ}D_{\bullet}$. First, note that the morphism $C_{i\leq i+1}\colon C_i\rightarrow C_{i+1}$ is an equivalence for all $i\geq k$. Using this equivalence, we may assume that $C_{i\leq i+1}$ is given by $\id_M$ (it is not hard to find an equivalence of filtrations); we do the same for $D_\bullet$. Now let us look at $(C_\bullet\otimes D_\bullet)_{k+k'}$; we claim that this term is equivalent to $M\otimes N$. Indeed, $C_k\otimes_A D_{k'}\simeq M\otimes_A N$ by construction. Now let $(i,j)\in \ZZ^2$ be such that $i+j\leq k+k'$ but $i> k$ or $j>k'$, so there is no morphism from $C_i\otimes_{A} D_j$ to $C_k\otimes_A D_{k'}$. Without loss of generality, assume $i>k$ (in particular, $j<k'$).
 
	Let us visualize what we are going to do. Considering the zigzag from Remark~\ref{zigzag}, we will look at the following diagram: 
	$$
	\begin{tikzcd}
	C_{k}\otimes_A D_{j}\arrow[dr,"\id"]\arrow[d,""]\arrow[ddddr,"f"]& &  &\\
	\vdots \arrow[d,""]& C_{k+1}\otimes_{A}D_{j}\arrow[dr,"\id"]&  &\\
	\vdots \arrow[d,""] & & \dots \arrow[dr,"\id"] &\\
	C_{k}\otimes_A D_{k'} \arrow[dr,"g"] & & &  C_{i}\otimes_A D_{j}\arrow[dll,"h"]\\
	& (C_\bullet\otimes D_\bullet)_{k+k'}\rlap{.}& &
	\end{tikzcd}
	$$
By the definition of colimits, we automatically get a homotopy between $f$ and $h$ and a homotopy between $f$ and $g$. In particular, this diagram shows that $h$ is, up to homotopy, uniquely determined by $f$ and $g$. But certainly the morphisms of the filtrations and $g$ uniquely (up to homotopy) determine $f$. Using this and the universal property of colimits, we see that there exists a morphism $p\colon (C_\bullet\otimes D_\bullet)_{k+k'}\rightarrow C_{k}\otimes_A D_{k'}$ such that $\id_{C_{k}\otimes_A D_{k'}}\simeq p\circ g$. But $g\circ p$ induces a map $(C_\bullet\otimes D_\bullet)_{k+k'}\rightarrow (C_\bullet\otimes D_\bullet)_{k+k'}$ that is compatible with all transition maps in the colimit diagram, and thus $g\circ p\simeq \id_{(C_\bullet\otimes D_\bullet)_{k+k'}}$. In other words, $M\otimes_A N \simeq C_k\otimes_A D_{k'}\simeq (C_\bullet\otimes D_\bullet)_{k+k'}$.
 
The same argument shows that $(C_\bullet\otimes D_\bullet)_{l}\simeq M\otimes_A N$ and that the canonical maps $(C_\bullet\otimes D_\bullet)_{l}\rightarrow (C_\bullet\otimes D_\bullet)_{l+1}$ are homotopic to the identity for all $l\geq k+k'.$ So for all $l\geq k+k'$, the natural map
$$
(C_\bullet\otimes D_\bullet)_{l}\xrightarrow{\;\lowsim\;} \colim_{\ZZ}(C_\bullet\otimes D_\bullet)_{\bullet}\simeq M\otimes_{A} N
$$
is an equivalence;  \textit{i.e.}\ $(C_\bullet\otimes D_\bullet)_{\bullet}$ is right bounded.

The above computations show that since $C_{\bullet}$ and $D_{\bullet}$ are bounded,  for any $k\in \ZZ$, we have that $(C_{\bullet}\otimes D_{\bullet})_{k}$ is equivalent to the colimit taken over a finite filtered subset of $\ZZ$ (again seen as a category via the nerve functor and morphisms uniquely given by relations). Since finite colimits of perfect modules are perfect, we see that $(C_{\bullet}\otimes D_{\bullet})_{\bullet}$ is not only bounded but also perfect (note that stable $\infty$-categories are closed under finite colimits; see \cite[Proposition 1.1.3.4]{HA}).

	Combining everything above, we see that the Day convolution descends to bounded perfect filtrations and therefore gives us a symmetric monoidal structure on bounded perfect filtrations (see \cite[Proposition~2.2.1.1, Remark~2.2.1.2]{HA}).
 
	The proof for descending filtrations works analogously.
\end{proof}

\begin{remark}
	\label{graded of tensor}
	For two filtrations $C_\bullet$ and $D_\bullet$ of $A$-modules, it is known that $$\gr^k (C_\bullet\otimes D_\bullet)  \simeq \bigoplus_{n\in\ZZ} \gr^nC\otimes_{A} \gr^{k-n}D$$
	(see \cite[Lemma 5.2]{BMS2}).
\end{remark}

\begin{rem}
\label{day on gr}
Let us note that the construction of the Day convolution can also be done for $\Fun(\Ccal,\Dcal)$, where $\Ccal$ and $\Dcal$ are arbitrary symmetric monoidal $\infty$-categories (see \cite[Example 2.2.6.17]{HA}).

	An interesting example for us occurs if $\Ccal\simeq \ZZ^{\textup{disc}}$ (recall that this means the set $\ZZ$ as a discrete $1$-category and thus an $\infty$-category via the nerve functor), where we endow $\ZZ^{\disc}$ with a symmetric monoidal structure by addition, \textit{i.e.}\ $a\otimes b\coloneqq a+b$, and $\Dcal\simeq \MMod_{A}^{\perf}$. Then for functors $F,G\in \Fun(\ZZ^{\textup{disc}},\MMod_{A}^{\perf})$, we have 
	$$
	(F\otimes G)(k) \simeq \bigoplus_{n+m=k} F(n)\otimes_{A}G(m)\simeq \bigoplus_{n\in \ZZ} F(n)\otimes_{A}G(k-n).
	$$
	This will become important when constructing a symmetric monoidal structure on derived $F$-zips since we have to take into account the morphisms between graded pieces  and the behaviour of graded pieces of the tensor product of bounded perfect filtrations.
\end{rem}

\subsection{Derived $\boldsymbol{F}$-zips over affine schemes}

We are ready to define derived $F$-zips, and we will do so by axiomatizing the structures occurring in Example~\ref{ex fzip de rham}. We will first restrict ourselves to the local case; \textit{i.e.}\ we define derived $F$-zips over animated rings. The reason, besides simplicity, is that the theory of derived algebraic geometry was only developed for animated rings since we want a ``nice'' model category such as the model category associated to animated rings.  This is not a real issue since globalization of the results is achieved by considering right Kan extensions. There is also a direct way of defining derived $F$-zips for derived stacks, but we will see that the two constructions agree (see Remark~\ref{Fzips global given by affine}).

Recall that we fixed an $\FF_{p}$-algebra $R$ at the beginning of this section.

\begin{defi}
\label{defi derived F-zip}
	Let $A$ be an animated $R$-algebra.
	A \textit{derived $F$-zip over $A$} is a tuple $(C^\bullet,D_\bullet,\phi,\varphi_\bullet)$ consisting of 
	\begin{itemize}
		\item a descending bounded perfect filtration of $A$-modules $C^\bullet$,
		\item an ascending bounded perfect filtration of $A$-modules $D_\bullet$, 
		\item an equivalence $\phi\colon\colim_{\ZZ^{\op}}C^{\bullet}\simeq \colim_{\ZZ}D_{\bullet},$ and
		\item a family of equivalences $\varphi_k\colon (\gr^kC)^{(1)}\xrightarrow{\lowsim} \gr^kD$.
	\end{itemize}

	The $\infty$-category of $F$-zips over $A$, denoted by $\FZip_{\infty,R}(A)$, is defined as the full subcategory of 
	\begin{align*}
		(\Fun(\ZZ^{\op},\MMod_A^{\pPerf})&\times_{\colim,\MMod_A,\colim}\Fun(\ZZ,\MMod_A^{\pPerf}))\\ &\times_{((\gr^{i}_{-})^{(1)},\gr^{i}_{-})_{i\in\ZZ},\prod_{\ZZ}\Fun(\del\Delta^{1},\MMod_{A}^{\perf})}\prod_\ZZ \Fun(\Delta^1,\MMod_{A}^{\perf})
	\end{align*}
	consisting of derived $F$-zips over $A$.

	For an animated $R$-algebra homomorphism $A\rightarrow A'$, we have an obvious base change functor $\FZip_{\infty,R}(A)\rightarrow \FZip_{\infty,R}(A')$ via the tensor product, where the filtrations are base changed componentwise with induced morphisms.
\end{defi}

\begin{rem}
	In the above definition we have to fix the equivalence between the colimits of the ascending and descending filtrations. This comes from the fact that we want to define derived $F$-zips as a full subcategory, as above. To be more specific, let us look at a pullback diagram of $\infty$-categories
	$$
	\begin{tikzcd}
		\Dcal\arrow[r,""]\arrow[d,""]& \Acal\arrow[d,"g"]\\
		\Bcal\arrow[r,"n"]&\Ccal\rlap{.}
	\end{tikzcd}
	$$
	A morphism $\Delta^{0}\rightarrow \Dcal$ is by definition the same as a diagram 
	$$
	\begin{tikzcd}
		\Delta^{0}\arrow[r,"f"]\arrow[d,"m"]\arrow[rd,"h"]& \Acal\arrow[d,"g"]\\
		\Bcal\arrow[r,"n"]&\Ccal
	\end{tikzcd}
	$$
	with equivalences $g\circ f \simeq h$ and $n\circ m\simeq h$ in the $\infty$-category $\Fun(\Delta^{0},\Ccal)$, \textit{i.e.}~$1$-morphisms in $\Hom_{\ICat}(\Delta^{0},\Ccal)\coloneqq \Fun(\Delta^{0},\Ccal)^{\simeq}$ (resp.\ $2$-morphisms in $\ICat$). Important here is that we have to fix the homotopy equivalences $g\circ f \simeq h$ and $n\circ m\simeq h$; \textit{i.e.}~they are an additional datum. So an object in $\Dcal$ is the same as a tuple $(A,B,C)\in\Acal\times\Bcal\times\Ccal$ together with equivalences $g(A)\simeq C$ and $n(B)\simeq C$. This is equivalent to giving a tuple $(A,B,\phi)$ of objects $A\in \Acal$ and $B\in\Bcal$ and an equivalence $\phi\colon g(A)\simeq n(B)$.
\end{rem}

\begin{rem}
\label{F-zips.easy}
	Let $A$ be an animated $R$-algebra. The homotopy category of derived $F$-zips over $A$ forgets the extra datum of the equivalence between the colimits. This follows from the fact that the filtrations in the definition of a derived $F$-zips over $A$ are bounded. So, any derived $F$-zip $(C^{\bullet},D_{\bullet},\phi,\varphi_{\bullet})$ is isomorphic (not canonically) in $h\FZip_{\infty}(S)$ to a derived $F$-zip, where the equivalence between the colimits is actually given by the identity. In particular, up to equivalence we may replace $\phi$ with the identity, and we will write $(C^{\bullet},D_{\bullet},\varphi_{\bullet})$ in this case for a derived $F$-zip when we work with derived $F$-zips up to homotopy.
\end{rem}

\begin{example}
\label{ex fzip de rham local}
	Let us come back to Example~\ref{ex fzip de rham}. Let $f\colon X\rightarrow \Spec(A)$ be a proper smooth morphism of schemes. Then the associated Hodge and conjugate filtrations $\HDG$ and $\conj$ define, respectively, descending, and ascending perfect bounded filtrations of $A$-modules. We also have equivalences $\varphi_{n}\colon( \gr^{n}\HDG)^{(1)}\xrightarrow{\lowsim} \gr^{n}\conj$ between the graded pieces (up to Frobenius twist), induced by the Cartier isomorphism. Therefore, we get a derived $F$-zip associated to the proper smooth map $f$ of schemes
	$$
		\underline{R\Gamma_{\dR}(X/A)}\coloneqq (\HDG^{\bullet},\conj_{\bullet},\varphi_{\bullet})
	$$
	(note that $\colim_{\ZZ^{\op}}\HDG^{\bullet}\simeq\colim_{\ZZ}\conj$ naturally by the identity).
\end{example}

\begin{remark}
  Let us note that for any $A\in \AniAlg{R}$, the $\infty$-category of $F$-zips over $A$ is essentially small, even if we do not assume $\MMod_{A}$ to be small.\footnote{We want to note that we did not assume any smallness of the module categories explicitly, and this remark shows that it is not needed in this section. But as Remark~\ref{descent qcoh} shows, we need smallness of the module categories for globalization purposes, \textit{i.e.}\ when we want to extend derived $F$-zips to derived schemes via right Kan extension.} This is because the $\infty$-category $\MMod_A$ is compactly generated (see \cite[Proposition 7.2.4.2]{HA}) (thus accessible), and therefore the full subcategory of perfect objects is essentially small (see \cite[Proposition 5.4.2.2]{HTT}). Hence for any small $\infty$-category $K$, the $\infty$-category $\Fun(K,\MMod^{\perf}_A)$ is again essentially small (see \cite[Propositions~5.3.4.13 and 5.4.4.3]{HTT}).
   Finally, since $\FZip_{\infty,R}(A)$ is a full subcategory of finite limits of the objects of the form above, we see that indeed $\FZip_{\infty,R}(A)$ is  essentially small (note that by \cite[Corollary 4.2.4.8]{HTT}, the $\infty$-category of small $\infty$-categories has small limits). 
\end{remark}

\begin{lem}
	The $\infty$-category of derived $F$-zips over an animated $R$-algebra $A$ is stable.
\end{lem}
\begin{proof}
We know that $\MMod_{A}$ and $\MMod^{\perf}_{A}$ are stable and thus also that for any $\infty$-category $\Ccal$, the $\infty$-category $\Fun(C,\MMod_{A}^{\perf})$ is stable. Since the limit of stable $\infty$-categories with finite limit-preserving transition maps is stable, it is enough to show that $\FZip_{\infty,R}(A)$ is a stable subcategory of 
\begin{align*}
		\left( \Fun\left(\ZZ^{\op},\MMod_A^{\pPerf}\right)\right.&\left. \times_{\colim,\MMod_A,\colim}\Fun\left(\ZZ,\MMod_A^{\pPerf}\right) \right) \\ &\times_{((\gr^{i}_{-})^{(1)},\gr^{i}_{-})_{i\in\ZZ},\prod_{\ZZ}\Fun\left(\del\Delta^{1},\MMod_{A}^{\perf}\right)}\prod_\ZZ \Fun\left(\Delta^1,\MMod_{A}^{\perf}\right)
	\end{align*}
	(note that $\prod_\ZZ \Fun(\Delta^1,\MMod_{A}^{\perf})\simeq \Fun(\ZZ^{\disc},\Fun(\Delta^{1},\MMod_{A}^{\perf}))$ and filtered colimits preserve finite limits).
	For this, we have to show that the perfect bounded filtrations, equivalences between colimits of filtrations and equivalences between the graded pieces (up to Frobenius twist) are stable under shifts and cofibers.
 
	 That perfect bounded filtrations are stable under shift and cofibers follows immediately from the fact that limits and colimits of functors can be computed pointwise (see \cite[Corollary 5.1.2.3]{HTT}). The same argument implies that equivalences between the graded pieces (up to Frobenius twist) are stable under cofibers and shifts. Since filtered colimits commute with shifts and cofibers, we also see that the equivalence between the colimits is preserved under those operations.
\end{proof}

In the following, we want to construct a symmetric monoidal structure on derived $F$-zips. The idea is very simple. We know that derived $F$-zips are contained in a larger $\infty$-category (see Definition~\ref{defi derived F-zip}); let us denote this category with $\Ccal$. This $\infty$-category $\Ccal$ is constructed by limits of functor categories, that we can endow with the Day convolution. For the morphisms between graded pieces, we have to be a bit careful, but Remarks~\ref{graded of tensor} and~\ref{day on gr} show us that this will not be a problem. Since passing to the graded pieces and taking the colimit of a filtration are both monoidal functors, we see that indeed $\Ccal$ is symmetric monoidal. Now we only need to show that the unit object of $\Ccal$ is a derived $F$-zip, which follows immediately, that the Day convolution of perfect bounded filtrations is bounded perfect (this is Proposition~\ref{Day descend}) and that the induced morphism of graded pieces (up to Frobenius twist) is an equivalence, which is also immediate.

\begin{prop}
\label{fzip monoidal}
	The $\infty$-category of derived $F$-zips over an animated $R$-algebra $A$ admits a symmetric monoidal structure.
\end{prop}
\begin{proof}
	We know that $\MMod_A^{\perf}$ admits a symmetric monoidal structure (see \cite[Remark~2.2.1.2]{HA} and note that as an $A$-module, $A$ is perfect and the tensor product of perfect $A$-modules is again perfect). We now show how to construct a symmetric monoidal structure on derived $F$-zips.
  
	The monoidal structure on the filtrations is  given by the Day convolution (see Proposition~\ref{Day descend}). The monoidal structure on the equivalences of graded pieces is given in the following way.

	We endow $\Fun(\ZZ^{\disc},\MMod_A^{\perf})$ with the Day convolution (as explained in Remark~\ref{day on gr}), where we endow $\ZZ^{\disc}$ with a symmetric monoidal structure by usual addition. The unit object in $\Fun(\ZZ^{\disc},\MMod_A^{\perf})$ is given by $A_{\triv}^{\disc}$, where $A_{\triv}^{\disc}(n) \simeq A$ if $n=0$ and $0$ otherwise.
        Now we endow the $\infty$-category $\Fun(\Delta^1,\Fun(\ZZ^{\disc},\MMod_A^{\perf}))$ with the pointwise tensor product (see \cite[Remark 2.1.3.4]{HA});  we do exactly the same for $\Fun(\del\Delta^1,\Fun(\ZZ^{\disc},\MMod_A^{\perf})).$ Certainly, by this construction, for a derived $F$-zip $(C^\bullet,D_\bullet,\varphi_\bullet)$ over $A$, the family $\varphi_\bullet$ defines an element of $\Fun(\Delta^1,\Fun(\ZZ^{\disc},\MMod_A^{\perf}))$.

	Now let us note that taking the colimit defines a symmetric monoidal functor from ascending (resp.\ descending) filtrations to $\MMod_{A}$ (as the tensor product of spectra commute with colimits in each variable, see \cite[Corollary 4.8.2.19]{HA}). Also, sending a filtration to its graded piece is symmetric monoidal by Remark~\ref{graded of tensor}. Therefore, we can attach a symmetric monoidal structure to 
	\begin{align*}
		\Ccal(A)\coloneqq \left(\Fun\left(\ZZ^{\op},\MMod_A^{\pPerf}\right)\right.&\times_{\colim,\MMod_A,\colim}\Fun\left.\left(\ZZ,\MMod_A^{\pPerf}\right)\right)\\ &\times_{\Fun\left(\del\Delta^1,\Fun\left(\ZZ^{\disc},\MMod_A^{\perf}\right)\right)}\Fun\left(\Delta^1,\Fun\left(\ZZ^{\disc},\MMod_A^{\perf}\right)\right), 
	\end{align*}
	where we use that the $\infty$-category of symmetric monoidal $\infty$-categories has limits (see\footnote{The reference shows the existence of limits in commutative algebra objects of symmetric monoidal $\infty$-categories. But using \cite[Proposition 4.1.7.10]{HA}, we can endow the $\infty$-category of $\infty$-categories with the Cartesian model structure (a concrete description of the associated $\infty$-operad is given in \cite[Notation 4.8.1.2]{HA}). The commutative algebra objects of the $\infty$-category of $\infty$-categories with this monoidal structure is then equivalent to the $\infty$-category of symmetric monoidal $\infty$-categories (defined for example in \cite[Variant 2.1.4.13]{HA}).}
        \cite[Remark 2.4.2.6, Proposition 3.2.2.1]{HA}).
 
Since derived $F$-zips over $A$ form a full subcategory of $\Ccal(A)$, it suffices to check that the unit element of $\Ccal(A)$ is in $\FZip_{\infty,R}(A)$ and that it is closed under the tensor product (see \cite[Remark~2.2.1.2]{HA}).  But this follows from Proposition~\ref{Day descend} and Remarks~\ref{graded of tensor} and~\ref{day on gr}. Concretely, the unit element in $\FZip_{\infty,R}(A)$ is given by 
	$$
	\underline{\onebb_A}\coloneqq (A^\bullet_{\triv},A_\bullet^{\triv},\id_{A},(\id_A)_0),
	$$
	where $A_\bullet^{\triv}$ is defined as in Remark~\ref{Day convolution}, $A_{\triv}^\bullet$ is defined dually, \textit{i.e.}~is given by $A_{\triv}^n\simeq A$ for $n\leq 0$  and $0$ elsewhere and has identity as transition maps, $(\id_A)_0$ denotes the family of morphisms $\varphi_\bullet$, where $\varphi_0\simeq \id_A$ and $\varphi_n=0$ elsewhere.
        
	Note that for $\varphi_{\bullet},\vartheta_{\bullet}\in\Fun(\Delta^1,\Fun(\ZZ^{\disc},\MMod_A^{\perf}))$ which induce equivalences in $\Fun(\ZZ^{\disc},\MMod_A^{\perf})$, their tensor product is still an equivalence, by the explicit description given in Remark~\ref{day on gr}.
\end{proof}

Our next goal is to show that the functor that sends an animated $R$-algebra to the $\infty$-category of derived $F$-zips over it is locally geometric. For this we need that it is a hypercomplete sheaf for the \'etale topology. We will show that it is a hypercomplete sheaf even for the fpqc topology.  Since every geometric derived stack is hypercomplete (see Lemma~\ref{geometric hypercomplete}), the hypercompleteness condition is necessary, at least for the \'etale topology.
 
Again the idea is very simple and follows the proof of descent for perfect modules seen in \cite[proof of Lemma 5.4]{AG}. We again embed $F$-zips into a larger category as in the proof of Proposition~\ref{fzip monoidal}, which satisfies hyperdescent. Then we only need to check that the property \textit{bounded perfect} of a filtration and the property \textit{equivalence} of a morphism between modules satisfy fpqc hyperdescent. But since our cover is affine and perfectness is equivalent to dualizability, both properties satisfy hyperdescent, and we are done.
 
To see that the larger category satisfies descent, one only needs that perfect filtrations satisfy descent, which will follow the from descent of perfect modules.

\begin{lem}
	\label{Fun sheaf}
	Let $F\colon \AniAlg{R}\rightarrow \ICat$ be a hypercomplete  fpqc sheaf. Then for any $\infty$-category $\Ccal$, the functor $\Fun(\Ccal, F(-))\colon  \AniAlg{R}\rightarrow \ICat$ is a hypercomplete fpqc sheaf.
\end{lem}
\begin{proof}
	As $\Fun(\Ccal,-)$ is right adjoint to the product, it preserves limits. Since $F$ is a hypercomplete sheaf, we see that indeed the natural morphism $\Fun(\Ccal, F(A))\xrightarrow{\lowsim}\lim_{\Delta_{s}} \Fun(\Ccal, F(A^\bullet))$ is an equivalence for any fpqc hypercovering $A\rightarrow A^\bullet$.
\end{proof}

\begin{defprop}
	\label{fzip infty sheaf}
	The functor 
	\begin{align*}
	\FZip_{\infty,R}\colon \AniAlg{R}&\longrightarrow \ICat\\
	A &\longmapsto \FZip_{\infty,R}(A)
	\end{align*}
	is a hypercomplete sheaf for the fpqc topology.
\end{defprop}
\begin{proof}
	In the following, we will denote the functor $\AniAlg{R}\rightarrow \ICat,\ A\mapsto \MMod^{\perf}_{A}$ by $\MMod^{\perf}_{(-)}$ to avoid confusion with the notation of the stack of perfect modules over $R$.

	Let $A\rightarrow A^{\bullet}$ be an fpqc hypercovering given by a functor $\Delta_{+,s}\rightarrow \AniAlg{R}$. We have to show that $\FZip_\infty (A) \rightarrow \lim_{\Delta_{s}}  \FZip_{\infty,R}(A^{\bullet})$ is an equivalence.
 
	For convenience, we first set
	\begin{align*}
		\Ccal(-)\coloneqq \Fun\left(\ZZ^{\op},\MMod_{-}^{\pPerf}\right)&\times_{\colim,\MMod_{-},\colim}\Fun\left(\ZZ,\MMod_{-}^{\pPerf}\right)\\ &\times_{\Fun\left(\del\Delta^1,\Fun\left(\ZZ^{\disc},\MMod_{-}^{\perf}\right)\right)}\Fun\left(\Delta^1,\Fun\left(\ZZ^{\disc},\MMod_{-}^{\perf}\right)\right). 
	\end{align*}

	For the fully faithfulness, let us look at the  diagram
	$$
	\begin{tikzcd}
	\FZip_{\infty,R}(A)\arrow[r,""]\arrow[d,"",hook]&\lim_{\Delta_{s}}  \FZip_{\infty,R}(A^{\bullet})\arrow[d,"",hook]\\
	\Ccal(A)\arrow[r,""]&\lim_{\Delta_{s}} \Ccal(A^{\bullet})\rlap{.}
	\end{tikzcd}
	$$ 
	By Lemma~\ref{Fun sheaf} and descent of (perfect) modules (see Remark~\ref{descent qcoh}), we see that $\Ccal$ is a hypercomplete sheaf for the fpqc topology, and thus the bottom horizontal arrow is an equivalence, and thus the upper horizontal arrow is fully faithful.
 
	For the essential surjectivity, note that we have to check that a filtration is bounded if and only if it is fpqc hyperlocally. But this follows immediately from the definition of a hypercovering since $A\rightarrow A_{0}$ has to be an fpqc covering, and thus if a filtration is bounded on the hypercovering, it  certainly is on $A$. Also, we have to check that the induced morphism on the graded pieces (up to Frobenius twist) is an equivalence if and only if it is so fpqc hyperlocally, but again this follows from the descent of modules.
\end{proof}

Recall that the functor $(-)^{\simeq}$ that sends an $\infty$-category to its underlying Kan complex is right adjoint to the inclusion and thus preserves limits. In particular, the hypercomplete sheaf $\FZip_{\infty,R}$ induces a derived stack. We want to show that this stack is locally geometric. To do so, we have to write it as a filtered colimit of geometric stacks. In the case of perfect modules, we restricted ourselves to perfect modules of fixed Tor-amplitude $\PPerf_{R}^{[a,b]}$. To use the geometricity of $\PPerf_{R}^{[a,b]}$ to our advantage, we fix the Tor-amplitude of the graded pieces associated to the descending filtration of a derived $F$-zip $(C^{\bullet},D_{\bullet},\phi,\varphi_{\bullet})$. By the boundedness of the filtrations, this also fixes the Tor-amplitude of each $C^{i}$ and $D_{i}$ for all $i\in\ZZ$ (for $D_{i}$ we use the equivalences given by $\varphi_{\bullet}$). But this is still not enough for geometricity since we would need to cover filtrations that could get bigger and bigger. To solve this problem, we also fix a finite subset $S\subseteq\ZZ$, where the $\supth{i}$ graded piece vanishes for $i\not\in S$. This is analogous to fixing the \textit{type} (see Definition~\ref{local type}), which is done in the classical setting  by \cite{moon-wed}. This approach also works, as seen later in Remark~\ref{F-zip as colim of open im}, but amounts to the same proof.

\begin{defi}
\label{defi derived fzip}
	We define the \textit{derived stack of\, $F$-zips}
	\begin{align*}
	\FZip_{R}\colon  \AniAlg{R}&\longrightarrow \SS\\
	A&\longmapsto \FZip_{\infty,R}(A)^\simeq.
	\end{align*}

	For a finite subset $S\subseteq \ZZ$ and $a\leq b\in\ZZ$, we define $\FZip_{R}^{[a,b],S}$ as the derived substack over $\FZip_{R}$, where we restrict ourselves to the $F$-zips $(C^\bullet,D_\bullet,\phi,\varphi_\bullet)$ such that $\gr^{i}C\simeq 0$ for  $i\not\in S$ and the Tor-amplitude of $\gr^i C$ is contained in $[a,b]$ for all $i\in S$ (note that both conditions can be tested locally, and thus $\FZip_{R}^{[a,b],S}$ indeed defines a derived substack).
\end{defi}

Let us quickly insert a remark that is needed for the proof of the next theorem.

\begin{rem}
\label{kan fib pull}
	Let us look at the following diagram in $\SS$: 
	$$
	\begin{tikzcd}
		Z\arrow[r,""]\arrow[d,""]& X\arrow[d,"a"]\\
		Y\arrow[r,"b"]&W\rlap{.}
	\end{tikzcd}
	$$
	Let us assume this diagram is a pullback diagram in $\SS$. If one of the morphisms $a$ or $b$ is a Kan fibration, then $Z$ is equivalent in $\SS$ to the ordinary pullback in the category of simplicial sets. This is a standard result for homotopy pullbacks in model categories (see \cite[Remark~A.2.4.5]{HTT}). 
        
	If we replace $\SS$ with $\ICat$, where we also replace ``Kan fibration'' with ``isofibration'', the same holds after applying $(-)^{\simeq}\colon \ICat\rightarrow \SS$ to the diagram. A more detailed overview of this can also be found in \cite[Remark~4.13]{NotesDAG}.
 
	An important example of a Kan fibration (resp.\ an isofibration) is the naturally induced morphism $\Fun(B,X)\rightarrow \Fun(A,X)$ for a simplicial set $B$, a simplicial subset $A\subseteq B$ and a Kan complex (resp.\ an $\infty$-category) $X$ (see \cite[00TJ, 01BS]{kerodon}; again, a more details can also be found in \cite[Remark~4.13]{NotesDAG}).
\end{rem}

\begin{thm}
	\label{F-Zips locally geometric}
	Let $S\subseteq \ZZ$ be a finite subset and $a\leq b\in\ZZ$. The derived stack $\FZip_{R}^{[a,b],S}$ is $(b-a+1)$-geometric and locally of finite presentation. Further, the functor $\FZip^{[a,b],S}_{R}\rightarrow \PPerf^{\,[a,b]}_{R}$ induced by $(C^\bullet,D_\bullet,\phi,\varphi_\bullet)\mapsto \colim_{\ZZ^{\op}}C^{\bullet}$ is locally of finite presentation.
\end{thm}
We want to note that in fact $\FZip$ is locally geometric, which will be shown later on (see Theorem~\ref{fzip local geometric plus}).

The idea of the proof is straightforward. We know that perfect modules with fixed Tor-amplitude, morphisms between those and stacks classifying equivalences between those are geometric. Since the filtrations have finite length, we can see them as finite chains between perfect modules with fixed Tor-amplitude, which are also geometric. The only thing left is to extend finite chains of perfect modules to functors from $\ZZ$ (resp.\ $\ZZ^{\op}$) to perfect modules with fixed degree where the graded pieces do not vanish. But this is also straightforward since the only thing left is to degenerate each vertex in the finite chain such that it sits in the right degree.

\begin{proof}[Proof of Theorem~\ref{F-Zips locally geometric}]
	Let $k$ be the number of elements of $S$, and let us index $S$ in the following way $\lbrace s_0<\dots<s_{k-1}\rbrace$. Let us set $n\coloneqq b-a+1$.
	Consider the pullback square 
	$$
	\begin{tikzcd}
	V\arrow[r,""]\arrow[d,"p"]&\Fun\left(\Delta^1,\MMod^{\perf}_{(-)}\right)^{\simeq}\arrow[d,""]\\
	\PPerf^{\,[a,b]}_{R}\times_R\PPerf^{\,[a,b]}_{R}\arrow[r,""]&\Fun\left(\del\Delta^1,\MMod^{\perf}_{(-)}\right)^{\simeq}\rlap{.}
	\end{tikzcd}
	$$
	Then $V$ is an $n$-geometric stack locally of finite presentation, since it classifies morphisms between perfect complexes with Tor-amplitude in $[a,b]$. Thus, the fiber under $p$ of a point $\Spec(A)\rightarrow  \PPerf^{\,[a,b]}_{R}\times_R\PPerf^{\,[a,b]}_{R}$, classified by perfect $A$-modules $(P,Q)$ of Tor-amplitude in $[a,b]$, is given by $F_{P\otimes Q^\vee}$. This is an $(n-2)$-geometric derived stack and is locally of finite presentation by Lemma~\ref{spec sym geometric}. Hence, using that the product $\PPerf^{\,[a,b]}_{R}\times_R\PPerf^{\,[a,b]}_{R}$ is $n$-geometric (since $\PPerf^{\,[a,b]}_{R}$ is $n$-geometric by Theorem~\ref{main thm}), we see that $V$ is $n$-geometric.
 
	We will now glue copies of $V$ together so that we can classify a chain of morphisms, and since $F$-zips have two filtrations, we will do this twice. Let us start with $\codom \colon V\rightarrow \PPerf^{\,[a,b]}_{R}$, which sends a morphism to its codomain. This morphism is $n$-geometric and locally of finite presentation since it is the composition of $V\rightarrow \PPerf^{\,[a,b]}_{R}\times_R \PPerf^{\,[a,b]}$ and $p_2\colon\PPerf^{\,[a,b]}_{R}\times_R \PPerf^{\,[a,b]}_{R}\rightarrow \PPerf^{\,[a,b]}_{R}$, which both are $n$-geometric and locally of finite presentation (analogously, the map $\dom\colon V\rightarrow \PPerf^{\,[a,b]}$ which sends a morphism to its domain is $n$-geometric and locally of finite presentation). In particular, $\widetilde{V_{1}}\coloneqq V\times_{\codom, \PPerf^{\,[a,b]}_{R},\codom}V$ is $n$-geometric and locally of finite presentation.

	The derived stack $\widetilde{V_{1}}$ classifies tuples of morphisms of perfect modules $(M\rightarrow M',N\rightarrow N')$ such that $M'$ is equivalent to $N'$. This is not an extra datum, as $\codom$ from $V$ to $\PPerf_{R}^{[a,b]}$ is pointwise a Kan fibration, so we can use Remark~\ref{kan fib pull} to see that $\widetilde{V_{1}}$ is pointwise equivalent in $\SS$ to the ordinary pullback of simplicial sets, where we do not need to keep track of the equivalence of $M'$ and $N'$.
	Since we want to keep track of the equivalence between the colimits, we define $V_{1}$ via the pullback square
	$$
	\begin{tikzcd}
		V_{1}\arrow[rrr,""]\arrow[d,"\widetilde{p}"]&&& \Fun\left(\Delta^{1},\PPerf^{\,[a,b]}_{R}\right)\arrow[d,""]\\
		\widetilde{V_{1}}\arrow[rrr,"{(\codom,\codom)}"]&&&\Fun\left(\del\Delta^{1},\PPerf^{\,[a,b]}_{R}\right)\rlap{.}
	\end{tikzcd}
	$$
	Note that for any morphism $\Spec(A)\rightarrow\widetilde{V_{1}}$, classified by two morphisms $(M\rightarrow M',N\rightarrow N')$ of perfect $A$-modules of Tor-amplitude $[a,b]$, the fiber under $\widetilde{p}$ is given by the stack classifying equivalences between $M'$ and $N'$, which is open in the derived stack $\Hom_{\MMod_{A}}(M'\otimes_{A} (N')^{\vee},-)$ over $A$ by Lemma~\ref{equiv zariski open} and thus $(n-2)$-geometric and locally of finite presentation by Lemma~\ref{spec sym geometric}. By our construction, $V_{1}$ is $n$-geometric and classifies tuples $(f,g,\psi)$, where $f,g$ are morphisms of perfect modules and $\psi$ is an equivalence between their codomains.

        Let us set recursively
        $$
        V_{l}\coloneqq (V\times_R V) \times_{\codom\times\codom,\PPerf^{\,[a,b]}_{R}\times_R \PPerf^{\,[a,b]}_{R},\dom\times \dom} V_{l-1}
        $$
        for $l\geq 2$. Let us also define $V_0$ as the stack classifying equivalences between perfect modules (analogously defined as $V$). Here $\dom\times \dom\colon V_{l-1}\rightarrow \PPerf_{R}^{\,[a,b]}\times_R \PPerf_{R}^{\,[a,b]} $ is defined for $l>2$ by projecting to $V\times_R V$ and then further projecting by $\dom\times\dom$, and for $l=2$, it is directly given by $\dom\times\dom$ (note that we have two projections from $V_{1}$ to $V$). Using the same arguments as before, we see that $V_l$ is $n$-geometric and locally of finite presentation for all $l\geq 0$. Now let us look at $V_{k-1}$. The stack $V_{k-1}$ classifies two chains of length $k-1$ of morphisms of perfect modules with Tor-amplitude in $[a,b]$, with an equivalence of the ends of the two chains.

	Later in the proof, we will identify $V_{k-1}$ with the stack that classifies perfect modules with Tor-amplitude in $[a,b]$ and two bounded filtrations  with graded pieces of Tor-amplitude in $[a,b]$ (one descending, one ascending) with $k$ non-trivial graded pieces. 

	We extend the chains by zero on the left by defining
        $$
        \widetilde{V_{k-1}}\coloneqq V_{k-1}\times_{\dom\times\dom, \PPerf_{R}^{\,[a,b]}\times_R\PPerf_{R}^{\,[a,b]},p_0\times p_0} W\times_R W,
        $$
	where $W$ is defined via the pullback
	$$
	\begin{tikzcd}
		W\arrow[rr,""]\arrow[d,"p_0"]&&\Fun\left(\Delta^1,\MMod^{\perf}_{(-)}\right)^{\simeq}\arrow[d,""]\\
		\PPerf^{\,[a,b]}_{R}\arrow[rr,"{M\mapsto (0, M)}"]&&\Fun\left(\del\Delta^1,\MMod^{\perf}_{(-)}\right)^{\simeq}\rlap{.}
	\end{tikzcd}
	$$	
	Since $0$ is the initial object in perfect modules, we see that $W\simeq \PPerf_{R}^{\,[a,b]}$ and thus $\widetilde{V_{k-1}}\simeq V_{k-1}$, but this description will ease the connection to derived $F$-zips (we have to add zeroes to get left-bounded filtrations, and by working with $\widetilde{V_{k-1}}$, this is automatic if we degenerate the leftmost vertex). Note that we can identify an element in $\widetilde{V_{k-1}}(A)$ with two functors 
	$$
	C_{\bullet},D_{\bullet}\colon S_{-}\coloneqq \lbrace s_{0}-1< s_0<\dots< s_{k-1}\rbrace\longrightarrow \MMod^{\perf}_{A}.
	$$
	Sending $C_{\bullet}$ to the functor 
	\begin{align*}
		C^{\bullet}\colon S_{+}^{\op}\coloneqq\lbrace s_0<\dots< s_{k-1}< s_{k-1}+1\rbrace^{\op}&\longrightarrow \MMod^{\perf}_{A}\\
		s_{i}&\longmapsto C_{s_{k-1-i}}\\
		s_{k-1}+1&\longmapsto C_{s_{0}-1}
	\end{align*}
	defines an obvious equivalence between the corresponding functor categories, and so we may see an element $\widetilde{V_{k-1}}(A)$ as a tuple $(C^{\bullet},D_{\bullet},\psi)$ of a finite descending chain $C^{\bullet}$ of perfect $A$-modules, a finite ascending chain $D_{\bullet}$ of perfect $A$-modules and an equivalence of their respective colimits, \textit{i.e.}\ 
	$$
		\psi\colon\colim_{S_{+}^{\op}}C^{\bullet}\simeq \colim_{S_{-}}D_{\bullet}.
	$$

	With this identification, we get a morphism from $\widetilde{V_{k-1}}$ to the $k$-fold product of $\Fun(\del\Delta^1,\PPerf)$, by sending a pair of chains to the graded pieces of the chains (resp.\ if we see them as filtrations, we send them to the non-trivial graded pieces); \textit{i.e.}\ if the filtrations of an element in $\widetilde{V_{k-1}}(A)$, for some $A\in\AniAlg{R}$, are given by $(C^\bullet,D_\bullet,\psi)$, we take $((\gr^{s}C)^{(1)},\gr^{s}D)_{s\in S}$ (here  $\gr^{s}C$ is defined as the cofiber of $C^{s_{i+1}}\rightarrow C^{s_{i}}$, where $s_{i}$ corresponds to $s$ in the notation above, and analogously for $\gr^{s}D$). With this, let us look at the  pullback square
	$$
	\begin{tikzcd}
	\widetilde{V}\arrow[rrrr,""]\arrow[d,"p"]&&&&\prod_{s\in S}\Fun(\Delta^1,\PPerf_{R})\arrow[d,""]\\
	\widetilde{V_{k-1}}\arrow[rrrr,"{((\gr^{s}(-))^{(1)},\gr^{s}(-))_{s\in S}}"]&&&&\prod_{s\in S}\Fun(\del\Delta^1,\PPerf_{R})
	\end{tikzcd}
	$$
	(note the difference with the previous squares: previously, we considered morphisms in $\MMod_{-}^{\perf}$, whereas here we consider morphisms in the underlying Kan complex, so only invertible ones).

	Now let $\Spec(A)\rightarrow \widetilde{V_{k}}$ be a morphism classified by a perfect $A$-module with Tor-amplitude in $[a,b]$ and a descending (resp.\ ascending) chain $C^\bullet$ (resp.\ $D_\bullet$). Then $p^{-1}(\Spec(A))$ is given by the stack $$\prod_{s\in S}\Equiv((\gr^{s}C)^{(1)},\gr^{s}D),$$ where $\Equiv((\gr^{s}C)^{(1)},\gr^{s}D)$ denotes the stack classifying equivalences between $ (\gr^{s}C)^{(1)}$ and $\gr^{s}D$. This stack is $n$-geometric and locally of finite presentation (since each term is $n$-geometric and locally of finite presentation)\footnote{Again, $\Equiv((\gr^{s}C)^{(1)},\gr^{s}D)$ is open in the derived stack $\Hom_{\MMod_{A}}((\gr^{s}C)^{(1)}\otimes(\gr^{s}D)^{\vee},-)$ over $A$ by Lemma~\ref{equiv zariski open}, and $\Hom_{\MMod_{A}}((\gr^{s}C)^{(1)}\otimes(\gr^{s}D)^{\vee},-)\rightarrow\Spec(A)$ is $(n-2)$-geometric by Lemma~\ref{spec sym geometric}.}, and therefore we see that $\widetilde{V}$ is $n$-geometric and locally of finite presentation. 
	By construction, we also have that the morphism 
	$$
	\widetilde{V}\longrightarrow \PPerf^{\,[a,b]}_{R}
	$$
        given by sending one of the chains to its colimit is $n$-geometric and locally of finite presentation (as it is given by projections down to $\PPerf^{\,[a,b]}_{R}$, which are all $n$-geometric and locally of finite presentation).

	We can naturally define a functor $F\colon \widetilde{V}\rightarrow \FZip^{[a,b],S}_{R}$ by extending the filtrations via the identity to get non-trivial graded pieces at the points of $S$ and the isomorphisms of the graded pieces by the essentially unique zero morphism.
 
	To be more specific, we will show how to give $F$ as a functor of simplicial sets. The extension by the identity will just be degeneration of simplicial sets. Let $\sigma \in\widetilde{V}(A)$ be an $m$-simplex. Then $\sigma$ is given by functors 
	\begin{align*}
		D\colon S_{-}\times\Delta^m\longrightarrow \MMod^{\perf}_{A},\quad 
		C\colon S_{+}^{\op}\times\Delta^m\longrightarrow \MMod^{\perf}_{A}
	\end{align*}
	with 
	$$\psi\colon C_{|\lbrace s_{0}\rbrace\times \Delta^m} \simeq D_{|\lbrace s_{k-1}\rbrace\times \Delta^m},\quad C_{|\lbrace s_{k-1}+1\rbrace\times \Delta^m} \simeq 0\simeq D_{|\lbrace s_{0}-1\rbrace\times \Delta^m}
	$$ 
	and $(\varphi_{s})_{s\in S}$, where $\varphi_s\colon \Delta^1\times\Delta^m\rightarrow \PPerf_{R}(A)$ is a functor between simplicial sets, such that 
	$$
	\varphi_{s|\Delta^{\lbrace 0\rbrace} \times \Delta^{\lbrace i\rbrace}} \simeq (\gr^s C)^{(1)}\quad \textup{and}\quad \varphi_{s|\Delta^{\lbrace 1\rbrace} \times \Delta^{\lbrace i\rbrace}} \simeq \gr^s D.
	$$
	 We will show how to define an extension of $D$, \textit{i.e.}\ a functor $\Dbar\colon \ZZ\times\Delta^m\rightarrow \MMod^{\perf}_{A}$, that restricted to $S_{-}$ is equivalent to $D$.

	   Let $(\alpha,\sigma)$ be an $l$-simplex in $\ZZ\times \Delta^m$; note that $\alpha$ is given by a sequence of integers $\alpha_0\leq\dots\leq \alpha_{l-1}$. We are going to count the number of $\alpha_{i}$ that are between two vertices in $S_{-}$ and degenerate our finite chain for this amount. This can be thought of as adding the right amount of identities between vertices so that the resulting filtration has non-vanishing graded pieces precisely at all $s\in S$.

	Let us denote the degeneracy maps of $S_{-}$ by $\lambda_{\bullet}$ (here we see the finite ordered set $S_{-}$ as an $\infty$-category and set $\lambda_{-\infty}$ as the degeneracy map corresponding to $s_{0}-1$). For $s_d\in S_{-}$, define sets $A_d\coloneqq \lbrace\alpha_j\mid  s_{d}\leq \alpha_j \leq s_{d+1}\rbrace$, where we set  $s_{k}=\infty$ and $A_{-\infty}\coloneqq \lbrace \alpha_{j}\mid \alpha_{j}\leq s_{0}-1\rbrace$. Let us further set $n_i=\#A_i$. We consider the $l$-simplex in $S_{-}$ of the form 
	$$
	\agbar\coloneqq \lambda^{n_{k-1}-1}_{k-1}\circ \dots\circ \lambda^{n_{-\infty}-1}_{-\infty}\langle s_j\mid A_j\not = \emptyset\rangle,
	$$ where $\lambda_i^{-1}\coloneqq\id$. Thus, we can set $\Dbar(\alpha,\sigma)\coloneqq D(\agbar,\sigma)$. Similarly, this can be done to extend $C$, $\psi$ and $(\varphi_s)_s$, which defines the functor $F$.
 
	We also get a projection $P\colon \FZip^{[a,b],S}_{R}\rightarrow \widetilde{V}$ via restricting the filtrations to $\lbrace s_0-1\leq s_0\leq\dots\leq s_{k-1}\rbrace\simeq S_{-}$, resp.\ $\lbrace  s_0\leq\dots\leq s_{k-1}\leq s_{k-1}+1\rbrace\simeq S_{+}$, and the equivalences to $\Fun(\Delta^{1},\Fun(S^{\disc},\PPerf_{R}(A)))$. Since by this construction $P\circ F$ is the identity, we see that $F$ is in fact a monomorphism, and since it is an effective epimorphism (by Remark~\ref{F-zips.easy}, $F$ is pointwise essentially surjective on the level of homotopy categories), this shows that $F$ is an equivalence of derived stacks.  
\end{proof}

\begin{cor}
	Let $S\subseteq \ZZ$ be a finite subset and $a\leq b\in\ZZ$. The derived stack $\FZip^{[a,b],S}_{R}$ has a perfect cotangent complex .
\end{cor}
\begin{proof}
	This follows from Theorem~\ref{F-Zips locally geometric} and Corollary~\ref{global lfp cotangent}. 
 
\end{proof}

\subsection{On some substacks of $\FZip$}
\label{sec substacks}

In this section, we want to define the \textit{type} of a derived $F$-zip and look at the derived substacks classified by the type. We do the same with those derived $F$-zips where the underlying module has some fixed Euler characteristic. These derived substacks will be open (resp.\ locally closed), and we will use these to write the derived stack of derived $F$-zips as a filtered colimit of open derived substacks.

\medskip
We do this in the spirit of the classical theory of $F$-zips over a scheme $S$. There, the type of a classical $F$-zip $(M,C^{\bullet},D_{\bullet},\varphi_{\bullet})$ is a function from $S$ to functions $\ZZ\rightarrow \NN_{0}$ with finite support that assigns to a point $s\in S$ the function $k\mapsto \dim_{\kappa(s)}(\gr^{k}_{C}M\otimes_{\Ocal_{S}}\kappa(s))$ (one uses that the graded pieces of a classical $F$-zip are finite projective and thus the dimension on the fibers is locally constant with respect to $s$).
 
Since we are working with complexes, we have to modify the definition of the type. To be more specific, we will look at fiberwise dimensions of all cohomologies at once. For a perfect complex $P$ over $S$, the assignment $s\in S\mapsto H^{i}(P\otimes_{\Ocal_{S}}\kappa(s))$ defines an upper semi-continuous function. Since derived $F$-zips have only finitely many non-zero graded pieces, we will use this result to analyze the geometry of the derived substacks classifying derived $F$-zips with certain type.

\begin{defrem}
	\label{local rank}
	Let $A$ be an animated ring, and let $P$ be a perfect $A$-module. We define the function
	\begin{align*}
	\beta_P\colon \Spec(\pi_0A)_{\cl}&\longrightarrow \NN_0^{\ZZ}\\
	a&\longmapsto (\dim_{\kappa(a)}\pi_i(P\otimes_A \kappa(a)))_{i\in\ZZ}.
	\end{align*}
 
	Since $P\otimes_A\pi_0A$ is perfect and thus has bounded Tor-amplitude, we see that $\beta_{P}$ takes values in functions $\ZZ\rightarrow \NN_{0}$ with finite support  and $\beta_P^{-1}(([0,k_i])_{i\in\ZZ})$ is open and quasi-compact for any $(k_i)_i\in\NN_0^{\ZZ}$ (see 
Lemma~\ref{upper semi-cont qc}).
 
	Let $I\subseteq \ZZ$ be a finite subset. Assume that $\textup{Supp}(\beta_P(a))\subseteq I$ for all $a\in\Spec(\pi_{0}A)_{\cl}$. Recall from Remark~\ref{tor and type} that this implies that $P$ has Tor-amplitude in $[\min(I),\max(I)]$. Further, as explained in the proof of Lemma~\ref{lem perfect zero}, Nakayama's lemma  implies that if $\beta_{P}$ constant with value equal to $(0)_{i\in\ZZ}$, then $P\simeq 0$.  
\end{defrem}

The following definition will be used to ease notation.

\begin{defi}
	Let $f\colon \ZZ\rightarrow\NN_{0}^{\ZZ}$ be a function. We say that $f$ has \textit{finite support} if the induced function $\ZZ\times\ZZ\rightarrow \NN_{0}$ given by $(n,m)\mapsto f(n)_{m}$ has finite support.
\end{defi}

\begin{rem}
	A function $f\colon\ZZ\rightarrow \NN_{0}^{\ZZ}$ has finite support if the sets
	$$
	\lbrace n\in\ZZ\mid f(n)\neq (0)_{i\in\ZZ}\rbrace,\quad \lbrace k\in\ZZ\mid f(n)_{k}\neq 0\rbrace
	$$
	are finite.
\end{rem}

\begin{defrem}
	\label{local type}
	Let $A$ be an animated $R$-algebra and $\Fline\coloneqq (C^\bullet,D_\bullet,\phi,\varphi_\bullet)$ be a derived $F$-zip over $A$. Consider the function 
	\begin{align*}
		\beta_{\Fline}\colon a\longmapsto (k\longmapsto \beta_{\gr^kC}(a))
	\end{align*}
	from $\Spec(\pi_{0}A)_{\cl}$ to functions with finite support.
	\begin{enumerate}
		\item The function $\beta_{\Fline}$ is called the \textit{type of the derived $F$-zip $\Fline$}.
		\item Let $\tau\colon \ZZ\rightarrow \NN_0^\ZZ$ be a function with finite support. We say that \textit{$\Fline$ has type at most $\tau$} if for all $a\in\Spec(\pi_{0}A)_{\cl}$, we have $\beta_{\Fline}(a)\leq \tau$.\footnote{Again, we view $\beta_{P}(a)$ and $\tau$ as functions from $\ZZ\times\ZZ$ to $\NN_{0}$ and define the inequality pointwise.}
	\end{enumerate}

	Further, for any $a\in\Spec(\pi_0A)_{\cl}$, there exist a quasi-compact open neighbourhood $U_a$ of $a$  (resp.\ locally closed subset $V_{a}$ containing $a$) and a function $\tau\colon\ZZ\rightarrow\NN_{0}^{\ZZ}$ with finite support, such that $\beta_{\Fline |U_a}\leq \tau$ in the above sense (resp.\ $\beta_{\Fline |V_a}$ is constant and equal to $\tau$); this follows from Lemma~\ref{upper semi-cont qc}.
\end{defrem}

\begin{rem}
	Let us come back to our example of a proper smooth morphism $f\colon X\rightarrow \Spec(A)$.
	Let $\tau\colon \ZZ\rightarrow \NN_{0}^{\ZZ}$ be a function with finite support. Then the derived $F$-zip  $\underline{R\Gamma_{\dR}(X/A)}$ of Example~\ref{ex fzip de rham local} has type at most $\tau$ if the Hodge numbers satisfy $$\dim_{\kappa(a)}H^{-j-i}(X_{\kappa(a)},\Omega^{i}_{X_{\kappa(a)}/\kappa(a)})\leq \tau(i)_{j}$$ for all $i,j\in\ZZ$ and $a\in\Spec(A)$ (note that the minus signs in the Hodge numbers appear since we use \textit{homological} notation).
\end{rem}

\begin{defrem}
	\label{defi Euler char}
	Let $A$ be an animated $R$-algebra and $\Fline\coloneqq (C^\bullet,D_\bullet,\phi,\varphi_\bullet)$ be a derived $F$-zip over $A$.	Let us look at the function 
	\begin{align*}
	\chi_k(\Fline)\colon \Spec(\pi_0A)_{\cl}&\longrightarrow \ZZ\\
	s&\longmapsto \chi(\gr^kC \otimes_{A}\kappa(s)).
	\end{align*}
	This is a locally constant function (see \cite[0B9T]{stacks-project}). Since the filtrations on derived $F$-zips are bounded, we also know that the function $\chi_{\Fline}\colon a\mapsto (k\mapsto \chi_k(\Fline)(a))$ is also locally constant as a map from $\Spec(\pi_{0}A)_{\cl}$ to functions $\ZZ\rightarrow \ZZ$ with finite support. We call $\chi_{\Fline}$ the \textit{Euler characteristic} of $\Fline$.
 
	If $\tau\colon \ZZ\rightarrow \ZZ$ is a function with finite support, we say that \textit{$\Fline$ has Euler characteristic $\tau$} if  $\chi_{\Fline}$ is constant with value $\tau$.
\end{defrem}

The reason behind the following definition will become clear later on. One problem will be that we cannot classify $F$-zips of fixed type since the type is only ``upper semi-continuous'' (we do not explicitly define this notion but hope that the idea is clear from the previous definitions and remarks). This can be resolved when we assume that the homotopies of the graded pieces are finite projective.

\begin{defi}
	Let $A$ be an animated ring.
	\begin{enumerate}
		\item Let $M\in\MMod_A$ be a perfect $A$-module. Fix a map $r\colon \ZZ\rightarrow \NN_0$. We call $M$ \textit{homotopy finite projective of rank $r$} if for all schemes $S$ and scheme morphisms $f\colon S\rightarrow \Spec(\pi_0A)_{\cl}$, the $\Ocal_{S}$-module $\pi_i( f^*(M\otimes_{A}\pi_{0}A))$ is finite locally free of rank $r_i$.
 
		\item Let $\Fline\coloneqq (C^\bullet,D_\bullet,\phi,\varphi_\bullet)$ be a derived $F$-zip over $A$ and $\tau\colon \ZZ\rightarrow \NN_0^{\ZZ}$ be a function with finite support. We say that $\Fline$ is \textit{homotopy finite projective of type $\tau$} if for all $i\in\ZZ$, we have that $\gr^iC$ is homotopy finite projective of rank $\tau(i)$.
	\end{enumerate}
\end{defi}

\begin{remark}
	By the above definition, if a derived $F$-zip is homotopy finite projective of type $\tau$, for $\tau$ as above, then its ascending filtration has type $\tau$. 
\end{remark}

\begin{defi}
	Let $A$ be an animated $R$-algebra. Let $\tau\colon \ZZ\rightarrow \NN_0^\ZZ$ be a function with finite support. 
	\begin{enumerate}
		\item We define $\FZip^{\leq \tau}_{\infty,R}(A)$ to be the subcategory of $\FZip_{\infty,R}(A)$ consisting of those derived $F$-zips with type at most $\tau$.
		\item We define $\FZip^\tau_{\infty,R}(A)$ to be the subcategory of $\FZip_{\infty,R}(A)$ consisting of those derived $F$-zips that are homotopy finite projective of type $\tau$.
		\end{enumerate} 
	The associated functors are denoted by $\FZip^{\leq \tau}_{\infty,R}$ and $\FZip^{\tau}_{\infty,R}$, and the associated functors from $\AniAlg{R}$ to $\SS$ are denoted by $\FZip^{\leq \tau}_{R}$ and $\FZip^{\tau}_{R}$, respectively.
\end{defi}

\begin{prop}
	\label{open and closed substacks of Fzip}
	 The functors  $\FZip^{\leq \tau}_{\infty,R}$ and $\FZip^{\tau}_{\infty,R}$ are hypercomplete fpqc sheaves.
	In particular, $\FZip^{\leq \tau}_{R}$ and $\FZip^{\tau}_{R}$ are derived substacks of $\FZip_{R}$.
\end{prop}
\begin{proof}
	Using  arguments as in the proof of Proposition~\ref{fzip infty sheaf}, we only have to show that if $A\rightarrow \widetilde{A}$ is faithfully flat, a derived $F$-zip over $A$ has type at most $\tau$, respectively is homotopy finite projective of type $\tau$, if and only this holds after base change to $\widetilde{A}$. But by faithful flatness, we know that $\Spec(\pi_{0}\widetilde{A})_{\cl}\rightarrow \Spec(\pi_{0}A)_{\cl}$ is faithfully flat.  Now the definitions involved easily show the claim as we note that for an $A$-module $M$ and any commutative diagram of the form
	$$
	\begin{tikzcd}
		\kappa(a')\arrow[r,""]\arrow[d,""]& \Spec(\pi_{0}\widetilde{A})\arrow[d,""]\\
		\kappa(a)\arrow[r,""]&\Spec(\pi_{0}A)\rlap{,}
	\end{tikzcd}
	$$
	where $\kappa(a)$ is the residue field of a point $a\in \Spec(\pi_{0}A)$ and $\kappa(a')$ is the residue field of a lift $a'\in\Spec(\pi_{0}\widetilde{A})$ of $a$ (this exists by faithful flatness), we have
	\begin{align*}
		\dim_{\kappa(a')}\pi_i (M\otimes_{A}\widetilde{A}\otimes_{\widetilde{A}} \kappa(a'))&= \dim_{\kappa(a')}\pi_i \left(M\otimes_A \kappa(a) \otimes_{\kappa(a)} \kappa(a')\right)\\
		& = \dim_{\kappa(a')}\pi_i \left(M\otimes_A \kappa(a)\right) \otimes_{\kappa(a)} \kappa(a') \\&= \dim_{\kappa(a)}\pi_i \left(M\otimes_{A} \kappa(a)\right),
	\end{align*}
	where we use the flatness of field extensions for the second equality.
\end{proof}

\begin{defi}
	Let $A$ be an animated $R$-algebra. Let $\tau\colon \ZZ\rightarrow\ZZ$ be a locally constant function with finite support. We define $\FZip_{\infty,R}^{\chi=\tau}(A)$  to be the subcategory of $\FZip_{\infty,R}(A)$ classifying those derived $F$-zips $\Fline$ with $\chi_{\Fline}=\tau$.
 
	The associated functor is denoted by $\FZip_{\infty,R}^{\chi=\tau}$, and the associated presheaf from $\AniAlg{R}$ to $\SS$ is denoted by $\FZip_{R}^{\chi=\tau}$.
\end{defi}

\begin{lem}
	The functor $\FZip_{\infty,R}^{\chi=\tau}$ is a hypercomplete fpqc sheaf. In particular, the functor $\FZip_{R}^{\chi=\tau}$ is a derived substack of $\FZip_{R}$.
\end{lem}
\begin{proof}
	The proof is completely analogous to that of Proposition~\ref{open and closed substacks of Fzip}.
\end{proof}

In the following, we want to show that the inclusions of the derived substacks $\FZip^{\leq \tau}$, $\FZip^{\tau}$ and $\FZip^{\chi=\tau}$ into $\FZip$ are in fact geometric. To show the geometricity of $\FZip^{\tau}$, we will need a proposition from the book \cite{WED2} of G\"ortz--Wedhorn. This proposition in particular shows the reason behind the definition of \textit{homotopy finite projectiveness}. The finite projectiveness of the homotopy groups is needed to have some geometric structure if we fix the type.

\begin{lem}[\textit{cf.}~\protect{\cite[Proposition 23.130]{WED2}}]
	\label{betti numbers locally closed}
	Let $S$ be a scheme, let $E$ be a perfect complex in $D(S)$ of Tor-amplitude $[a,b]$, and let $I\subseteq [a,b]$ be an interval containing $a$ or $b$. Fix a map $r\colon I\rightarrow \NN_0$, $i\mapsto r_i$. Then there exists a unique locally closed subscheme $j\colon Z= Z_r\hookrightarrow S$ such that a morphism $f\colon T\rightarrow S$ factors through $Z$ if and only if for all morphisms $g\colon T'\rightarrow T$, the $\Ocal_{T'}$-module $\pi_i(L(f\circ g)^\ast E)$ is finite locally free of rank $r_i$ for all $i\in I$. Moreover,
	\begin{enumerate}
		\item the immersion $j\colon Z\hookrightarrow S$ is of finite presentation;
		\item as a set, one has
		$$
		Z=\lbrace s\in S\mid \dim_{\kappa(s)}\pi_i(E\otimes^L_{\Ocal_S} \kappa(s))=r_i\text{ for all }i\in I\rbrace; 
		$$
		\item\label{bnlc-3} if $f\colon T\rightarrow S$ factors as $T\xrightarrow{\fbar}Z\xrightarrow{j} S$, then $\pi_i(Lf^*E\otimes_{\Ocal_T}^{L} \Gcal)=\fbar^*\pi_i(j^*E)\otimes_{\Ocal_T} \Gcal$ for all $i\in I$ and for all quasi-coherent $\Ocal_{T}$-modules $\Gcal$.
	\end{enumerate}
\end{lem}

We are going to show that derived $F$-zips of certain type are classified by open (resp.\ locally closed) substacks of $\FZip_{R}$. We have seen in Remark~\ref{lift along affine} that an open immersion $U\hookrightarrow \Spec(\pi_{0}A)$ of $R$-algebras, where $A\in\AniAlg{R}$, can be lifted to an open immersion $\widetilde{U}\hookrightarrow \Spec(A)$. Also, any morphism $\Spec(T)\rightarrow \Spec(A)$ factors through $\widetilde{U}$ if \'etale locally $\Spec(\pi_{0}T)\rightarrow \Spec(\pi_{0}A)$ factors through $U$.
 
We could now try to do the same for closed immersions. So, for a closed immersion, let us say induced by an element $a\in\pi_{0}A$, so of the form $\Spec(\pi_{0}A/(a))\rightarrow \Spec(\pi_{0}A)$, we get a closed immersion $\Spec(A\!\sslash\!\!(a))\rightarrow \Spec(A)$, where $A\!\sslash\!\!(a)$ is the derived quotient.\footnote{The derived quotient is defined in the following way. An element in $a\in\pi_{0}A=\pi_{0}\Omega^{\infty}A$ gives rise to an element $f\in\Hom(\ZZ[X],A)\simeq \Omega^{\infty}A$. Thus we can define the derived quotient as $A\!\sslash\!\!(a)\coloneqq A\otimes_{f,\ZZ[X],X\mapsto 0}\ZZ.$} But it is not clear that a morphism $\Spec(T)\rightarrow \Spec(A)$ factors through $\Spec(A\!\sslash\!\! (a))$ if and only it does \'etale locally on $\pi_{0}$. In particular, in the following proposition we cannot show that $\FZip^{\tau}\hookrightarrow \FZip^{\leq\tau}$ is a locally closed immersion; we can only do so after restricting the functors to $R$-algebras (we can even show that it is open and closed).

\begin{rem}
\label{rem.substack.forget}
	In the next proposition, we will analyze the geometric structure of the inclusion $i\colon \FZip_{R}^{\leq\tau}\hookrightarrow \FZip_{R}$. For this, we want to understand the pullback
	$$
		\begin{tikzcd}
			X\arrow[r,""]\arrow[d,""]& \Spec(A)\arrow[d,""]\\
			\FZip_{R}^{\leq\tau}\arrow[r,"i"]&\FZip_{R}\rlap{,}
		\end{tikzcd}
	$$
	where $\Spec(A)\rightarrow \FZip_{R}$ is given by some derived $F$-zip $\Fline\coloneqq (C^{\bullet},D_{\bullet},\phi,\varphi_{\bullet})$. As explained in Remark~\ref{F-zips.easy}, we can find an equivalence between $\Fline$ and a derived $F$-zip $\Gline\coloneqq (C'^{\bullet},D'_{\bullet},\phi',\varphi'_{\bullet})$, where $\phi'$ is given by the identity on $\colim_{\ZZ^{\op}}C'^{\bullet}$. This equivalence between $\Fline$ and $\Gline$ defines homotopies
	$$
		\begin{tikzcd}
			\Spec(A)\arrow[rd,"\Gline",swap]\arrow[r,"\id_{\Spec(A)}"]& \Spec(A)\arrow[d,"\Fline"]\\
			&\FZip_{R}\rlap{,}
		\end{tikzcd} \quad
		\begin{tikzcd}
			\Spec(A)\arrow[rd,"\Fline",swap]\arrow[r,"\id_{\Spec(A)}"]& \Spec(A)\arrow[d,"\Gline"]\\
			&\FZip_{R}\rlap{.}
		\end{tikzcd} 
	$$
	So, in particular, the pullbacks $\FZip_{R}^{\leq\tau}\times_{i,\FZip_{R},\Fline}\Spec(A)$ and $\FZip_{R}^{\leq\tau}\times_{i,\FZip_{R},\Gline}\Spec(A)$ are equivalent. Thus, for geometric properties of the inclusion $i$, we may only work with derived $F$-zips of the form $\Gline$ (\textit{i.e.}\ those derived $F$-zips where the equivalence between the colimits is given via the identity).
 
	The same reasoning also works  for other substacks. 
\end{rem}

\begin{prop}
	\label{prop open and closed of F-zip}
	Let $\tau\colon \ZZ\rightarrow \NN_0^\ZZ$ be a function with finite support. Then the inclusion $i\colon \FZip_{R}^{\leq\tau}\hookrightarrow \FZip_{R}$ is a quasi-compact open immersion.
  Further, let  $p\colon \FZip_{R}^\tau\hookrightarrow \FZip^{\leq \tau}_{R}$ denote the inclusion. Then $t_0 p$ is a closed immersion locally of finite presentation. $($Recall that $t_{0}p$ is the restriction of $p$ to $R$-algebras by the inclusion $\Alg{R}\hookrightarrow\AniAlg{R}$.$)$
\end{prop}
\begin{proof}
	Let $\Spec(A)\rightarrow \FZip_{R}$ be a morphism of derived stacks with $A\in \AniAlg{\FF_p}$ classified by a derived $F$-zip $\Fline=(C^\bullet,D_\bullet,\varphi_\bullet)$.\footnote{Keeping Remark~\ref{rem.substack.forget} in mind, we do not need to keep track of the equivalence connecting the colimits of the ascending and descending filtrations.} Then a morphism $f\colon \Spec(T)\rightarrow \Spec(A)$ factors through the product $i^{-1}(\Fline)\coloneqq \Spec(A)\times_{\FZip_{R}}\FZip^{\leq \tau}$ if and only if $f^*\Fline$ has type at most $\tau$. By Lemma~\ref{upper semi-cont qc}, we know that there is a quasi-compact open subscheme $\widetilde{U}$ of $\Spec(\pi_0A)$ classifying those points of $\Spec(\pi_0A)$ where $\Fline$ has type at most $\tau$ (note that the filtrations are bounded and that perfect complexes have only finitely many non-zero homotopy groups). We claim that $f$ factors through $i^{-1}(\Fline)$ if and only if it factors through the lift $U$ of $\widetilde{U}$ constructed in Remark~\ref{lift along affine}; \textit{i.e.}\ if we write $\widetilde{U}=\bigcup_{i=0}^{n}\Spec(\pi_{0}A_{f_{i}})$, we define $U$ as the image of $\coprod_{i=0}^{n} \Spec(A[f_{i}^{-1}])\rightarrow \Spec(A)$ (note that this construction implies that $U$ is quasi-compact).
 
	Indeed, it is clear that if $f$ factors through $U$, then it certainly factors through $i^{-1}(\Fline)$. Now assume $f$ factors through $i^{-1}(\Fline)$. In particular, we have $$\dim_{\kappa(t)}\pi_i(f^*\gr^jC\otimes_T \kappa(t))\leq \tau(j)_i$$ for all $t\in \Spec(\pi_0T)$ and $i,j\in \ZZ$. But we have the equalities
	\begin{align*}
		\dim_{\kappa(t)}\pi_i (f^*\gr^jC\otimes_T \kappa(t))&= \dim_{\kappa(t)}\pi_i (\gr^jC\otimes_A \kappa(t))\\&= \dim_{\kappa(t)}\pi_i \left(\gr^jC\otimes_A \kappa(\pi_0f(t)) \otimes_{\kappa(\pi_0f(t))} \kappa(t)\right)\\
		& = \dim_{\kappa(t)}\pi_i \left(\gr^jC\otimes_A \kappa(\pi_0f(t))\right) \otimes_{\kappa(\pi_0f(t))} \kappa(t) \\&= \dim_{\kappa(\pi_0f(t))}\pi_i \left(\gr^jC\otimes_A \kappa(\pi_0f(t))\right),
	\end{align*}
	where we use the flatness of $\kappa(f(t))\rightarrow\kappa(t)$ in the fourth equality. This shows that $\pi_0f$ factors through~$\widetilde{U}$. Let us write $\widetilde{U}=\bigcup_{j\in J}\Spec((\pi_{0}A)_{f_j})_{\cl}$ as a finite union of principal affine opens in $\Spec(\pi_0A)$. Then $f$ factors though $U= \bigcup_{j\in J}{\Spec(A[f_j^{-1}])}$ if and only if there is an \'etale cover $(T\rightarrow T_k)_{k\in I}$ such that $\Spec(\pi_0T_i)$ factors through some $\Spec((\pi_{0}A)_{f_j})$ (see Remark~\ref{lift along affine}). But this is clear since the base change of $\coprod_{j\in J} \Spec((\pi_{0}A)_{f_{j}})\rightarrow \widetilde{U}$ to $\Spec(\pi_0T)$ gives an \'etale cover of $\Spec(\pi_0T)$, which can be lifted to an \'etale cover of $\Spec(T)$ (use Proposition~\ref{lift etale} and note that faithful flatness can be checked on $\pi_{0}$), where this property holds by definition.
 
	For $t_0\FZip^{\tau}_{R}$, let $\Spec(A)\rightarrow t_0\FZip^{\leq\tau}_{R}$ be classified by a derived $F$-zip $\Fline$ over an $R$-algebra $A$. Then a morphism $f\colon \Spec(T)\rightarrow \Spec(A)$ factors through the projection $t_0p^{-1}(A)\rightarrow\Spec(A)$ if and only if $f^*\Fline$ is homotopy finite projective of type $\tau$. By Lemma~\ref{betti numbers locally closed} and the finiteness of the filtrations, we can find a locally closed subscheme $Z$ of $\Spec(A)$, where $Z\hookrightarrow \Spec(A)$ is finitely presented, such that $f$ factors through $Z$ if and only if $f^*\Fline$ is homotopy finite projective of type $\tau$. So, in particular, $Z\simeq t_0p^{-1}(A)$.
 
What is left to show is that $Z$ is also closed in $\Spec(A)$. Using the upper semi-continuity of the Betti numbers (see Lemma~\ref{upper semi-cont qc}), we see that the image of $Z$ in $\Spec(A)$ is a closed subset, and therefore the immersion $Z\hookrightarrow \Spec(A)$ is in fact closed (see \cite[01IQ]{stacks-project}).
\end{proof}

\begin{prop}
	\label{f-zip decomp}
	Let $\tau\colon\ZZ\rightarrow\ZZ$ be a function with finite support. The inclusion $i\colon \FZip^{\chi=\tau}_{R}\hookrightarrow \FZip_{R}$ is an open immersion of derived stacks, and further $t_0i$ is an open and closed immersion. 
\end{prop}
\begin{proof}
	This is completely analogous to the proof of Proposition~\ref{prop open and closed of F-zip} with Remark~\ref{defi Euler char}. Nevertheless, we give a proof for completeness.

	Let $\Spec(A)\rightarrow \FZip$ be a morphism of derived stacks with $A\in \AniAlg{\FF_p}$ classified by a derived $F$-zip $\Fline=(C^\bullet,D_\bullet,\varphi_\bullet)$. Then a morphism $f\colon \Spec(T)\rightarrow \Spec(A)$ factors through $i^{-1}(\Fline)$ if and only if $\chi(f^*\Fline)=\tau$. Since $\chi$ is locally constant, we know that we can find an open and closed subscheme $\widetilde{U}$ of $\Spec(\pi_0A)$ classifying those points of $\Spec(\pi_0A)$ where $\Fline$ has Euler characteristic $\tau$ (note that the filtrations are bounded). Let $U$ be the lift of $\widetilde{U}$ on $\Spec(A)$ constructed in Remark~\ref{lift along affine}. Then by construction, $U$ is open in $\Spec(A)$. We claim that $f$ factors through $i^{-1}(\Fline)$ if and only if it factors through $U$.
 
	Indeed, if $f$ factors through $U$, then certainly it also factors through $i^{-1}(\Fline)$. Now assume $f$ factors through $i^{-1}(\Fline)$. In particular, we have $\chi(f^*\gr^k_CM\otimes_T\kappa(t))=\tau(k)$ for all $t\in \Spec(\pi_0T)$ and $k\in \ZZ$. But analogously to the proof of Proposition~\ref{prop open and closed of F-zip}, we have
        $$
        \chi(f^*\gr^k_CM\otimes_T\kappa(t))=\chi(f^*\gr^k_CM\otimes_A\kappa(\pi_0f(t))).
        $$
        This shows that $\pi_0f$ factors through $\widetilde{U}$. Now, this factorization can be lifted to a factorization of $f$ through $U$ (again the argumentation is the same as in the proof of Proposition~\ref{prop open and closed of F-zip}).
 
	That $t_0i$ is open and closed follows immediately from the above.
\end{proof}

\begin{remark}
	\label{F-zip as colim of open im}
	By the above, $\FZip^{\leq\tau}_{R}\hookrightarrow \FZip_{R}$ is a $0$-geometric open immersion. In fact, we can see that $\colim_{\tau}\FZip^{\leq\tau}_{R}\simeq \FZip_{R}$, where $\tau$ runs through the functions $\ZZ\rightarrow\NN_0^\ZZ$ with finite support. This colimit is filtered, as we can view $\tau$ as a function $\ZZ\times\ZZ\rightarrow\NN_{0}$ and get a pointwise order.
 
	For the equivalence, note that any $F$-zip $\Fline$ has finitely many graded pieces, which all have finite Tor-amplitude, so we can find a function $\sigma\colon \ZZ\rightarrow \NN_{0}^{\ZZ}$ with finite support such that $\Fline$ has type at most $\sigma$, \textit{i.e.}\ $\Fline\in \FZip^{\leq\sigma}$.
 
	Further, the inclusion $\FZip^{\leq\tau}_{R}\hookrightarrow \FZip_{R}$ has to factor through some $\FZip^{[a,b],S}_{R}$ by the fact that the filtrations are bounded and by Remark~\ref{local rank}. By the same arguments as in the proof of Proposition~\ref{prop open and closed of F-zip}, we see that $\FZip^{\leq\tau}_{R}\hookrightarrow \FZip^{[a,b],S}_{R}$ is a quasi-compact open immersion, which implies that $\FZip^{\leq\tau}_{R}$ itself is $(b-a+1)$-geometric and locally of finite presentation. In particular, we can write $\FZip_{R}$ as the filtered colimit of geometric derived open substacks:  
	$$
	\FZip_{R}\simeq \limind_{\tau}\FZip^{\leq \tau}_{R}.
	$$
\end{remark}

\begin{thm}
\label{fzip local geometric plus}
	The derived stack $\FZip_{R}$ is locally geometric and locally of finite presentation.
\end{thm}
\begin{proof}
	This follows from Remark~\ref{F-zip as colim of open im}.
\end{proof}

\subsection{Strong derived $\boldsymbol{F}$-zips over affine schemes}

In the following, we want to look at derived $F$-zips where the underlying ascending and descending filtrations are strong. A morphism between modules is a monomorphism if and only if the induced long exact sequence splits into short exact sequences. Using classical theory, it is not hard to see that restricting to those derived $F$-zips where the underlying filtrations are strong and homotopy finite projective gives an open and closed derived substack of $t_{0}\FZip$. We can also use this open and closed derived substack to easily embed the stack of classical $F$-zips into the derived version via an open immersion.

\begin{defi}
	Let $A$ be an animated $R$-algebra and $\tau\colon \ZZ\rightarrow \NN_0^{\ZZ}$ be a function with finite support. We define the full sub-$\infty$-category $\sFZip^{\tau}_{\infty,R}(A)$ of $\FZip^{\tau}_{\infty,R}(A)$ as consisting of those derived $F$-zips $(C^\bullet,D_\bullet,\phi,\varphi)$ where the filtrations $C^\bullet$ and $D_\bullet$ are in fact strong filtrations. An element in $\sFZip^{\tau}_{\infty,R}(A)$ is called a \textit{strong derived $F$-zip over $A$ of type $\tau$}.
\end{defi}

\begin{rem}
	Let $\tau\colon \ZZ\rightarrow \NN_0^{\ZZ}$ be a function with finite support. Let $A$ be a classical ring. By definition, it is immediate that the base change along classical rings of strong derived $F$-zips of type $\tau$ is again strong of type $\tau$. Therefore, the pullback induces a functor $A\mapsto \sFZip^{\tau}_{\infty,R}(A)$ from $\Alg{R}$ to $\ICat$ (here $\Alg{R}$ denotes the nerve of the category of $R$-algebras).
\end{rem}

\begin{rem}
	The condition that the underlying filtrations are strong  may seem  useful,  as  it turns out, this is a very strong condition (see Theorem~\ref{de Rham strong degen}). This is because monomorphisms of perfect complexes (on a classical ring) with finite projective homotopy groups are automatically \textit{split}. So a strong filtration is automatically determined by its underlying graded pieces, which indicates that the corresponding spectral sequence should degenerate. But the reason behind the derived $F$-zips is precisely the study of those filtrations with non-degenerate spectral sequences.
 
	In the above definition, we have to fix a type as otherwise it is not clear that strongness is stable under base change. Nevertheless, in this paper we will not encounter strong filtrations that are not homotopy finite projective.
\end{rem}

\begin{defprop}
	\label{strong fzip infty sheaf}
	 Let $\tau\colon \ZZ\rightarrow \NN_0^{\ZZ}$ be a function with finite support.
	The presheaf 
	\begin{align*}
	\sFZip_{\infty,R}^{\tau}\colon \Alg{R}&\longrightarrow \ICat\\
	A &\longmapsto \sFZip^{\tau}_{\infty,R}(A)
	\end{align*}
	is a hypercomplete sheaf for the fpqc topology.
\end{defprop}
\begin{proof}
	Analogously to the proof of Proposition~\ref{fzip infty sheaf},  it is enough to show that a morphism of perfect modules with homotopy finite projective cofiber is a monomorphism if and only if it is after passage to an fpqc cover.
 
	Let $A\rightarrow A^{\bullet}$ be an  fpqc hypercover  of an $R$-algebra. For a morphism $f\colon M\rightarrow N$ of $A$-modules to be a monomorphism is equivalent to the fact that the natural maps $\pi_{i}N\rightarrow \pi_{i}\cofib(f)$ are surjections for all $i\in\ZZ$. By the flatness of $A\rightarrow A^{\bullet}$,  for any $A$-module $L$, we have  that $\pi_{i}(L\otimes_{A}A^{\bullet})\cong \pi_{i}(L)\otimes_{A}A^{\bullet}$.  Note that by assumption, the $\pi_{i}\cofib(f)$ are finite projective and compatible with base change.  Thus, we see that it is enough to check that any morphism of classical $A$-modules with finite projective target is surjective if it is so after passage to an fpqc cover. But this is classical (see \cite[Proposition 8.4]{WED}).
\end{proof}

\begin{defi}
	Let $\tau\colon \ZZ\rightarrow \NN_0^{\ZZ}$ be a function with finite support. We define the \textit{derived stack of strong $F$-zips of type $\tau$} as
	\begin{align*}
	\sFZip^{\tau}_{R}\colon  \Alg{R}&\longrightarrow \SS\\
	A&\longmapsto \sFZip^{\tau}_{\infty,R}(A)^\simeq .
	\end{align*}
\end{defi}

Next, we want to show that the stack of strong derived $F$-zips is geometric. This is clear if we know that strong bounded perfect filtrations are open and closed in bounded perfect filtrations (note that the proof of Theorem~\ref{F-Zips locally geometric} shows that the derived stack of bounded perfect filtrations is geometric). This is an immediate consequence of the following lemma.  

\begin{lem}
\label{lem.ret.mono}
	Let $A$ be a $($classical\,$)$ ring and $M,N$ be perfect $A$-modules. Let $\Vcal\colon \Alg{A}\rightarrow \SS$ be the derived stack classifying $A$-module morphisms $f\colon M\rightarrow N$ such that $M$ and $\cofib(f)$ are homotopy finite projective of some types. Further, let $\Wcal$ be the derived substack defined by those morphisms $f\colon M\rightarrow N$ that are monomorphisms. Then the inclusion $\Wcal\hookrightarrow \Vcal$ is an open and closed immersion.
\end{lem}
\begin{proof}
	The locus where $N\rightarrow \cofib(f)$ is surjective is represented by an open subscheme of $\Spec(A)_{\cl}$ (see \cite[Proposition 8.4]{WED}). In particular, $\Wcal\hookrightarrow \Vcal$ is represented by an open immersion. By the projectivity of the $\pi_{i}(M)$, the locus where $\pi_{i}f$ is injective for all $i\in \ZZ$ is equivalently the locus where $\pi_{i}\cofib(f)\rightarrow \pi_{i-1}M$ vanishes. But this is closed in $\Spec(A)_{\cl}$ (again by \cite[Proposition 8.4]{WED}).
\end{proof}

\begin{lem}
\label{lem.strong.fil.open}
Let $R$ be a ring and $\tau\colon \ZZ\rightarrow \NN_0^{\ZZ}$ be a function with finite support. Let $\Fun^{b}_{\tau}(\ZZ,\PPerf_{R})$ be the derived stack of perfect filtrations that are homotopy finite projective of type $\tau$. Let $$\Fun^{s}_{\tau}(\ZZ,\PPerf_{R})\subseteq t_{0}\Fun_{\tau}(\ZZ,\PPerf_{R})$$ be the derived substack of strong perfect filtrations that are homotopy finite projective of type $\tau$. Then the inclusion 
$$
	\Fun^{s}_{\tau}(\ZZ,\PPerf_{R})\xhookrightarrow{\hphantom{aaa}} t_{0}\Fun_{\tau}(\ZZ,\PPerf_{R})
$$
is an open and closed immersion.
 
The same holds if we replace $\ZZ$ with $\ZZ^{\op}$.
\end{lem}
\begin{proof}
	By Lemma~\ref{lem.ret.mono}, the stack classifying monomorphisms between perfect complexes is open and closed. As the filtrations are bounded, this proves the lemma (see Remark~\ref{local rank}).
\end{proof}

\begin{prop}
\label{strong open}
	Let  $\tau\colon \ZZ\rightarrow \NN_0^{\ZZ}$ be a function with finite support. The inclusion $\sFZip^{\tau}_{R}\hookrightarrow t_{0}\FZip^{\tau}_{R}$ is an open and closed immersion.
\end{prop}
\begin{proof}
	This follows from Lemma~\ref{lem.strong.fil.open}.
\end{proof}

In the next proposition, we want to show that the stack of classical $F$-zips lies quasi-compact open in $t_{0}\FZip_{R}$. To do so, we fix some type of a classical $F$-zip, let us say $\sigma$ (as the classical $F$-zips are the disjoint union of such, this is enough). Then we only look at strong derived $F$-zips where the graded pieces are all finite projective modules sitting in one degree. This can also be achieved by fixing a type $\tau$ of the strong derived $F$-zips (see Remark~\ref{local rank}). But if we choose $\tau$ nicely in relation to $\sigma$, we see that strong derived $F$-zips of type $\tau$ are precisely classical $F$-zips of type $\sigma$. Since we only work with derived $F$-zips corresponding to finite projective modules, we thus see the quasi-compact openness of classical $F$-zips in derived ones.

\begin{lem}	
	\label{F-zip in classical}
	Let $\sigma\colon \ZZ\rightarrow \NN_0$ be a function with finite support. Let us define $\tau^\sigma\colon \ZZ\rightarrow \NN_0^{\ZZ}$ to be the function given by $k\mapsto \tau^\sigma(k)_0=\sigma(k)$ and $\tau^\sigma(k)_j=0$ for $j\neq 0$. 
	Let $\cl \FZip_R$ denote the classical stack of $F$-zips. Then the inclusion $\cl \FZip^{\sigma}_R\hookrightarrow t_0\FZip_R$ is a quasi-compact open.
\end{lem}
\begin{proof}
	We know that $t_{0}\FZip_R^{\leq \tau^\sigma}$ is open in $t_{0}\FZip_R$, but by our construction, for an element $\Fline=(C^\bullet,D_\bullet,\varphi_\bullet)$ given by $\Spec(A)\rightarrow t_{0}\FZip_R^{\leq \tau^\sigma}$, we have that $\gr^k C$ is equivalent to a finite projective module sitting in degree $0$. In particular, the function $\beta_{\gr^k C}$ is locally constant. Therefore, even $t_{0}\FZip_R^{\tau^\sigma}$ is quasi-compact open in $t_{0}\FZip_R$. Now we have an equivalence $\cl \FZip_{R}^{\sigma}\simeq t_0\sFZip_{R}^{\tau^\sigma}$, concluding the proof using Pro\-po\-si\-tion~\ref{strong open}.
\end{proof}

\subsection{Globalization}

\subsubsection{Globalization of derived stacks}

In the following, we want to look at derived $F$-zips over derived schemes. We also want to look at some properties of the corresponding sheaf. Important here is that derived stacks take \textit{affine derived schemes} as parameters and not \textit{derived schemes}. Let us show how to fix this.

Let $R$ be a ring. We start by extending a derived stack $X\colon \AniAlg{R}\rightarrow \SS$ via right Kan extension to a presheaf $\RR X\colon \Pcal(\AniAlg{R}^{\op})^{\op}\rightarrow \SS$. Using Remark~\ref{rem.sheaf.kan}, we see that $\RR X$ is in fact an \'etale sheaf. In particular, since every derived scheme has an open cover by affines, $\RR X_{|\dSch}$ is uniquely determined by $\RR X_{|\AniAlg{R}}\simeq X$.

\medskip
In the case of derived $F$-zips, we could finish this section with the arguments above. But we can also define derived $F$-zips over derived schemes analogously to Definition~\ref{defi derived fzip}. We will see that this definition agrees with the definition given by right Kan extension.

\subsubsection{Filtrations over derived schemes}
\label{sec global filtration}

We want to globalize the construction of derived $F$-zips. We could do this by right Kan extension but also give a direct definition by globalizing filtrations and define derived $F$-zips over derived schemes
analogously to Definition~\ref{defi derived fzip}. In fact, both definitions will agree (see Lemma~\ref{Fzips global given by affine}).
 
When working with derived schemes, we always assume that our module categories are small; \textit{i.e.}\ for any animated ring $A$, we assume that $\MMod_{A}$ is small. This is because we want to use Proposition~\ref{right kan of sheaf} to see that quasi-coherent modules satisfy descent. 
\newline

Let us look at the functor from $\AniAlg{R}$ to $\ICat$ given by $A\mapsto \Fun(\ZZ, \MMod_{A})$. By Lemma~\ref{Fun sheaf}, we know that this functor satisfies fpqc hyperdescent. So its right Kan extension to derived schemes will still be an fpqc sheaf (see Remark~\ref{descent qcoh}). Since $\Fun(\ZZ, -)$ commutes with limits (which was used in the proof of Lemma~\ref{Fun sheaf}), we immediately see that $\RR\Fun(\ZZ,\MMod_{-})(S)\simeq \Fun(\ZZ,\QQCoh(S))$ for any derived scheme $S$.

We have that boundedness can be checked fpqc hyperlocally (which was used in the proof of Proposition~\ref{fzip infty sheaf}). Thus, we can right Kan extend to derived schemes and get a sheaf we denote by
$$
\Fun^{b}(\ZZ,\QQCoh(-)).
$$
By Proposition~\ref{right kan ex derived scheme}, we see that for a derived scheme $S$, an element in $\Fun^{b}(\ZZ,\QQCoh(-))$ is given by  a functor $F\in \Fun(\ZZ,\QQCoh(S))$ such that for any affine open $\iota\colon \Spec(A)\hookrightarrow S$, the ascending filtration $\iota^{*}F$ is bounded.

Also, we have that a functor $F\in\Fun(\ZZ,\QQCoh(S))$ is in $\Fun^{b}(\ZZ,\QQCoh(-))$ if and only if there is a flat atlas $(\Spec(A_{i})\xrightarrow{p_{i}}S)_{i\in I}$ such that $p_{i}^{*}F$ is bounded.
 
The same argumentation as above also works for perfect filtrations and strong filtrations of some given type (here we have to restrict to classical schemes). Let us note that the above stays true if we replace $\ZZ$ with $\ZZ^{\op}$ (or in general with any $\infty$-category, but we do not need this).

\begin{defi}
	Let $\tau\colon\ZZ\rightarrow \NN_{0}^{\ZZ}$ be a function with finite support. Let $S$ be a derived scheme. An \textit{ascending} (\textit{resp.\ descending}) \textit{filtration of quasi-coherent modules over $S$} is an element $F\in \Fun(\ZZ,\QQCoh(S))$ (resp.\ $F\in \Fun(\ZZ^{\op},\QQCoh(S)))$.
	\begin{enumerate}
		\item We say that $F$ is \textit{locally bounded} if $F$ lies in $\Fun^{b}(\ZZ,\QQCoh(S))$ (resp.\ $\Fun^{b}(\ZZ^{\op},\QQCoh(S)))$,
		\item We say that $F$ is \textit{perfect} if $F$ lies in $\Fun_{\perf}(\ZZ,\QQCoh(S))$ (resp.\ $\Fun_{\perf}(\ZZ^{\op},\QQCoh(S)))$.
		\item Moreover, if $S$ is a (classical) scheme, then we say  that $F$ is \textit{strong of type $\tau$} if $F$ lies in $\Fun^{s}_{\tau}(\ZZ,\QQCoh(S))$ (resp.\ $\Fun^{s}_{\tau}(\ZZ^{\op},\QQCoh(S)))$.
	\end{enumerate} 
\end{defi}

\begin{rem}
  Note that for any derived scheme
  $S\in\Pcal(\AniAlg{R}^{\op})$, the $\infty$-categories $\Fun_{\perf}(\ZZ,\QQCoh(S))$ and
  $\Fun^{b}(\ZZ,\QQCoh(S))$ (resp.\ $\Fun^{s}_{\tau}(\ZZ,\QQCoh(S))$) can be seen as full sub-$\infty$-categories of $\Fun(\ZZ,\QQCoh(S))$ (resp.\  $t_{0}\Fun(\ZZ,\QQCoh(S))$) (the same holds with $\ZZ$ replaced by $\ZZ^{\op}$). 
\end{rem}

\begin{lem}
	Let $\tau\colon\ZZ\rightarrow \NN_{0}^{\ZZ}$ be a function with finite support. Let $S$ be a $($classical\,$)$ scheme and $F$ an ascending $($resp.\ descending\,$)$ strong filtration of $\Ocal_{S}$-modules of type $\tau$. Then for all $i\in\ZZ$, the morphism $\del_{i}\colon F(i-1)\rightarrow F(i)$ $($resp.\ $F(i+1)\rightarrow F(i))$ is a monomorphism. 
\end{lem}
\begin{proof}
	For all $n\in\ZZ$, the morphism $\pi_{n}F(i)\rightarrow \pi_{n}\cofib(\del_{i})$ is a morphism of $\Ocal_{S}$-modules (this follows from Proposition~\ref{derived cat is kan ext}). In particular, surjectivity can be checked Zariski locally (see \cite[Proposition 8.4]{WED}).
\end{proof}

\begin{notation}
	Let $S$ be a derived scheme, and let $\Fcal$ be an $\Ocal_{S}$-module. If $F$ is a descending filtration on $\Fcal$, we write $F^k\coloneqq F(k)$ for $k\in \ZZ$ and $F^\bullet\coloneqq F$. If $G$ is an ascending filtration on $\Fcal$, we write $G_\bullet$ for $G$ and denote its points by $G_k$.
\end{notation}

\begin{defi}
	Let $S$ be a derived scheme, and let $\Fcal$ be an $\Ocal_{S}$-module. Let $F$ be an ascending (resp.\ descending) filtration on $\Fcal$. For any $i\in \ZZ$, we define the \textit{$\supth{i}$ graded piece of\, $F$} as $\gr^i F\coloneqq \cofib (F(i-1)\rightarrow F(i))$ (resp.\ $\gr^i F\coloneqq \cofib (F(i+1)\rightarrow F(i))$).
\end{defi}

\subsubsection{Derived $\boldsymbol{F}$-zips over schemes}

Before defining derived $F$-zips for derived schemes, we first have to make sense of the Frobenius of derived schemes. Classically, the Frobenius on a scheme $X$ is equivalent to the morphism given by composition $X\rightarrow X^{(1)}\rightarrow X$ of the relative Frobenius and the natural map. The points of $X^{(1)}$ are given by restriction along the Frobenius. This can be used to define the Frobenius for derived schemes (even for derived stacks) as in the following.

\begin{remdef}
	 Let $X$ be a derived scheme over $R$. For an animated $R$-algebra $A$, we have an $R$-morphism $\Spec(A)\rightarrow X^{(1)}\coloneqq X\times_{\Spec(R),\Frob_R}\Spec(R)$ if and only if there is a morphism $\Spec(A_{\Frob_R})\rightarrow X$, where $A_{\Frob_{R}}$ is the restriction of $A$ along the Frobenius,\footnote{Using the Frobenius on $R$, we can restrict any $R$-algebra $A$ along the Frobenius $\Frob_{R}\colon R\rightarrow R$; \textit{i.e.}\ $A_{\Frob_{R}}$ is the animated $R$-algebra obtained by composing the natural morphism $R\rightarrow A$ with $\Frob_{R}$.} \textit{i.e.}\ $X^{(1)}(A) \simeq X(A_{\Frob_{R}})$. Also, we have an $R$-algebra map $\Frob_A\colon A\rightarrow A_{\Frob_{R}}$. Thus $X(\Frob_A)$ induces a map $F_{X/S}\colon X\rightarrow X^{(1)}$, which we call the \textit{relative Frobenius of\, $X$}. The composition with the projection gives a map $F_X\colon X\rightarrow X^{(1)}\rightarrow X$, called the \textit{Frobenius of\, $X$}.
\end{remdef}

\begin{rem}
	Let $S$ be a classical scheme. As the (nerve of the) category of schemes lies fully in the $\infty$-category of derived schemes, we see by construction that $F_{S}$ agrees with the classical Frobenius morphism. This can be tested on points given by $R$-algebras, where it holds by definition.
 
	The definition above also agrees with the definition of the Frobenius on animated $R$-algebras, as we have $\AniAlg{R}\simeq \Fun_{\pi}(\Poly_{R}^{\op},\SS)$ (recall that the Frobenius is induced by the Frobenius on polynomial $R$-algebras).

	Moreover, this argument shows that the Frobenius morphism on derived schemes defined above is equivalent to the morphism induced by right Kan extension of the Frobenius on animated rings.
\end{rem}

\begin{defi}
	Let $S$ be a derived scheme over $R$. A \textit{derived $F$-zip over $S$} is a tuple $(C^\bullet,D_\bullet,\phi,\varphi_\bullet)$ consisting of  
	\begin{itemize}
		\item a descending locally bounded perfect filtration $C^\bullet$ of quasi-coherent modules over $S$,
		\item an ascending locally bounded perfect filtration $D_\bullet$ of quasi-coherent modules over $S$,
		\item an equivalence $\phi\colon \colim_{\ZZ^{\op}}C^{\bullet}\simeq \colim_{\ZZ}D_{\bullet}$ and

		\item a family of equivalences $\varphi_k\colon F_S^*\gr^{k}C \xrightarrow{\lowsim} \gr^{k}D$.
        \end{itemize}

	The $\infty$-category of $F$-zips over $S$, \textit{i.e.}\ the full subcategory of 
	\begin{align*}
		\left( \Fun_{\perf}(\ZZ^{\op},\QQCoh(S))\right.&\left.\times_{\colim,\QQCoh(S),\colim}\Fun_{\perf}(\ZZ,\QQCoh(S))\right)\\ &\times_{\prod_\ZZ \Fun(\del\Delta^1,\QQCoh(S))}\prod_\ZZ \Fun(\Delta^1,\QQCoh(S))
	\end{align*}
	 consisting of $F$-zips, is denoted by $\FZip_{\infty,R}(S)$.

	For a morphism $S'\rightarrow S$ of derived schemes over $R$, we have an obvious base change functor $\FZip_{\infty,R}(S)\rightarrow \FZip_{\infty,R}(S')$ via the pullback.
\end{defi}

\begin{rem}
	By definition, for any affine derived scheme $\Spec(A)$, the $\infty$-category $\FZip_{\infty,R}(\Spec(A))$ indeed recovers Definition~\ref{defi derived fzip}. 
 
	Also, as in Remark~\ref{F-zips.easy}, on affine schemes, up to equivalence we may assume that the equivalence between the colimits of the filtrations of a derived $F$-zip is given by the identity.
\end{rem}

\begin{rem}
	Note that if we have a locally bounded perfect filtration $C^{\bullet}$ over some derived scheme $S$,  we have that its colimit in $\QQCoh(S)$ is actually perfect. This can be checked Zariski locally, where the filtrations actually become bounded (note that the colimit is filtered), so  the colimit can be taken over a finite subcategory of $\ZZ$ and thus is perfect.
\end{rem}
 
 The next lemma shows that the $\infty$-category of derived $F$-zips satisfies fpqc descent and that we can extend the definition of derived $F$-zips to arbitrary derived (pre-)stacks.
 
\begin{lem}
	\label{Fzips global given by affine}
	Let\, $\RR\FZip_{\infty,R}$ be the the right Kan extension of\, $\FZip_{\infty,R}\colon \AniAlg{R}\rightarrow \SS$ along the Yoneda embedding $\AniAlg{R}\hookrightarrow\Pcal(\AniAlg{R}^{\op})^{\op}$. Then for any derived scheme $S$ over $R$,  the natural morphism induced by base change
	$$
	\FZip_{\infty,R}(S)\longrightarrow\RR\FZip_{\infty,R}(S)
	$$
	is an equivalence.
\end{lem}
\begin{proof}
	Affine locally on $S$, the assertion is certainly true. So it is enough to show that for an affine open cover $(\Spec(A_{i})\hookrightarrow S)_{i\in I}$, we have
	$$
		\FZip_{\infty,R}(S)\simeq \lim_{\Delta}\FZip_{\infty,R}\left(\Cv\left(\coprod\Spec( A_{i})/S\right)_{\bullet}\right).
	$$
	This is again completely analogous to the proof in the affine case (see Proposition~\ref{fzip infty sheaf}), where we embedded derived $F$-zips into a larger category that satisfied descent. Here we have to use that 
	\begin{align*}
		S\mapsto \Fun_{\perf}(\ZZ^{\op},\QQCoh(-))&\times_{\colim,\QQCoh(-),\colim}\Fun_{\perf}(\ZZ,\QQCoh(S))\\ &\times_{\prod_\ZZ \Fun(\del\Delta^1,\QQCoh(-))}\prod_\ZZ \Fun(\Delta^1,\QQCoh(S))
	\end{align*}
	is given by the limit of right Kan extensions (see the discussion in Section~\ref{sec global filtration}) of sheaves and Remark~\ref{rem.sheaf.kan}.
\end{proof}

Lemma~\ref{Fzips global given by affine} allows us to globalize the derived stack of derived $F$-zips and gives us a direct description of its points.

\begin{defrem}
  We define 
  $$
  \FZip_R\colon \Pcal(\AniAlg{R})^{\op}\longrightarrow \SS
  $$
  as the right Kan extension of $\FZip\colon \AniAlg{R}\rightarrow\SS$ along the inclusion $\AniAlg{R}\hookrightarrow \Pcal(\AniAlg{R})^{\op}$.

	By Remark~\ref{rem.sheaf.kan}, we see that $\FZip$ is a hypercomplete fpqc sheaf.

	Further,  for any derived $R$-scheme $S$, we have that $\FZip_{R}(S)\simeq \FZip_{\infty,R}(S)^{\simeq}$ by Lemma~\ref{Fzips global given by affine} as $(-)^{\simeq}$ commutes with limits.
\end{defrem}

\begin{example}
	\label{ex F-zip proper smooth}
	Let us globalize Example~\ref{ex fzip de rham local}. Let $f\colon X\rightarrow S$ be a proper smooth morphism of schemes. Again, the associated Hodge and conjugate filtrations $\HDG$ and $\conj$ define, respectively, a descending and an ascending perfect bounded filtration of quasi-coherent modules over $S$. We also have equivalences $\varphi_{n}\colon F^{*}_{S}\gr^{n}\HDG\xrightarrow{\lowsim} \gr^{n}\conj$ between the graded pieces (up to Frobenius twist), induced by the Cartier isomorphism. Therefore, we get a derived $F$-zip associated to the proper smooth map $f$ of schemes
	$$
		\underline{Rf_{*}\Omega^\bullet_{X/S}}\coloneqq (\HDG^{\bullet},\conj_{\bullet},\varphi_{\bullet}).
	$$
\end{example}

\begin{example}
	The above construction works analogously for log smooth scheme morphisms (\textit{i.e.}\ schemes with a fine log structure as explained in \cite{kato}).
 
	If $f\colon X\rightarrow S$ is a proper log smooth morphism of Cartier type (note that $f$ is by definition integral and thus flat; see \cite[Corollary 4.5]{kato}), then $\Omega^{1}_{X/S}$ (the sheaf of log differentials) is locally free of finite rank (see \cite[Proposition 3.10]{kato}), and because $f$ is proper and flat, the associated Hodge filtrations $\HDG_{\log}$ are perfect (use the distinguished triangle associated to the stupid truncation and conclude via induction and the fact that $f$ is proper, locally of finite presentation and flat; see \cite[0B91]{stacks-project}; this is analogous to the proof of \cite[0FM0]{stacks-project}). We also have a Cartier isomorphism in this setting (see \cite[Theorem 4.12]{kato}) (this implies, using the distinguished triangles for the conjugate filtration $\conj_{\log}$ and induction, that the conjugate filtration is perfect). Hence, we have equivalences $\varphi_{n}\colon F^{*}_{S}\gr^{n}\HDG_{\log}\xrightarrow{\lowsim} \gr^{n}\conj^{\log}$; thus analogously to the above, we can attach the structure of a derived $F$-zip to $f$ via, again,
        $$
        \left(\HDG^{\bullet}_{\log},\conj_{\bullet}^{\log},\varphi_{\bullet}\right). 
        $$ 
\end{example}

Let us consider the notion of strong $F$-zips. The condition that the filtration is given by monomorphisms seems very natural, but as Theorem~\ref{de Rham strong degen} will show, in this case we cannot expect a generalization from classical theory. In particular, the following lemma shows that perfect complexes with finite projective cohomologies are particularly easy to handle.

\begin{lem}
	\label{locally acyclic}
	Let $A$ be a ring, and let $P$ be a perfect complex over $A$ such that for all $i\in \ZZ$, the $A$-module $\pi_{i}(P)$ is finite projective. Then there exists a quasi-isomorphism $P\xrightarrow{\lowsim}\bigoplus_{n\in\ZZ}\pi_{n}(P)[n]$.
\end{lem}
\begin{proof}
	Since $P$ is perfect, we may assume that there exist  $a\leq b\in\ZZ$  such that $P$ has Tor-amplitude in $[a,b]$.  We further assume that $P$ is represented by the complex of finite projective $A$-modules
	$$
	\dots \longrightarrow 0\longrightarrow P_{b}\xrightarrow{\,\del_{b}\,} P_{b-1}\xrightarrow{\del_{b-1}}\cdots\xrightarrow{\del_{a+1}} P_{a}\longrightarrow 0\rightarrow \cdots .
	$$
	Let us define a new complex $P^{\leq a}$ given by 
	$$
		\dots \longrightarrow 0\longrightarrow P_{b}\xrightarrow{\,\del_{b}\,} P_{b-1}\xrightarrow{\del_{b-1}}\cdots\xrightarrow{\del_{a+1}} \im(\del_{a+1})\longrightarrow 0\longrightarrow \cdots .
	$$
	We get a short exact sequences of complexes
	$$
	0\longrightarrow P^{\leq a}\longrightarrow P\longrightarrow \pi_{a}(P)[a]\longrightarrow 0.
	$$
	Since $\pi_{a}(P)$ is projective, this induces a section $\pi_{a}(P)\rightarrow P_{a}$, and we can extend this to a morphism $\pi_{a}(P)[a]\rightarrow P$ which induces a section of $P\rightarrow \pi_{a}(P)[a]$. Also, this induces a retraction of $P^{\leq a}\rightarrow P$, and, in particular, $P\simeq P^{\leq a}\oplus \pi_{a}(P)[a]$. Now we claim that $P^{\leq a}$ is perfect and has Tor-amplitude in $[a+1,b]$, concluding the proof by induction on the Tor-amplitude of $P$.
 
	Indeed, note that $P^{\leq a}$ is equivalent to the complex
	$$
	\cdots \longrightarrow 0\longrightarrow P_{b}\xrightarrow{\,\del_{b}\,} P_{b-1}\xrightarrow{\del_{b-1}}\cdots\xrightarrow{\del_{a}} P_{a+1}\longrightarrow 0\longrightarrow \cdots
	$$
	which is by construction a complex of finite projective modules concentrated in degrees $[a+1,b]$, \textit{i.e.}~a perfect complex of Tor-amplitude in $[a+1,b]$.
\end{proof}

We can use Lemma~\ref{locally acyclic} to see that a morphism of perfect complexes with finite projective homotopy groups is a split monomorphism if and only if it is so on the cohomologies. This is clear since if the induced map on the cohomologies is injective, then the long exact homotopy sequence corresponding to a fiber sequence consists of short exact sequences. The projectiveness gives us retractions on the level of cohomology groups and thus a retraction on the whole complex.

\begin{rem}
\label{strong finite locally free injective}
	Let $A$ be a ring, and let $P$ and $Q$ be a perfect complexes over $A$. Further assume we have a morphism $f\colon P\rightarrow Q$ such that for all $i\in\ZZ$, the $A$-modules $\pi_{i}P$, $\pi_{i}Q$ and $\pi_{i}\cofib(f)$ are finite projective. Then we claim that $f$ is a split monomorphism if and only if $f$ is a monomorphism.
 
	Indeed, the ``only if'' direction is clear.\footnote{Let $P\xrightarrow{f} Q\xrightarrow{g} \cofib(f)$ be a cofiber sequence of $A$-modules,  and let $h\colon \cofib(f)[-1]\rightarrow P$ be the naturally induced morphism. By construction, we have $f\circ h\simeq 0$, and as $f$ is a monomorphism, this implies $h\simeq 0$ and so indeed $\ker(\pi_{i}f)=\im(\pi_{i+1}h)=0$.}
	For the ``if'' direction, let $\pi_{i}f$ be injective for all $i\in\ZZ$. By Lemma~\ref{locally acyclic}, we may assume that $P\xrightarrow{\lowsim}\bigoplus_{n\in\ZZ}\pi_{n}P[n]$ and $Q\xrightarrow{\lowsim}\bigoplus_{n\in\ZZ}\pi_{n}Q[n]$. It is enough to find retractions $g_{i}$ of $\pi_{i}f$ since then this induces a retraction of $f$.
 
	Since $\pi_{i}f$ is injective for all $i\in\ZZ$, we get short exact sequences
	\begin{equation}
	\label{eq.strong.SES}
		0\rightarrow \pi_{i}P\xrightarrow{\pi_{i}f}\pi_{i}Q\rightarrow \pi_{i}\cofib(f)\rightarrow 0.
	\end{equation}
	As $\pi_{i}\cofib(f)$ is projective, we see that the short exact sequence (\ref{eq.strong.SES}) is split, giving us the retractions of $\pi_{i}f$.
\end{rem}

Let $f\colon X\rightarrow S$ be a proper smooth morphism. In the following, we want to prove that the derived $F$-zip $\underline{Rf_{*}\Omega^{\bullet}_{X/S}}$ of Example~\ref{ex F-zip proper smooth} is strong if and only if the Hodge--de Rham spectral sequence degenerates and the $\Ocal_{S}$-modules $R^{i}f_{*}\Omega^{j}_{X/S}$ are finite locally free. 
 
We first treat the ``only if''   direction in the case where $S$ is locally Noetherian. This is done by proving the case where $S$ is a local Artinian ring and then reducing from complete Noetherian rings to local Artinian rings. After this, we insert a lemma showing us that we can reduce to the Noetherian case from the non-Noetherian one. Next, we treat the ```if'' direction more generally, for arbitrary perfect bounded strong filtrations in the derived category and the spectral sequence associated to it.

\begin{prop}
\label{lem.hodge.noeth}
	Let $f\colon X\rightarrow S$ be a proper smooth morphism of schemes, with $S$ locally Noetherian. If the filtrations of the derived $F$-zip $\underline{Rf_{*}\Omega^{\bullet}_{X/S}}$ of Example~\ref{ex F-zip proper smooth} are strong $($\textit{i.e.}\ $\HDG$ and $\conj$ are strong$)$, then the Hodge--de Rham spectral sequence degenerates and $R^{i}f_{*}\Omega_{X/S}^j$ is finite locally free for all $i,j\in\ZZ$.
\end{prop}

\begin{proof}
	Since the question is local, we may assume that $S=\Spec(A)$ is a Noetherian affine scheme. 
	 
	Let us first treat the case where $A$ is a local Artinian ring with maximal ideal $\mfr$ and residue field $k$. We can check the degeneracy of the Hodge--de Rham spectral sequence via comparing lengths of the limit term and the $E_1$-terms of the spectral sequence.  But in this case this is clear since we get a short exact sequence 
	\begin{equation*}
	\label{eq.SES}
	0\longrightarrow H^n\left(X,\sigma_{\geq i+1}\Omega_{X/S}^\bullet\right)\longrightarrow H^n\left(X,\sigma_{\geq i}\Omega_{X/S}^\bullet\right)\longrightarrow H^{n-i}\left(X,\Omega_{X/S}^i\right)\longrightarrow 0
	\end{equation*}
	for all $i\geq 0$ and $n\in \ZZ$ (by the strongness of our filtration). This implies that
        $$
        \length_AH^n\left(X,\sigma_{\geq n}\Omega_{X/S}^\bullet\right)= \length_AH^n\left(X,\sigma_{\geq i+1}\Omega_{X/S}^\bullet\right)+\length_AH^{n-i}\left(X,\Omega_{X/S}^n\right)
        $$
        for all $i\geq 0$ and $n\in \ZZ$. As $H^{n}_{\dR}(X/S)=H^{n}(X,\sigma_{\geq 0}\Omega_{X/S}^{\bullet})$, this implies inductively that 
	$$
		\length_{A}H^{n}_{\dR}(X/S) = \sum_{i\geq 0} \length_AH^{n-i}\left(X,\Omega_{X/S}^i\right) =  \sum_{p+q = n} \length_AH^{q}\left(X,\Omega_{X/S}^p\right).
	$$
		  
	Thus, for all $n\in \ZZ$, we get  the equation
	$$
	\sum_{p+q=n}\length_A E^{p,q}_{\infty} =\length_AH^{n}_{\dR}(X/S)= \sum_{p+q=n}\length_AH^{q}\left(X,\Omega_{X/S}^p\right),
	$$
	so, in particular, in this case the Hodge--de Rham spectral sequence degenerates.
 
	Analogously, we see that  the conjugate spectral sequence also degenerates (again by a length argument using the strongness of the conjugate filtration).
	Let us now show, following the arguments of the proof of \cite[Proposition 4.1.2]{DI}, that the $H^{i}(X,\Omega^{j}_{X/S})$ are finite (locally) free.
 
	Since $f$ is proper smooth, we know that $Rf_{*}\Omega^{j}_{X/S}$ is a perfect complex and its formation commutes with arbitrary base change (see \cite[0FM0]{stacks-project}). In particular, $(Rf_{*}\Omega^{j}_{X/S})\otimes_{A}k\cong Rf_{*}\Omega^{j}_{X_{k}/k}$. Therefore, we may assume that $Rf_{*}\Omega^{j}_{X/S}$ is given by a complex $K^{j}_{\bullet}$ of free $A$-modules $K^{j}_{i}=A^{h^{ij}}$, where $h^{ij}\coloneqq\dim_{k}H^{i}(X,\Omega^{j}_{X_{k}/k})$ (see \cite[0BCD]{stacks-project}). In particular, we see that
        $$
        \dim_{k}(K^{j}_{\bullet}\otimes_{A}k)=\dim_{k}H^{i}(K^{j}_{\bullet}\otimes_{A}k)=h^{ij}.
        $$
 
	Thus, we get the following inequality: 
	\begin{equation}
	\label{eq.hodge.lem.1}
	\length_{A}H^{i}\left(X,\Omega^{j}_{X/S}\right)=\length_{A} H^{i}(K^{j}_{\bullet})\leq \length_{A}K^{j}_{\bullet} = h^{ij}\length_{A}A.
	\end{equation}
	Equality in (\ref{eq.hodge.lem.1})  holds if and only if the differentials of $K^{j}_{\bullet}$ are zero, implying that $H^{i}(X,\Omega^{j}_{X/S})$ is finite (locally) free. The inequality (\ref{eq.hodge.lem.1}) also shows that equality holds if and only if
	\begin{equation}
	\label{eq.hodge.lem.2}
		\sum_{i+j=n}\length_{A}H^{i}\left(X,\Omega^{j}_{X/S}\right)=\sum_{i+j=n}h^{ij}\length_{A}A.
	\end{equation}
By the degeneracy of the Hodge--de Rham spectral sequence, the left-hand side of (\ref{eq.hodge.lem.2}) is equal to $\length_{A}H^{n}_{\dR}(X/S)$. The degeneracy of the conjugate spectral sequence, together with the Cartier isomorphism, implies that
\begin{equation}
	\label{eq.hodge.lem.3}
		\sum_{i+j=n}\length_{A}H^{i}\left(X,\Omega^{j}_{X/S}\right)=\sum_{i+j=n}\length_{A}H^{i}\left(X^{(1)},\Omega^{j}_{X^{(1)}/S}\right),
	\end{equation}
	where the notation is induced by the pullback diagram
	$$
	\begin{tikzcd}
		X^{(1)}\arrow[r,"F_{X/S}"]\arrow[d,"f^{(1)}", swap]& X\arrow[d,"f"]\\
		S\arrow[r,"\Frob_{S}"]&S\rlap{,}
	\end{tikzcd}
	$$
	where $\Frob_{S}$ denotes the Frobenius of $S$.
	We will show that
	\begin{equation}
	\label{eq.hodge.lem.4}
		\sum_{i+j=n}\length_{A}H^{i}\left(X^{(1)},\Omega^{j}_{X^{(1)}/S}\right) =\sum_{i+j=n}h^{ij}\length_{A}A
	\end{equation}
	by induction on the length of $A$, proving the lemma in the local Artinian case.
 
	If $A$ is a field, the assertion is trivial. So let us assume that $\length_{A}A=N$ and assume the lemma was shown for $N'< N$. Let $1\leq N'<N$ be such that $pN'>N$, and let us set $T=\Spec(A/\mfr^{N'})$. The choice of $N'$ assures that $\Frob_{S}$ factors through $T$. So, we get the following diagram with pullback squares: 
	$$
	\begin{tikzcd}
		X^{(1)}\arrow[rr,"F_{X/S}", bend left =15]\arrow[r,""]\arrow[d,"f^{(1)}", swap]& X_{T} \arrow[d,"f_{T}"]\arrow[r,"",hook]&X\arrow[d,"f"]\\
		S\arrow[rr,"\Frob_{S}", bend right =15,swap]\arrow[r,""]&T\arrow[r,"",hook]&S\rlap{.}
	\end{tikzcd}
	$$
	By the induction hypothesis,  $H^{i}(X_{T},\Omega_{X_{T}/T}^{j})$ is finite locally free of rank $h^{ij}$ and commutes with arbitrary base change. Therefore,  the base change along $S\rightarrow T$ also has this property. Hence, $H^{i}(X^{(1)},\Omega_{X^{(1)}/S}^{j})$ is finite locally free of rank $h^{ij}$, and therefore (\ref{eq.hodge.lem.4}) is fulfilled. But this shows that (\ref{eq.hodge.lem.3}) and thus (\ref{eq.hodge.lem.2}) are fulfilled. So, by the discussion above, we see that $H^{i}(X,\Omega_{X/S}^{j})$ is finite locally free of rank $h^{ij}$, concluding the induction.

	Now let us show how to reduce to the case where $A$ is a local Artinian ring.	
As the question is local, we can even assume that $A$ is given by a local Noetherian ring (see \cite[Proposition 7.27]{WED}). By the faithful flatness of completion (see \cite[00MC]{stacks-project}), we can assume that $A$ is a complete Noetherian local ring $(A,\mfr)$.
 
Let us first show that $H^{i}(X,\Omega_{X/S}^{j})$ is finite locally free and commutes with arbitrary base change. Let us set $A_n\coloneqq A/\mfr^n$ and $X_{n}\coloneqq X\otimes_{A}A_{n}$. Then $f_{n}\colon X_{n}\rightarrow\Spec(A_{n})$ is proper smooth, and the Hodge filtration is strong. Since $A_{n}$ is a local Artinian ring, we know that the Hodge--de Rham spectral sequence relative to $X_{n}/A_{n}$ degenerates, the $H^{i}(X_{n},\Omega^{j}_{X_{n}/A_{n}})$ are finite locally free and
$$
H^{i}\left(X_{n},\Omega^{j}_{X_{n}/A_{n}}\right)\otimes_{A_{n}}A_{n-1}\cong H^{i}\left(X_{n-1}\Omega^{j}_{X_{n-1}/A_{n-1}}\right)
$$
for $n-1\geq n$ for all $n\geq 0$. The theorem of formal functions (see \cite[02OC]{stacks-project}) now implies that
\begin{equation}
\label{eq.hodge.degen.4}
	\lim_{n\geq 0}H^{i}\left(X_{n},\Omega^{j}_{X_{n}/A_{n}}\right) \cong \lim_{n\geq 0}  H^{i}\left(X,\Omega^{j}_{X/A}\right)\otimes_{A}A_{n}.
\end{equation}
Using \cite[0D4B]{stacks-project}, we see that the left-hand side of (\ref{eq.hodge.degen.4}) is finite projective and $H^{i}(X_{n},\Omega^{j}_{X_{n}/A_{n}})\simeq H^{i}(X,\Omega^{j}_{X/S})\otimes_{A}A_{n}$. But as $A$ is complete and Noetherian and $H^{i}(X,\Omega^{j}_{X/S})$ is finite (by the perfectness of $Rf_{*}\Omega_{X/S}^{j}$), the right-hand side of (\ref{eq.hodge.degen.4}) is equal to $Rf_{*}\Omega_{X/S}^{j}$ (see \cite[00MA]{stacks-project}). In particular, it is finite locally free, and its formation commutes with arbitrary base change.
 Again, since finite modules over complete Noetherian rings are complete, we see that if the Hodge spectral sequence degenerates for all $S_n=\Spec(A/\mfr^n)$, then it degenerates on the limit of all $A/\mfr^n$, namely $A$  (see \cite[00MA]{stacks-project} and note that the $H^{i}(X,\Omega^{j}_{X/S})$ are finite projective, and therefore their formation commutes with arbitrary base change). Hence, we may assume that $A$ is a local Artinian ring, which we already discussed at the beginning.
\end{proof}

\begin{lem}
\label{lem.Hodge.red}
	Let $f\colon X\rightarrow S$ be a proper smooth morphism of schemes, with $S=\Spec(A)$ affine. Assume that the Hodge filtration $\HDG_{f}$ associated to $f$ is strong. Then there exists a Noetherian scheme $S'$ with a morphism $g\colon S\rightarrow S'$ and a proper smooth morphism of schemes $f'\colon X'\rightarrow S'$ such that the diagram
	$$
	\begin{tikzcd}
		X\arrow[r,"f"]\arrow[d,""]& S\arrow[d,"g"]\\
		X'\arrow[r,"f'"]&S'
	\end{tikzcd}
	$$
	is Cartesian. Further, the Hodge filtration $\HDG_{f'}$ associated to $f'$ is strong.
\end{lem}
\begin{proof}
  The existence of an affine scheme $S_{0}=\Spec(A_{0})$ with a morphism $g_{0}\colon S\rightarrow S_{0}$ and a proper smooth $S_{0}$-scheme $X_{0}$ such that the base change of $f_{0}\colon X_{0}\rightarrow S_{0}$ along $g_{0}$ is equal to $f$ is standard.\footnote{Use \cite[Theorem 10.69]{WED} to find an affine Noetherian $\widetilde{S}=\Spec(\widetilde{A})$ and morphisms $S\rightarrow\widetilde{S}$ and $\widetilde{X}\rightarrow \widetilde{S}$ such that the induced base change morphism identifies with $f$. Then write $S$ as a projective limit of affine $\widetilde{S}$-schemes of finite type by adjoining variables to $\widetilde{A}$ and conclude with \cite[Th\'eor\`eme (8.10.5)]{EGA4-66} and \cite[Proposition~(17.7.8)]{EGA4-67}.}
	By Proposition~\ref{lem.hodge.noeth}, the locus where $\HDG_{f_0}$ is strong is equivalently the locus where $\HDG_{f_0}$ is strong and the graded pieces are homotopy finite projective (of some type). Thus, by Lemma~\ref{lem.strong.fil.open} there is an open subscheme $U\subseteq S_{0}$ such that $g_{0}$ factors through $U$ and the Hodge filtration associated to $X_{0,U}\rightarrow U$ is strong. So, in particular, setting $S'=U$ and $X'=X_{0,U}$ finishes the proof.
\end{proof}

On the other hand, we can show that if the Hodge--de Rham spectral sequence degenerates, then the Hodge filtration is strong. Also, we do not need the particular form of the Hodge filtration, and thus we will show generally that a bounded perfect filtration with degenerate associated spectral sequence is automatically strong. Later on in Section~\ref{classical theory}, we will use this result to show that the derived $F$-zips associated to a proper smooth morphism with degenerative Hodge--de Rham spectral sequence and finite projective $E_{1}$-terms is completely determined by the underlying classical $F$-zips. So, for example in the abelian scheme case, the theory of derived $F$-zips gives us no new information and recovers the classical theory by passing to the cohomologies of the filtration (in a suitable sense as explained in Section~\ref{classical theory}).

\begin{prop}
\label{lem.Hodge.degen}
	Let $A$ be a ring and $C^{\bullet}$ be a descending bounded perfect filtration of $A$-modules. Assume that the $\pi_{i}\gr^{k}C$ are finite projective for all $i,k\in\ZZ$ and that the spectral sequence 
	$$
	E_1^{p,q} = \pi_{p+q}\gr^{p}C\Longrightarrow \pi_{p+q}\colim_{\ZZ^{\op}}C^{\bullet}
	$$
	associated to $C^{\bullet}$  degenerates. Then the filtration $C^{\bullet}$ is strong, and the statement stays true if we replace $C^{\bullet}$ with an ascending filtration.
\end{prop}
\begin{proof}
For convenience, let us set $M\coloneqq  \colim_{\ZZ^{\op}}C^{\bullet}$.

	It is enough to show\footnote{As the $C^{\bullet}$ is bounded,  there  exists an $n$ large enough such that $C^{n}\rightarrow M$ is an equivalence and thus also a monomorphism. Hence, induction indeed shows that $C^{\bullet}$ is strong.} that 
	\begin{itemize}
		\item if for any $k\in\ZZ$, the natural map $C^{k}\rightarrow M$ is a monomorphism, then the map $C^{k+1}\rightarrow C^{k}$ is a monomorphism.
	\end{itemize}
	Also, we will see that changing a descending filtration to an ascending one only changes the indices in the following proof and thus works similarly.
 
	So let us fix some $k\in \ZZ$ and assume that $C^{k}\rightarrow M$ is a monomorphism. The degeneracy of the spectral sequence implies that $E_{1}^{k,q} =\pi_{k+q}\gr^{k} C$ is naturally isomorphic to $E_{\infty}^{k,q}$ for any $q\in\ZZ$. By the construction of the spectral sequence (see \cite[Section 1.2.2]{HA}), we see that $E_{\infty}^{k,q}\cong \im(h_{q})/\im(g_{q})$, where $h_{q}\colon \pi_{k+q}C^{k}\rightarrow \pi_{k+q}M$ and $g_{q}\colon \pi_{k+q}C^{k+1}\rightarrow \pi_{k+q}M$. As $C^{k}\rightarrow M$ is a monomorphism, we see that $h_{q}$ is injective and therefore $\im(h_{q})\cong \pi_{k+q}C^{k}$. As the filtration on $M$ induced by the spectral sequence is bounded and the graded pieces are finite projective by degeneracy, we see that $\pi_{k+q}C^{k}$ is finite projective (since it is isomorphic to a filtered piece by the relation $\im(h_{q})\cong \pi_{k+q}C^{k}$).	
We claim that the morphism $f_{q}\colon\pi_{k+q}C^{k}\rightarrow \pi_{k+q}\gr^{k}C$ induced by the fiber sequence
	$$
		C^{k+1}\longrightarrow C^{k}\longrightarrow \gr^{k}C 
	$$
	is surjective for all $q\in\ZZ$.
 
	Indeed, recall from \cite[Section 1.2.2]{HA} that $E_{r}^{p,q}$ is defined by
        $$
        \im\left(\pi_{p+q}\cofib\big(C^{p+r}\longrightarrow C^{p}\big)\longrightarrow \pi_{p+q}\cofib\left(C^{p+1}\longrightarrow C^{p-r+1}\right)\right).
        $$
        As the spectral sequence is degenerate, the proof of \cite[Proposition 1.2.2.7(2)]{HA} shows that for any $r\geq 2$, we get a commutative diagram
	$$
		\begin{tikzcd}
			\pi_{k+q}\cofib\left(C^{k+r}\longrightarrow C^{k}\right)\arrow[r,""]\arrow[rd,"\phi_{r}",swap]& \pi_{k+q}\cofib\left(C^{k+1}\longrightarrow C^{k-r+1}\right)\\
			&E_{r-1}^{k,q}\arrow[u,"\vartheta_{r}",swap]\rlap{,}
		\end{tikzcd}
	$$	
	where $\phi_{r}$ is surjective and $\vartheta_{r}$ is injective. The morphism $\phi_{r}$ is induced by the natural morphism $\cofib(C^{k+r}\rightarrow C^{k})\rightarrow \cofib(C^{k+1}\rightarrow C^{k-r+1})$. Since $C^{\bullet}$ is bounded, we see that for $r$ large enough, we have $\cofib(C^{k+r}\rightarrow C^{k})\simeq C^{k}$. Thus, we get a morphism $\alpha_{q}\colon\pi_{k+q}C^{k}\rightarrow E_{1}^{k,q}= \pi_{k+p}\gr^{k}C$ that is surjective, using the $\phi_{r}$ and the degeneracy. But by construction, $\alpha_{q}$ is induced by the natural map $C^{k}\rightarrow \gr^{k}C$, and hence $\alpha_{q}$ agrees with $f_{q}$, showing the desired surjectivity.
\end{proof}

Combining all the arguments, we get the connection between the degeneracy of the Hodge--de Rham spectral sequence and the strongness of the Hodge filtration.

\begin{thm}
	\label{de Rham strong degen}
	Let $f\colon X\rightarrow S$ be a smooth proper morphism of schemes. Let us consider the Hodge--de Rham spectral sequence
	$$
	E_1^{p,q} = R^qf_*\Omega_{X/S}^p\Leftrightarrow R^{p+q}f_*\Omega_{X/S}^\bullet.
	$$
	 The filtrations of the derived $F$-zip $\underline{Rf_*\Omega^\bullet_{X/S}}$ of Example~\ref{ex F-zip proper smooth} are strong $($in the sense that $\HDG$ and $\conj$ are strong$)$ if and only if the Hodge--de Rham spectral sequence degenerates and $R^if_*\Omega_{X/S}^j$ is finite locally free for all $i,j\in\ZZ$
\end{thm}

\begin{proof}
  The ``if'' part follows from Proposition~\ref{lem.Hodge.degen}, where we use that degeneracy of the Hodge--de Rham spectral sequence implies that the conjugate spectral sequence is degenerate (see \cite[Proposition~(2.3.2)]{Katz};
  note that we need that the $R^if_*\Omega_{X/S}^j$ are finite locally free).
 
	For the ``only if'' part, we may assume that $S=\Spec(A)$ is affine, as the question is local. We can use Lemma~\ref{lem.Hodge.red}  to find a proper smooth scheme morphism $f'\colon X'\rightarrow S'$, where $S'$ is Noetherian, and a morphism $g\colon S\rightarrow S'$ such that the following diagram is a pullback diagram: 
	$$
	\begin{tikzcd}
		X\arrow[r,"f"]\arrow[d,""]& S\arrow[d,"g"]\\
		X'\arrow[r,"f'"]&S'\rlap{.}
	\end{tikzcd}
	$$
	Also, by the lemma the Hodge filtration associated to $f'$ is strong. Now if the Hodge--de Rham spectral sequence associated to $f'$ degenerates and the $R^{i}f'_{*}\Omega^{j}_{X'/S'}$ are finite locally free,  by the compatibility of the spectral sequence with base change,  we have that the Hodge--de Rham spectral sequence associated to $f$ degenerates and the $R^{i}f_{*}\Omega^{j}_{X/S}$ are finite locally free. In particular, we may assume that $S$ is Noetherian and affine. But this case was already treated in Proposition~\ref{lem.hodge.noeth}.
\end{proof}

We conclude this section by defining the substacks corresponding to a function $\tau\colon \ZZ\rightarrow \NN_{0}^{\ZZ}$. Again, we could do this by right Kan extension, but as before, there is also an \textit{ad hoc} definition that agrees with the one given by right Kan extension.

The following definitions are globalizations of the definitions given in Section~\ref{sec substacks}.

\begin{defi}\label{def:3.75}
	Let $S$ be a derived scheme, and let $\Fcal\in\PPerf(S)$. We define the function
	\begin{align*}
	\beta_\Fcal\colon S_{\cl}&\longrightarrow \NN_0^{\ZZ}\\
	s&\longmapsto (\dim_{\kappa(s)}\pi_i(\cl^*\Fcal\otimes^L_{\Ocal_{S_{\cl}}} \kappa(s)))_{i\in\ZZ},
	\end{align*}
	where $\cl\colon S_{\cl}\rightarrow S$ is the natural morphism.
 
	This function is locally upper semi-continuous in the sense that for every $s\in S_\cl$, there is a neighbourhood $U_s$ such that for any family $(k_i)_i\in\NN_0^\ZZ$, the set $\beta_{\Fcal|U_{s}}^{-1}(([0,k_i])_{i})$ is open (see \cite[0BDI]{stacks-project}).
\end{defi}

\begin{remark}
	Note that in Definition~\ref{def:3.75}, we implicitly assume that $\cl^\ast \Fcal$ is a perfect complex of quasi-coherent $\Ocal_{S_{\cl}}$-module. This makes sense as $\cl^{*}\Fcal$ is in $\Dcal_{\textup{qc}}(S_{\cl})$ by Proposition~\ref{derived cat is kan ext} and is by definition perfect.
\end{remark}

\begin{defrem}
Let $S$ be a derived scheme. Let $\Fline\coloneqq (C^{\bullet},D_{\bullet},\phi,\varphi_{\bullet})$ be a derived $F$-zip over $S$. Consider the function 
	\begin{align*}
		\beta_{\Fline}\colon s\longmapsto (k\longmapsto \beta_{\gr^kC}(s))
	\end{align*}
	from $S_\cl$ to functions $\ZZ\rightarrow \NN_0^\ZZ$.
	\begin{enumerate}
		\item The function $\beta_{\Fline}$ is called the \textit{type of the derived $F$-zip $\Fline$}.
		\item Let $\tau\colon \ZZ\rightarrow \NN_0^\ZZ$ be a function with finite support. We say that \textit{$\Fline$ has type at most $\tau$} if for all $s\in S_{\cl}$, we have $\beta_{\Fline}(s)\leq \tau$ (again, the relation is given pointwise as functions $\ZZ\times\ZZ\rightarrow\NN_{0}$).
	\end{enumerate}
	 
	Further, for any $s\in S_{\cl}$, there exist a quasi-compact open (resp.\ locally closed) neighbourhood $U_s$ of $s$ and a function $\tau\colon\ZZ\rightarrow\NN_{0}^{\ZZ}$ with finite support, such that $\beta_{\Fline |U_s}\leq \tau$ in the sense above (resp.\ $\beta_{\Fline |U_s}$ is constant and equal to $\tau$) (this follows from Lemma~\ref{upper semi-cont qc}).
\end{defrem}

\begin{defrem}
	\label{defi Euler char2}
	Let $S$ be a derived scheme. Let $\Fline\coloneqq (C^{\bullet},D_{\bullet},\phi,\varphi_{\bullet})$ be a derived $F$-zip over $S$. Let us look at the function 
	\begin{align*}
	\chi_k(\Fline)\colon S_{\cl}&\longrightarrow \ZZ\\
	s&\longmapsto \chi(\cl^{*}\gr^kC \otimes_{\Ocal_{S_{\cl}}}\kappa(s)),
	\end{align*}
	where $\cl\colon S_{\cl}\rightarrow S$ is the natural morphism.
	This is a locally constant function (see \cite[0B9T]{stacks-project}). Since the filtrations on derived $F$-zips are locally bounded, we  know that the function $\chi_{\Fline}\colon s\mapsto (k\mapsto \chi_k(\Fline)(s))$ is also locally constant as a map from $S_{\cl}$ to functions $\ZZ\rightarrow \ZZ$ with finite support. We call $\chi_{\Fline}$ the \textit{Euler characteristic} of $\Fline$.
 
	If $\tau\colon \ZZ\rightarrow \ZZ$ is a function with finite support, we say that \textit{$\Fline$ has Euler characteristic $\tau$} if $\chi_{\Fline}$ is constant with value $\tau$.
\end{defrem}

\begin{defi}
	Let $S$ be a derived scheme.
	\begin{enumerate}
		\item Let $M\in\PPerf(S)$ be a perfect module over $S$. Fix a map $r\colon \ZZ\rightarrow \NN_0$. We call $M$ \textit{homotopy finite locally free of rank $r$} if for each scheme $f\colon X\rightarrow S_{\cl}$, the $\Ocal_{X}$-module $\pi_i(f^*\cl^{*}M)$ is finite locally free of rank $r_i$, where we say that it is finite locally free of rank $0$ if it is isomorphic to $0$.
		\item Let $\Fline\coloneqq (C^\bullet,D_\bullet,\varphi_\bullet)$ be a derived $F$-zip over $S$ and $\tau\colon \ZZ\rightarrow \NN_0^{\ZZ}$ be a function with finite support. We say that $\Fline$ is \textit{homotopy finite locally free of type $\tau$} if for all $i\in\ZZ$, we have that $\gr^iC$ is homotopy finite locally free of rank $\tau(i)$.
	\end{enumerate}
\end{defi}

\begin{defi}
	Let $\tau\colon\ZZ\rightarrow\NN_0^\ZZ$ be a function with finite support, and let $S$ be a derived $R$-scheme. We define $\FZip_{\infty,R}^{\leq\tau}(S)$ $($resp.\ $\FZip_{\infty,R}^{\tau}(S)$, $\FZip_{\infty,R}^{\chi=\tau}(S))$ as the full subcategory of derived $F$-zips over $S$ of type at most $\tau$ $($resp.\ homotopy finite locally free of type $\tau$, of Euler characteristic $\tau)$. 

\end{defi}

\begin{lem}
	Let $\tau\colon\ZZ\rightarrow\NN_0^\ZZ$ be a function with finite support.
	Let $\RR\FZip^{\leq\tau}_{\infty,R}$ $($resp.\ $\RR\FZip_{\infty,R}^{\tau}$, $\RR\FZip_{\infty,R}^{\chi=\tau})$ be the  right Kan extension of $\FZip^{\leq\tau}_{\infty,R}$ $($resp.\ $\FZip_{\infty,R}^{\tau}$, $\FZip_{\infty,R}^{\chi=\tau})$ along the Yoneda embedding $\AniAlg{R}\hookrightarrow\Pcal(\AniAlg{R}^{\op})^{\op}$. Then for any derived scheme $S$ over $R$,  the natural morphism induced by base change
	\begin{align*}
		&\FZip^{\leq\tau}_{\infty,R}(S)\longrightarrow\RR\FZip^{\leq\tau}_{\infty,R}(S) \\ &\textup{ (resp.\ }\FZip^{\tau}_{\infty,R}(S)\longrightarrow\RR\FZip^{\tau}_{\infty,R}(S),\ \FZip^{\chi=\tau}_{\infty,R}(S)\longrightarrow\RR\FZip^{\chi =\tau}_{\infty,R}(S) \textup{)}
	\end{align*}
	is an equivalence.
\end{lem}
\begin{proof}
	The proof is completely analogous to that of Lemma~\ref{Fzips global given by affine}. We only need to verify that the properties ``has type $\leq \tau$'', ``is homotopy finite locally free of type $\tau$'' and ``has Euler characteristic $\tau$'' can be checked on an affine open cover of $S$, but this is clear.
\end{proof}

\begin{defrem}
  Let $\tau\colon\ZZ\rightarrow\NN_0^\ZZ$ be a function with finite support, and let $S$ be a derived $R$-scheme. We define 
  $$
  \FZip_{R}^{\leq\tau}\colon \Pcal(\AniAlg{R}^{\op})^{\op}\longrightarrow \SS
  $$
  (resp.\ $\FZip_{R}^{\tau}$, $\FZip_{R}^{\chi=\tau}$) as the right Kan extension of $\FZip_{R}^{\leq\tau}$ (resp.\ $\FZip_{R}^{\tau}$, $\FZip_{R}^{\chi=\tau}$) along the Yoneda embedding $\AniAlg{R}\hookrightarrow \Pcal(\AniAlg{R}^{\op})^{\op}$.

	By Remark~\ref{rem.sheaf.kan}, these define fpqc sheaves and define subsheaves of $\FZip_{R}$.
\end{defrem}
 
\subsection{Perfect complexes on the pinched projective space}

\label{app perf on pinched}

In this subsection, we want to understand the perfect complexes on the pinched projective space, \textit{i.e.}\ the $\infty$-category $\QQCoh_{\perf}(\Xfr_{S})$ for a scheme $S$ in characteristic $p>0$. (See the appendix for the notation). We  show in the appendix that the vector bundles on $\Xfr$ are precisely the $F$-zips. In this section, we will see a similar result for derived $F$-zips; namely, we will show that $\QQCoh_{\perf}(\Xfr_{S})\simeq \FZip_{\infty,R}(S)$.
 
As explained in \cite[Proposition 4.1.1]{halpern}, quasi-coherent sheaves on $[\AA_{S}^{1}/\Gm_{,S}]$ are the same as $\ZZ$-indexed diagrams of quasi-coherent $\Ocal_{S}$-modules, so a chain of morphisms of $\Ocal_{S}$-modules
$$
\cdots\longrightarrow \Fcal_{i}\longrightarrow \Fcal_{i+1}\longrightarrow \Fcal_{i+2}\longrightarrow \cdots
$$
(for vector bundles, we showed the computation behind it in Theorem~\ref{F-zips as VB}). Equivalently, the category of quasi-coherent sheaves on $[\AA_{S}^{1}/\Gm_{,S}]$ is given by the category of graded $\Ocal_{S}$-modules together with an endomorphism of degree $1$ (this endomorphism is induced by multiplication with $X$).
This gives a description of the category of chain complexes and so
$$
D_{\textup{qc}}\left(\left[\AA_{S}^{1}/\Gm_{,S}\right]\right)\simeq\Fun(\ZZ,D_{\textup{qc}}(S)).
$$
 
Now let us endow the abelian category of chain complexes of quasi-coherent modules over $[\AA_{S}^{1}/\Gm_{,S}]$ with the usual model structure and $\Fun(\ZZ,\Ch(\QQCoh(S)))$ with the pointwise model structure. The natural identification of the categories explained above induces a Quillen equivalence and therefore an equivalence of $\infty$-categories
$$
\Dcal_{\textup{qc}}\left(\left[\AA_{S}^{1}/\Gm_{,S}\right]\right)\simeq \Fun(\ZZ,\Dcal_{\textup{qc}}(S)).
$$
In fact, this equivalence can be upgraded naturally to a symmetric monoidal equivalence.\newline
 
We will compute $\QQCoh_{\perf}(\Xfr_{S})$ via descent. For this, we want to understand the pullback of perfect complexes along $[\lbrace 0\rbrace/\Gm_{,\FF_{p}}]\rightarrow[\AA^{1}_{S}/\Gm_{,\FF_{p}}]$. We will need the following proposition to ease computation.

\begin{lem}
\label{lem restr}
Let $S$ be a scheme. Further, let $G$ be a group scheme over $S$, and let $X$ and $Z$ be $S$-schemes with a smooth $G$-action denoted by $a_{X}$ and $a_{Z}$, respectively.
Let $i\colon Z\hookrightarrow X$ be a $G$-equivariant closed immersion such that the diagram 
$$
\begin{tikzcd}
	G\times_{S} Z\arrow[r,"\id\times i"]\arrow[d,"a_{Z}"]& G\times_{S}X\arrow[d,"a_{X}"]\\
	Z\arrow[r,"i",hook]&X
\end{tikzcd}
$$
is Cartesian. Then the restriction functor $i_{*}\colon\Dcal_{\textup{qc}}(\left[Z/G\right])\rightarrow \Dcal_{\textup{qc}}(\left[X/G\right])$ is conservative.
\end{lem}
\begin{proof}
	Using the Barr resolution of $\left[Z/G\right]$ and $\left[X/G\right]$, we get a commutative diagram 
	$$
	\begin{tikzcd}
		\dots\arrow[r,""]\arrow[r,"",shift right = 0.3em,swap]\arrow[r,"",shift left = 0.3em]\arrow[d,""]&G\times_{S} Z\arrow[r,"a_{Z}",shift right = 0.3em,swap]\arrow[r,"p",shift left = 0.3em]\arrow[d,"\id\times i"]& Z\arrow[d,"i"]\\
		\dots\arrow[r,"",shift right = 0.3em,swap]\arrow[r,"",shift left = 0.3em]\arrow[r,""]&G\times_{S} X\arrow[r,"a_{X}",shift right = 0.3em,swap]\arrow[r,"p",shift left = 0.3em]&X\rlap{,}
	\end{tikzcd}
	$$
	where the vertical arrows are all closed immersion. The derived pushforward along the vertical arrows is therefore conservative (see \cite[08I8]{stacks-project}). As the derived $\infty$-category with quasi-coherent cohomology satisfies fpqc descent (see Remark~\ref{descent qcoh}), this concludes the proof.
\end{proof}

\begin{prop}
\label{prop.B3}
Let $\Gm_{,S}$ act on $\AA^{1}_{S}$ by multiplication of degree $1$ $($resp.\ $-1)$ and trivially on the closed subscheme $\lbrace 0\rbrace\subseteq\AA^{1}_{S}$. We denote by $f\colon[\lbrace0\rbrace /\Gm_{,S}]\rightarrow [\AA_{S}^{1}/\Gm_{,S}]$ the naturally induced morphism. Further, let $M\in \Dcal_{\textup{qc}}([\AA_{S}^{1}/\Gm_{,S}])$. Then we have
$$
Lf^{*}M\simeq \bigoplus_{i\in\ZZ}\gr^{i}M,
$$
where we consider $M$ as an element of $\Fun(\ZZ,\Dcal_{\textup{qc}}(S))$.
\end{prop}

\begin{proof}
We first claim that it is enough to show that after restricting to $\Dcal_{\textup{qc}}([\AA_{S}^{1}/\Gm_{,S}])$, we have $Lf^{*}M\simeq \bigoplus_{i\in\ZZ}\gr^{i}M$. In particular, to compute $Lf^{*}$, we may resolve $\Ocal_{S}$ as a K-flat complex in $\Dcal_{\textup{qc}}([\AA_{S}^{1}/\Gm_{,S}])$, which is straightforward as we are going to see. 

Indeed, let $F\colon \Dcal_{\textup{qc}}([\AA_{S}^{1}/\Gm_{,S}])\rightarrow \Dcal_{\textup{qc}}([\lbrace 0\rbrace /\Gm_{,S}])$ denote the graded functor $M\mapsto \bigoplus_{i\in\ZZ} \gr^{i}M$. Assume that $f_{*}Lf^{*}M\simeq f_{*}F(M)$. Then by adjunction, we get a morphism $\alpha \colon Lf^{*}M\rightarrow F(M)$ in $\Dcal_{\textup{qc}}([\lbrace 0\rbrace /\Gm_{,S}])$ induced by the identity on $Lf^{*}M$. But $\alpha$ is an equivalence by Lemma~\ref{lem restr} (the smoothness of the $\Gm_{,S}$-action on $\AA^{1}_{S}$ follows from the proof of Lemma~\ref{lem.action.smooth}), concluding the reduction step.

We will give a proof in the case where $\Gm_{,S}$ acts by multiplication of degree $1$. The degree $-1$ case is completely analogous;  we will note the places where the proof changes.

Important for us is that, as explained above, a quasi-coherent $\Gm_{,S}$-equivariant $\Ocal_{\AA^{1}_{S}}$-module $\Fcal$ is equivalently a graded $\Ocal_{S}$-module $\Fcal=\bigoplus_{i\in\ZZ} \Fcal^{i}$ together with an endomorphism $\Fcal\rightarrow\Fcal$ of degree $1$ (resp.~$-1$) that is induced by multiplication with $X$.

	The category of quasi-coherent modules over $[\lbrace0\rbrace /\Gm_{,S}]$ is analogously equivalent to the category of quasi-coherent graded $\Ocal_{S}$-modules. As $f$ is equivariant, we get a pullback functor $f^{*}$ from the category of cochain complexes of graded $\Ocal_{S}$-modules with endomorphism of degree $1$ (resp.\ $-1$) to the category of cochain complexes of graded $\Ocal_{S}$-modules. Let us write $M$ as $(M_{\bullet},\del_{\bullet})=(\bigoplus_{i\in\ZZ} M^{i}_{\bullet},\bigoplus_{i\in\ZZ}\del^{i}_{\bullet})$ (a chain complex of graded $\Ocal_{S}$-modules) together with an endomorphism $X\colon M\rightarrow M$ that is induced by multiplication with $X$. The complex $f^{*}M$ is given by $M\otimes_{\Ocal_{\AA^{1}_{S}}} \Ocal_{S}$, where we identify $\Ocal_{S}$ with $\Ocal_{S}[X]/(X)$ (which endows $\Ocal_{S}$ with a trivial grading and degree $1$ (resp.\ $-1$) endomorphism given by $0$).
	
	To compute $Lf^{*}M$, it is enough to find a cochain complex $P\in\Dcal_{\textup{qc}}([\AA_{S}^{1}/\Gm_{,S}])$ with a quasi-isomorphism $P\xrightarrow{\lowsim} \Ocal_{S}$ in $\Dcal_{\textup{qc}}([\AA_{S}^{1}/\Gm_{,S}])$, such that the functor $P\otimes_{\Ocal_{\AA^{1}_{S}}}\!\!-$ in the category of cochain complexes of $\Gm_{, S}$-equivariant $\Ocal_{\AA^{1}_{S}}$-modules is exact. Then $Lf^{*}M$ is equivalent to $M\otimes_{\Ocal_{\AA^{1}_{S}}}\!\!P$. We claim that $P$ is naturally given by the Koszul complex of $\Ocal_{S}$.
 
        Indeed, a flat resolution of $\Ocal_{S}$ is given by the complex $P^{\bullet}$ that is zero everywhere except in degrees $-1$ and $0$, where it is given by the morphism
        $$
        \Ocal_{S}[X]\xrightarrow{\,\cdot X\,}\Ocal_{S}[X].
        $$
        Now $\Ocal_{S}[X]$ is endowed with the obvious grading together with an endomorphism of degree $1$ (resp.\ $-1$) by multiplication with $X$. Then $P^{\bullet}$ becomes an element of $\Dcal_{\textup{qc}}([\AA_{S}^{1}/\Gm_{,S}])$ by shifting the grading of $P^{-1}$ by $-1$ (resp.\ $1$). Note that the functor $P^{\bullet}\otimes_{\Ocal_{\AA^{1}_{S}}}-$ is exact, and therefore we have
        $$
        Lf^{*}M\simeq P^{\bullet}\otimes_{\Ocal_{\AA^{1}_{S}}}\!\!M.
        $$
 
        Now let us explicitly compute $Lf^{*}M$. By the definition of the tensor product of chain complexes, we deduce from the above that $M\otimes_{\Ocal_{\AA^{1}_{S}}}\!\!P^{\bullet}$ is equivalent to the complex
        $$
        (M_{n+1}\oplus M_{n},\iota_{n})_{n\in\ZZ},
        $$
        where the differentials $\iota_{n}$  are given by 
$$
\begin{pmatrix}
\del_{n+1}&0\\
(-1)^{n}X&\del_{n}
\end{pmatrix}.
$$
The induced grading is given by $(M_{n+1}\oplus M_{n})^{i}=M_{n+1}^{i-1}\oplus M_{n}^{i}$ (resp.\ $(M_{n+1}\oplus M_{n})^{i}=M_{n+1}^{i+1}\oplus M_{n}^{i}$).

Now let us analyze the graded pieces $\gr^{i}M$. As explained above, we can also consider
$$
\cdots\longrightarrow (M^{i-1}_{\bullet},\del^{-1}_{\bullet})\xrightarrow{\,\cdot X\,} (M^{i}_{\bullet},\del^{i}_{\bullet})\longrightarrow \cdots
$$
as a filtration in $\Dcal_{\textup{qc}}(S)$. Let us calculate $\cofib((M^{i-1}_{\bullet},\del^{i-1}_{\bullet})\xrightarrow{\cdot X} (M^{i}_{\bullet},\del^{i}_{\bullet}))$. We can do so by calculating the cone of multiplication with $X$, which is given by
$$
\left(M^{i-1}_{n+1}\oplus M^{i}_{n},\iota_{n}^{i}\right)_{n\in\ZZ},
$$
where the differentials $\iota_{n}^{i}$ are, up to equivalence, given by 
$$
\begin{pmatrix}
\del^{i-1}_{n+1}&0\\
(-1)^{n}X&\del^{i}_{n}
\end{pmatrix}. 
$$
(For $\cofib((M^{i+1}_{\bullet},\del^{i-1}_{\bullet})\xrightarrow{\cdot X} (M^{i}_{\bullet},\del^{i}_{\bullet}))$, this is analogous; one just changes the indices).
 
Finally, these constructions imply that $Lf^{*}M\simeq \bigoplus_{i\in\ZZ} \gr^{i}M$ in $\Dcal_{\textup{qc}}([\AA_{S}^{1}/\Gm_{,S}])$.
\end{proof}

We are finally ready to compute the perfect complexes on $\Xfr_{S}$ for any scheme $S$. We will describe it as a full subcategory of 
\begin{align*}
		 \Ccal(S)\coloneqq(\Fun(\ZZ^{\op},\Dcal_{\textup{qc}}(S))&\times_{\colim,\Dcal_{\textup{qc}}(S),\colim} \Fun(\ZZ,\Dcal_{\textup{qc}}(S)))\\ &\times_{(\bigoplus(\gr^{i})^{(1)},\bigoplus\gr^{i}),\Dcal_{\textup{qc}}(S)\times\Dcal_{\textup{qc}}(S),\dom\times\codom}\Fun(\Delta^{1},\Dcal_{\textup{qc}}(S)).
\end{align*}

\begin{thm}
\label{thm.B1}
	For any scheme $S$, the $\infty$-category $\QQCoh_{\perf}(\Xfr_{S})$ is equivalent to the full subcategory in $\Ccal(S)$ of tuples $(C^{\bullet},D_{\bullet},\phi,\varphi)$, where
\begin{itemize}
	\item $C^{\bullet}\in\Fun(\ZZ^{\op},\Dcal_{\textup{qc}}(S))$ is  such that all  $\gr^{i}C$ and $\bigoplus_{i\in\ZZ}\gr^{i}C$ are perfect;  
	\item $D_{\bullet}\in\Fun(\ZZ,\Dcal_{\textup{qc}}(S))$ is such that all $\gr^{i}D$ and $\bigoplus_{i\in\ZZ}\gr^{i}D$ are perfect; 
	\item $\phi\colon \colim_{\ZZ^{\op}}C^{\bullet}\xrightarrow{\lowsim}\colim_{\ZZ} D^{\bullet}$ is an equivalence;  and 
	\item $\varphi\colon \bigoplus_{i\in\ZZ}\gr^{i}C^{(1)}\xrightarrow{\lowsim}\bigoplus_{i\in\ZZ}\gr^{i}D$ is an equivalence.
\end{itemize}
\end{thm}
\begin{proof}
	By Proposition~\ref{right kan of sheaf}, we can use the Barr resolution of $\Xfr_{S}$ to see that  
\begin{equation}
\label{eq.barr.Xfr}
\begin{tikzcd}
	\QQCoh_{\perf}(\Xfr_{S})\simeq \lim\big(\QQCoh_{\perf}(X_{S})\arrow[r,"",shift left = 0.3em]\arrow[r,"",shift right = 0.3em]&\QQCoh_{\perf}\left(X_{S}\times_{\FF_{p}}\Gm_{,\FF_{p}}\right)\arrow[r,"",shift left = 0.3em]\arrow[r,"",shift right = 0.3em]\arrow[r,""]&\cdots\big)\rlap{.}
\end{tikzcd}
\end{equation}
We can again use Proposition~\ref{right kan of sheaf} to see that we have a limit diagram of the form
$$
\begin{tikzcd}
	\QQCoh_{\perf}(X_{S})\arrow[r,""]\arrow[d,""]& \QQCoh_{\perf}\left(\PP_{S}^{1}\right)\arrow[d,""]\\
	\QQCoh_{\perf}(\lbrace\infty\rbrace)\arrow[r,""]&\QQCoh_{\perf}(\lbrace \infty\rbrace)\times\QQCoh_{\perf}(\lbrace 0\rbrace)\rlap{.}
\end{tikzcd}
$$
Since the $\Gm_{,\FF_{p}}$-action on $X_{S}$  is induced by the pushout of the $\Gm_{,\FF_{p}}$-actions on $\PP^{1}_{S}$, $\lbrace\infty\rbrace$ and $\lbrace\infty\rbrace\amalg\lbrace 0\rbrace$ (see (\ref{eq.pushout})), we see that an object $X\in\QQCoh_{\perf}(\Xfr_{S})$ corresponds to a tuple $(M,N,\varphi)$, where 
\begin{itemize}
	\item $M\in\QQCoh_{\perf}([\PP^{1}_{S}/\Gm_{,\FF_{p}}])$,
	\item $N\in\QQCoh_{\perf}([\lbrace\infty\rbrace/\Gm_{,\FF_{p}}])$, and
	\item $\varphi$ is an equivalence of the images of $M$ and $N$ in
          $$
          \QQCoh_{\perf}\left(\left[\lbrace\infty\rbrace/\Gm_{,\FF_{p}}\right]\right)\times\QQCoh_{\perf}\left(\left[\lbrace 0\rbrace/\Gm_{,\FF_{p}}\right]\right).
         $$
\end{itemize}    
Using the standard cover of $\PP^{1}_{S}$ by affine lines and the discussion at the beginning of this subsection, we see that $M\in\QQCoh_{\perf}([\PP^{1}_{S}/\Gm_{,\FF_{p}}])$ is equivalently given by a tuple $(C^{\bullet},D_{\bullet},\phi)$, where $C^{\bullet}\in\Fun(\ZZ^{\op},\Dcal_{\textup{qc}}(S))$ is perfect, $D_{\bullet}\in\Fun(\ZZ,\Dcal_{\textup{qc}}(S))$ is perfect and  $\phi\colon \colim_{\ZZ^{\op}}C^{\bullet}\xrightarrow{\lowsim}\colim_{\ZZ} D^{\bullet}$ is an equivalence.\footnote{By construction, $\PP^{1}_{S}$ is the pushout of the maps $\Spec(\Ocal_{S}[X,X^{-1}])\xrightarrow{x\mapsto x^{-1}} \Spec(\Ocal_{S}[X])$ and  $\Spec(\Ocal_{S}[X,X^{-1}])\xrightarrow{x\mapsto x} \Spec(\Ocal_{S}[X])$,
where the maps are given on $T$-valued points. So again, the description of $\QQCoh_{\perf}(\PP^{1}_{S})$ follows from Proposition~\ref{right kan of sheaf} and the fact that the derived pullback along open immersions is given by the usual pullback.} 
 Further (as explained in the proof of Proposition~\ref{prop.B3}), $\QQCoh_{\perf}([\lbrace\infty\rbrace/\Gm_{,\FF_{p}}])$ consists of perfect chain complexes of graded $\Ocal_{S}$-modules. Also, we have seen in Proposition~\ref{prop.B3} that the image of $(C^{\bullet},D_{\bullet})$ in the product $\QQCoh_{\perf}([\lbrace\infty\rbrace/\Gm_{,\FF_{p}}])\times\QQCoh_{\perf}([\lbrace 0\rbrace/\Gm_{,\FF_{p}}])$ is equivalent to $\left(\bigoplus_{i} \gr^{i}C,\bigoplus_{i} \gr^{i}D\right)$.
 Lastly, we want to note that by \cite[Proposition 2.45]{GW}, an element in $\Fun(\ZZ^{\op},\Dcal_{\textup{qc}}(S))$ (resp.\ $\Fun(\ZZ,\Dcal_{\textup{qc}}(S))$) is perfect if and only if each graded piece is perfect.
 
Combining all of this, we get the desired description of $\QQCoh_{\perf}(\Xfr_{S})$ as a full subcategory of $\Ccal(S)$.
\end{proof}

It is clear by construction that $\FZip(S)$ is a full subcategory of $\QQCoh_{\perf}(\Xfr_{S})$. But let us show that we have an equivalence. This will follow immediately if we can show that the filtrations associated to an element in $\QQCoh_{\perf}(\Xfr_{S})$ are locally bounded.

\begin{lem}
\label{lem.B2}
	Let $S$ be a scheme, and let $F\in\Fun(\ZZ,\Dcal_{\textup{qc}}(S))$ be an ascending filtration such that $\gr^{i}F$ and $\bigoplus_{i\in\ZZ}\gr^{i}F$ are perfect $\Ocal_{S}$-modules. Then $F$ is locally bounded and perfect.
 
	The assertion stays true if we replace $\ZZ$ with $\ZZ^{\op}$.
\end{lem}
\begin{proof}
	As this is a local question, we may assume that $S=\Spec(A)$ is affine. Fiberwise, the question is clear since a perfect complex over a field is quasi-isomorphic to a finite direct sum of finite-dimensional vector spaces sitting in one degree. For every point $s\in S$, we can find an open neighbourhood $U_{s}$ around $s$ such that only finitely many $\gr^{i}F$ are non-zero (see \cite[0BCD]{stacks-project}). As $S$ is quasi-compact, we deduce the lemma.\footnote{Inductively, any bounded filtration with perfect graded pieces is perfect.} 
\end{proof}

\begin{cor}
\label{cor.fzip.perf}
	Let $R$ be an $\FF_{p}$-algebra and $S$ an $R$-scheme. Then we have $$\FZip_{\infty,R}(S)\simeq \QQCoh_{\perf}(\Xfr_{S}).$$
\end{cor}
\begin{proof}
  This follows immediately by combining Theorem~\ref{thm.B1}, Lemma~\ref{lem.B2} and that finite direct sums in $\Dcal_{\textup{qc}}(S)$ are the same as finite products as $\Dcal_{\textup{qc}}(S)$ is  stable.\footnote{The equivalence between finite direct sums and products shows that for an element $(C^{\bullet},D_{\bullet},\phi,\varphi)\in\QQCoh_{\perf}(\Xfr_{S})$, the equivalence $\varphi\colon \bigoplus_{i\in\ZZ}\gr^{i}C^{(1)}\xrightarrow{\lowsim}\bigoplus_{i\in\ZZ}\gr^{i}D$ is equivalently given by equivalences $\varphi_{i}\colon\gr^{i}C^{(1)}\xrightarrow{\lowsim}\gr^{i}D$.}
\end{proof}

\section{Connection to classical theory}
\label{classical theory}
Again, in the following $R$ will be an $\FF_{p}$-algebra.

\subsection{Derived $\boldsymbol{F}$-zips with degenerating spectral sequences}
In Lemma~\ref{F-zip in classical}, we showed that classical $F$-zips can be included in the theory of derived $F$-zips. But what if the homotopy groups associated to the graded pieces of a derived $F$-zip are finite locally free and the associated spectral sequences degenerate (see Definition~\ref{defi.fzip.degen})?  Then we would expect that we have a functor $\pi_n\colon \Xcal\rightarrow \cl\FZip$, where $\Xcal$ is a suitable substack of $t_0\FZip$, given by sending the underlying module to its $\supth{n}$ homotopy group and looking at the associated filtrations.
 
In fact, we will show that there is even more. For a smooth proper scheme morphism with degenerating Hodge--de Rham spectral sequence and finite locally free cohomologies, we get an $F$-zip. Important here is that the graded pieces and de Rham cohomology are finite locally free. For filtrations in our sense, we also get spectral sequences; \textit{i.e.}\ for some $R$-algebra $A$ and bounded perfect filtrations $C^{\bullet}\in\Fun(\ZZ^{\op},\Dcal(A))$ and $D_{\bullet}\in\Fun(\ZZ,\Dcal(A))$, we have  spectral sequences  
$$
E_{1}^{p,q}=\pi_{p+q}(\gr^{p}C)\Longrightarrow \pi_{p+q}\colim_{\ZZ^{\op}} C^{\bullet} \quad\textup{and}\quad E_{1}^{p,q}=\pi_{p+q}(\gr^{p}D)\Longrightarrow \pi_{p+q}\colim_{\ZZ} D_{\bullet}
$$
(see \cite[Proposition 1.2.2.14]{HA}). If we assume that the $\pi_{p+q}(\gr^{p}C)$ are finite projective and that the above spectral sequences are degenerate, we can associate, for any $n\in\ZZ$, a classical $F$-zip to a derived $F$-zip $\Fline\coloneqq (C^{\bullet},D_{\bullet},\phi,\varphi_{\bullet})$ via
$$
\Fline\mapsto \pi_{n}\Fline\coloneqq\left(\pi_{n}\left(\colim_{\ZZ^{\op}}C^{\bullet}\right),\widetilde{C}^{\bullet},\widetilde{D}_{\bullet},\pi_{n}\varphi_{\bullet}\right),
$$ 
where $\widetilde{C}^{\bullet}$, resp.\ $\widetilde{D}_{\bullet}$, is the filtration associated to the spectral sequences induced by $C^{\bullet}$, resp.\ $D_{\bullet}$. Let us verify that $(\pi_{n}(\colim_{\ZZ^{\op}}C^{\bullet}),\widetilde{C}^{\bullet},\widetilde{D}_{\bullet},\pi_{n}\varphi_{\bullet})$ is a classical $F$-zip.

For convenience, let us set $M\coloneqq \pi_{n}(\colim_{\ZZ^{\op}}C^{\bullet})$. By definition, we also have $\pi_{n}\colim_{\ZZ}{D}\cong M$. First of all, note that both $\widetilde{C}^{\bullet}$ and $\widetilde{D}_{\bullet}$ are finite and that  by the degeneracy of the spectral sequences,  their graded pieces are equivalent to
$$
\gr^{i}_{\widetilde{C}} M = \pi_{n}(\gr^{i}C),\quad \gr^{i}_{\widetilde{D}}M = \pi_{n}(\gr^{i}D).
$$
By homotopy finite projectiveness,  all graded pieces of $\widetilde{C}^{\bullet}$ and $\widetilde{D}_{\bullet}$ are finite projective.\footnote{Note that the homotopy finite projectiveness of $\pi_{n}\gr^{i}C$ implies the compatibility with base change (along Frobenius), by Lemma~\ref{betti numbers locally closed}\eqref{bnlc-3}, and so $\pi_{n}\gr^{i}D\cong \pi_{n}(\gr^{i}C^{(1)})\cong \pi_{n}(\gr^{i}C)^{(1)}$ is finite projective.} Since the filtrations are bounded, we see that the pieces of the filtrations are also finite projective, and thus  $M$ also is. The only thing left to see is that $\pi_{n}\varphi_{i}$ induce isomorphisms $(\gr^{i}_{\widetilde{C}} M)^{(1)}\xrightarrow{\lowsim} \gr^{i}_{\widetilde{D}} M$. But this again follows from the degeneracy of the spectral sequences (resp.\ the description above induced by the degeneracy) and the fact that homotopy finite projectiveness implies compatibility with base change (along Frobenius) by Lemma~\ref{betti numbers locally closed}\eqref{bnlc-3}.
 
Further, if the derived $F$-zip $\Fline$ is homotopy finite projective of some type $\tau$, then $\pi_{n}\Fline$ has type $\tau_{n}\colon k\mapsto \tau(k)_{n}$.

Moreover, using the arguments in the proof of Theorem~\ref{de Rham strong degen}, we see that a derived $F$-zip homotopy finite projective of some type with degenerating spectral sequences associated to the filtrations (as above) is automatically strong. Let us make everything we said above more precise.

\begin{defi}
\label{defi.fzip.degen}
	Let $A$ be an $R$-algebra. A derived $F$-zip $(C^{\bullet},D_{\bullet},\phi,\varphi_{\bullet})$ is called \textit{degenerate} if the spectral sequences
	$$
		E_{1}^{p,q}=\pi_{p+q}(\gr^{p}C)\Longrightarrow \pi_{p+q}\colim_{\ZZ^{\op}} C^{\bullet}\quad\textup{and}\quad E_{1}	^{p,q}=\pi_{p+q}(\gr^{p}D)\Longrightarrow \pi_{p+q}\colim_{\ZZ} D_{\bullet}
	$$
associated to the filtrations (see \cite[Proposition 1.2.2.14]{HA}) degenerate.
\end{defi}

\begin{lem}
\label{degen implies strong}
	Let $A$ be an $R$-algebra, and let $\tau\colon\ZZ\rightarrow\NN_{0}^{\ZZ}$ be a function with finite support. If a derived $F$-zip $\Fline$ over $A$ is homotopy finite projective of type $\tau$ and degenerate, then $\Fline$ is strong.
\end{lem}
\begin{proof}
	This follows from Proposition~\ref{lem.Hodge.degen}.
\end{proof}

\begin{prop}
	Let $A$ be an $R$-algebra, and let $\tau\colon\ZZ\rightarrow\NN_{0}^{\ZZ}$ be a function with finite support. Further, let $\Xcal^{\tau}_{\infty,R}(A)\subseteq \FZip_{\infty,R}^{\tau}(A)$ denote the full subcategory of those derived $F$-zips that are homotopy finite projective of type $\tau$ and degenerate. Then $A\mapsto \Xcal^{\tau}_{\infty,R}(A)$ defines a hypercomplete sheaf for the fpqc topology.
 
	Moreover, let $\Xcal_{R}^{\tau}$ denote the associated derived stack. Then the inclusion
        $$
        i\colon\Xcal_{R}^{\tau}\xhookrightarrow{\hphantom{aaa}}  t_0\FZip_{R}^{\tau}$$
        is a closed immersion.
\end{prop}
\begin{proof}
	Analogously to the proof of Proposition~\ref{strong fzip infty sheaf}, it suffices to check that the transition maps of the spectral sequences are zero if and only if they are zero fpqc locally, but as these are maps between discrete modules, we see that this is certainly an fpqc local property.

	Let $\Spec(A)\rightarrow t_0\FZip_{R}^{\tau}$ be given by a derived $F$-Zip $\Fline$ that is homotopy finite projective of type $\tau$. This in particular implies that the formation of the homologies of the graded pieces commutes with arbitrary base change, by Lemma~\ref{betti numbers locally closed}\eqref{bnlc-3}. Now a morphism $f\colon \Spec(T)\rightarrow \Spec(A)$ factors through $\Xcal_{R}^{\tau}\times_{t_0\FZip_{R}^{\tau}}\Spec(A)$ if and only if the spectral sequences associated to the filtrations of $f^*\Fline$ degenerate. Again, this is equivalent to the differentials of the spectral sequences being zero. By the commutativity of the homologies with base change and the fact that being zero for a morphism of finite projective modules is a closed property (see \cite[Proposition 8.4(2)]{WED}), we see that $i$ is in fact a closed immersion.
\end{proof}

\begin{rem}
\label{rem classical}
	Let $A$ be an $R$-algebra.
	Let $\tau\colon\ZZ\rightarrow\NN_0^{\ZZ}$ be a function with finite support, and let $n\in \ZZ$. As explained at the beginning of this section, we get a map of derived stacks 
	$$
	\pi_n\colon \Xcal_{R}^\tau\longrightarrow \cl\FZip_{R}^{\tau_{n}}
	$$
	that is induced by $\Fline\mapsto \pi_{n}\Fline$, where $\tau_{n}\colon \ZZ\rightarrow \NN_{0}$ is given by the function $k\mapsto \tau(k)_{n}$. Also, by Lemma~\ref{degen implies strong}, we see that the inclusion $\Xcal_{R}^\tau\hookrightarrow t_{0}\FZip_{R}^{\tau}$ factors through the open derived substack $t_{0}\sFZip_{R}^{\tau}$.
 
	We can also include $\cl\FZip_{R}^{\tau_{n}}$ in $\Xcal_{R}^{\tau}$ by considering the  functor 
	$$
	(-)[n]\colon\Mline\coloneqq (M,C^{\bullet},D_{\bullet},\varphi_{\bullet})\longmapsto \Mline[n]\coloneqq(C_{n}^{\bullet},D^{n}_{\bullet},\id^{n}_{M},\varphi^{n}_{\bullet}),
	$$
	where $C_{n}^{k} \coloneqq C^{k}[n]$, $D^{n}_{k} \coloneqq D_{k}[n]$, $\id_{M}^{n}\coloneqq\id_{M}[n]$ and $\varphi^{n}_{k}\coloneqq \varphi_{k}[n]$ (this is just the $n$-shift of $\Mline[0]$ in $\FZip_{R}(A)$). Thus, $\pi_{n}$ defines a section of $(-)[n]$.
 
	We see that the morphism given by
	$$
	\prod_{n\in\ZZ}\cl\FZip_{R}^{\tau_{n}}\xhookrightarrow{\hphantom{aaa}} \Xcal_{R}^{\tau},\quad (\Mline_{n})_{n\in\ZZ}\longmapsto \bigoplus_{n\in\ZZ}\Mline_{n}[n]
	$$
	is a monomorphism, as it is a product of monomorphisms (note that $\tau$ has finite support, and so the morphism above is induced by the termwise inclusions).
\end{rem}

\begin{lem}
\label{lem fzip classical}
	Let $\tau\colon\ZZ\rightarrow \NN_{0}^{\ZZ}$ be a function with finite support and finitely many values. Then the monomorphism $\prod_{n\in\ZZ}\cl\FZip_{R}^{\tau_{n}}\hookrightarrow \Xcal_{R}^{\tau}$ defined in Remark~\ref{rem classical} is an equivalence of derived stacks.
\end{lem}
\begin{proof}
We have to show that the map $\prod_{n\in\ZZ}\cl\FZip_{R}^{\tau_{n}}\hookrightarrow \Xcal_{R}^{\tau}$ is an effective epimorphism. 
 
	It is enough to show that for an $R$-algebra $A$, every $\Fline\in\Xcal^{\tau}_{R}(A)$ is equivalent to $\bigoplus_{n\in\ZZ}\Mline_{n}[n]$ for some $F$-zips $\Mline_{n}\in\cl\FZip^{\tau_{n}}_{R}(A)$.
 
	Let us set $\Fline\simeq (C^{\bullet},D_{\bullet},\varphi_{\bullet})$ (see Remark~\ref{F-zips.easy}). We can assume that for every $k\in\ZZ$, we have $C^{k}\simeq \bigoplus_{n\in\ZZ} C^{k}_{n}[n]$ and $D_{k}\simeq \bigoplus_{n\in\ZZ}D_{k}^{n}[n]$, by Lemma~\ref{locally acyclic}. Let $\pi_{n}\Fline = (M,\widetilde{C}^{\bullet},\widetilde{D}_{\bullet},\pi_{n}\varphi_{\bullet})$, as in the beginning of this section. By the strongness of $\Fline$ and the constructions of $\widetilde{C}^{\bullet}$ and $\widetilde{D}_{\bullet}$, we see that $\widetilde{C}^{k} = C^{k}_{n}$ and $\widetilde{D}_{k}=D_{k}^{n}$. Thus, we immediately see that $\Fline\simeq \bigoplus_{n\in\ZZ}\pi_{n}\Fline[n]$.
\end{proof}

\begin{lem}
Let us fix some $n\in\ZZ$. Further, let $\sigma\colon\ZZ\rightarrow \NN_{0}$ be a function with finite support and $\tau^{\sigma}_{n}\colon\ZZ\rightarrow \NN_{0}^{\ZZ}$ be given by $k\mapsto \tau^{\sigma}_{n}(k)_{n}=\sigma(k)$ and $k\mapsto\tau^{\sigma}_{n}(k)_{m} =0$ for $m\neq n$. Then the inclusion $\Xcal_{R}^{\sigma}\hookrightarrow t_{0}\FZip_{R}^{\sigma}$ is quasi-compact open.
\end{lem}
\begin{proof}
	By Lemma~\ref{lem fzip classical}, we have an equivalence $\Xcal_{R}^{\sigma}\simeq \cl\FZip_{R}^{\tau_{n}^{\sigma}}$ which is quasi-compact open and closed in $t_{0}\FZip^{\sigma}_{R}$ by Lemma~\ref{F-zip in classical} (note that Lemma~\ref{F-zip in classical} assumes $n=0$, but the proof for arbitrary $n$ works similarly).
\end{proof}

The results in this section show us that for morphisms with degenerating Hodge--de Rham spectral sequence, there is no new information coming from the theory of derived $F$-zips. This, for example, will also show that in the case of abelian schemes, where the $F$-zips associated to its de Rham cohomology are already determined by its $H^{1}_{\dR}$, the derived $F$-zip is also determined by $H^{1}_{\dR}$ (see Section~\ref{abelian schemes}).
 
 In the next section, we want to discuss some classical examples, like curves and K3-surfaces. Analogously to the abelian scheme case, we can use Lemma~\ref{lem fzip classical} to determine the associated derived $F$-zips by their classical counterparts. We will not do this, as this is completely analogous, but will focus on derived $F$-zips with type given by the types associated to proper smooth curves and K3-surfaces. As the type $\underline{Rf_{*}\Omega_{X/S}}$ in these cases will have a certain form, we will see that any strong derived $F$-zip with the same type is equivalent to a derived $F$-zip coming from a classical one.

\subsection{Classical examples}

We want to look at derived $F$-zips associated to abelian schemes, proper smooth curves and K3-surfaces, and explicitly show that we do not get anything new from the theory of derived $F$-zips.

\subsubsection{Abelian schemes}

\label{abelian schemes}

Let $X\rightarrow S$ be an abelian scheme of relative dimension $n$. A classical result is that $$H^i_{\dR}(X/S)=\wedge^iH^1_{\dR}(X/S),$$ $H^1_{\dR}(X/S)$ is locally free of rank $2n$ (and thus  the $H^i_{\dR}(X/S)$ are also finite locally free), the $R^jf_*\Omega_{X/S}^j$ are finite locally free and the Hodge--de Rham spectral sequence degenerates (see \cite[Proposition 2.5.2]{BBM}).

In this way, we can associate to any abelian scheme $X\rightarrow S$ of relative dimension $n$ and any $i\in\NN$ an $F$-zip $\Hline^i_{\dR}(X/S)$.
We can even go further and say that $\Hline^i_{\dR}(X/S)$ is characterized by $H^1_{\dR}(X/S)$, \textit{i.e.}\
$$
\Hline^i_{\dR}(X/S)=\wedge^i\Hline^1_{\dR}(X/S)
$$
(see \cite[Example 9.9]{PWZ}).

Therefore, Lemma~\ref{lem fzip classical} (or rather its proof) implies the following.

\begin{prop}
	Let $f\colon X\rightarrow S$ be an abelian scheme of relative dimension $n$. The derived $F$-zip $\underline{R\Gamma_{\dR}(X/S)}$ is equivalent to the derived $F$-zip $\bigoplus_{k=0}^{2n}\wedge^{k}\Hline^1_{\dR}(X/S)[k]$ $($see Remark~\ref{rem classical} for the notation$)$.
\end{prop}
\begin{proof}
	See the discussion above.
\end{proof}

\subsubsection{Proper smooth curves}
Let $C$ be a proper smooth connected curve of genus $g$ over an algebraically closed field $k$ of characteristic $p>0$. The de Rham complex consists of two terms. By the degeneracy of the spectral sequence, we know that
$$
	H^n_{\dR}(C/k)\cong\begin{cases}
					\Gamma(C,\Ocal_C) = k &\text{ if } n=0,\\
					H^1(C,\Ocal_C)\oplus \Gamma(C,\Omega^1_{C/k})  &\text{ if }n=1,\\
					H^1(C,\Omega^{1}_{C/k}) =k &\text{ if }n=2,\\
					0&\text{ else}.
					\end{cases}
$$
Further, we know that $g=\dim_kH^1(C,\Ocal_C)=\dim\Gamma(C,\Omega_{C/k}^{1})$. Therefore, the de Rham hypercohomology of $C$ is a perfect complex of Tor-amplitude in $[-2,0]$, and the filtrations are of the form 
\begin{align*}
	\HDG^{\bullet}\colon\cdots \longrightarrow 0\longrightarrow \HDG^{1}\longrightarrow R\Gamma_{\dR}(C/k)\longrightarrow R\Gamma_{\dR}(C/k)\longrightarrow \cdots \\
	\conj_{\bullet}\colon \cdots \longrightarrow 0\longrightarrow \conj_0\longrightarrow R\Gamma_{\dR}(C/k)\longrightarrow R\Gamma_{\dR}(C/k)\longrightarrow \cdots. 
\end{align*}
The graded pieces are given by
$$
\gr^0\HDG\simeq R\Gamma(C,\Ocal_C)\quad\textup{and}\quad\gr^{1}\HDG\simeq R\Gamma(C,\Omega^1_{C/k})[-1], 
$$
which are perfect complexes of Tor-amplitude in $[-1,0]$ and $[-2,-1]$ (in homological notation).
 
Let $\Mcal$ denote the moduli stack of smooth proper curves $X\rightarrow S$ (see \cite[0DMJ]{stacks-project}). The map
$$
R\Gamma_{\dR}\colon X/S\longmapsto \underline{R\Gamma_{\dR}(X/S)}
$$
from $\Mcal$ to $t_0\FZip^{[-2,0],\lbrace -1,0\rbrace}$ gives a decomposition $\Mcal=\coprod_g \Mcal_g$ into open and closed substacks classifying smooth proper curves of genus $g$. This follows from the fact that the Euler characteristic of the graded pieces of the de Rham hypercohomology are determined by the genus and from Proposition~\ref{f-zip decomp}.\newline

Consider the function $\sigma\colon \ZZ\rightarrow\NN_0^\ZZ$ defined as $\sigma(0)_0=\sigma(1)_{-2}=1$ and $\sigma(0)_{-1}=\sigma(1)_{-1}=g$ for some $g\in\NN$ and with value zero otherwise. An example of a strong derived $F$-zip homotopy finite projective of type $\sigma$ is $\underline{R\Gamma_{\dR}(C/k)}$ by the above (the strongness follows from Theorem~\ref{de Rham strong degen}). As the de Rham cohomologies are finite projective and the Hodge spectral sequence degenerates, we see that locally  $\underline{R\Gamma_{\dR}(C/k)}$ is determined by the graded pieces of the Hodge and conjugate filtrations (see Lemma~\ref{locally acyclic}). This allows us to construct $\underline{R\Gamma_{\dR}(C/k)}$ from the classical $F$-zip $\Hline^{1}_{\dR}(C/k)$. But not only that, since $\sigma$ is not too complicated (there is only one non-trivial homotopy with non-trivial filtration), it seems reasonable that any derived $F$-zip of type $\sigma$ is equivalent to one that is induced by a classical $F$-zip (see below for more details).
 
First let us show how to extend a classical $F$-zip of type $\tau\colon\ZZ\rightarrow \NN_{0},\ k\mapsto \sigma(k)_{-1}$ to a derived $F$-zip of type $\sigma$.

\begin{construction}
	\label{constr curve}
	Let $A$ be an $\FF_{p}$-algebra. Recall that the natural morphism $A\rightarrow A^{(1)}$ is an isomorphism of rings. Further, let $\tau$ be as above.

Let $\Mline=(M,C^\bullet,D_\bullet,\varphi_\bullet)$ be a classical $F$-zip over $A$ of type $\tau$. We define $M^+\coloneqq A[0]\oplus M[-1]\oplus  A[-2]$, $C_+\coloneqq C^{1}[-1]\oplus A[-2]$ and $D^+\coloneqq D_0[-1]\oplus A[0]$ as complexes in $\Dcal(A)$. This defines a descending filtration $C^{\bullet}_+\colon C_+\rightarrow M^+$, where $M^{+}$ is in degree $0$, and an ascending filtration $D^+_\bullet\colon D^+\rightarrow M^+$, where $M^{+}$ is in degree $1$, of $A$-modules. We also get natural equivalences between the graded pieces of the filtrations up to Frobenius twist induced by $A^{(1)}\xrightarrow{\lowsim} A$ and $\varphi_\bullet$, denoted by $\varphi_\bullet^+$. We define a new derived $F$-zip over $A$ via $$\Mline^+\coloneqq (C_+^\bullet,D^+_\bullet,\varphi_\bullet^+ ).$$
\end{construction}

The idea of the above construction is to take a classical $F$-zip and extend it by a trivial $F$-zip in the homotopical direction. So in the above construction,  $\Mline^{+}$ is a classical $F$-zip shifted to (homological) degree $-1$, and then we add a trivial $F$-zip via the direct sum to the homotopical degree $0$ and $-2$. All the information of $\Mline^{+}$ as a derived $F$-zip lies in homotopical degree $-1$. In particular, we can recover derived $F$-zips with type like $\Mline^{+}$ from classical $F$-zips. 

\begin{prop}
	\label{classical curve}
	Let $\sigma$ and $\tau$ be defined as above. Then for an $\FF_p$-algebra $A$, the map
	\begin{align*}
		\alpha\colon\cl\FZip_{R}^{\tau}(A)&\longrightarrow \sFZip_{R}^{\sigma}(A)\\
		\Mline&\longmapsto\Mline^+
	\end{align*}
	induces an effective epimorphism $\cl\FZip_{R}^{\tau}\rightarrow \sFZip_{R}^{\sigma}$ of derived stacks.
\end{prop}
\begin{proof}
	Let $A$ be an $\FF_p$-algebra. Consider a derived $F$-zip  $\Fline=(C^\bullet_F,D_\bullet^F,\varphi_\bullet^F)$ over $A$ that is homotopy finite projective of type $\sigma$ over $A$. We claim that there is a classical $M$-zip $\Mline=(M,C^\bullet,D_\bullet,\varphi)$ of type $\tau$ that induces an equivalence $\Mline^+\xrightarrow{\lowsim}\Fline$.

	We can apply Lemma~\ref{locally acyclic} to the filtrations and graded pieces of $\Fline$ and may only work with perfect complexes over $A$ that have vanishing differentials, \textit{i.e.}\ direct sums of shifts of finite free modules (note that we use that the filtrations are finite).
 
	Using the explicit type of $\Fline$, we get a long exact homotopy sequence
	$$
		0\longrightarrow (C^{0}_{F})_0\longrightarrow A\xrightarrow{\;\del\;} (C^1_F)_{-1}\longrightarrow (C^{0}_{F})_{-1}\longrightarrow A^g\xrightarrow{\;\del'\;}(C^1_F)_{-2}\longrightarrow (C^{0}_{F})_{-2}\longrightarrow 0,
	$$
	and using the ascending filtration, we get the long exact homotopy sequence
	$$
	0\longrightarrow(D^F_{0})_{0}\longrightarrow (D^{F}_{1})_0\longrightarrow 0\longrightarrow (D^F_{0})_{-1}\longrightarrow (D^{F}_{1})_{-1}\longrightarrow (A^g)^{(1)}\longrightarrow 0\longrightarrow (D^{F}_{1})_{-2}\longrightarrow A^{(1)}\longrightarrow 0.
	$$
	The strongness of our filtrations show that $\del$ and $\del'$ are zero, and we see that we can set
	\begin{align*}
		&M=F_{-1},\\
		&C^\bullet\colon 0=C^2\subseteq (C^1_F)_{-1}\subseteq F_{-1}=C^0,\\
		&D_\bullet\colon 0=D_{-1}\subseteq (D_0^F)_{-1}\subseteq F_{-1}=D_1\textup{ and}\\
		&\varphi_0=\pi_{-1}\varphi^F_{0},\ \varphi_{1}=\pi_{-1}\varphi^F_{1}.
	\end{align*}
	The acyclicity of the complexes involved give us an equivalence $\Mline^+\xrightarrow{\lowsim}\Fline$.
\end{proof}

\subsubsection{K3-surfaces}
Let $X$ be a K3-surface over a field $k$. It is well known that the Hodge--de Rham spectral sequence of $X/k$ degenerates and that the Hodge numbers are given by $h^{0,0}=h^{0,2}=h^{2,0}=h^{2,2}=1$, $h^{1,1}=20$ and are otherwise zero. This in particular gives us the type of the derived $F$-zip associated to a K3-surface over an arbitrary scheme in positive characteristic (which is written out in the following remark).

Using this, we will show as in the case of proper smooth curves that every derived $F$-zip of the same type as in the K3-surface case is equivalent to one which comes from the classical $F$-zip of type associated to $H^2_{\dR}$ of a K3-surface.

\begin{rem}
	\label{strong K3}
	Let $\sigma\colon\ZZ\rightarrow\NN_0^\ZZ$ be the function given by $\sigma(0)_0=\sigma(0)_{-2}=\sigma(2)_{-2}=\sigma(2)_{-4}=1$, $\sigma(1)_{-2}=20$ and otherwise zero.
 
	Let $\Fline\coloneqq(C^{\bullet},D_{\bullet},\varphi_{\bullet})\in t_0\FZip^{\sigma}(A)$. We may assume by Lemma~\ref{locally acyclic} that every perfect complex associated to $\Fline$ has vanishing differentials; \textit{i.e.}\ the filtrations and graded pieces are perfect complexes of $A$-modules with vanishing differentials. So, we can write
        $$
        C^{\bullet}\simeq \bigoplus_{n\in\ZZ}(C^{\bullet})_{n}[n]\quad\textup{and}\quad D_{\bullet}\simeq \bigoplus_{n\in\ZZ}(D_{\bullet})_{n}[n].
        $$
        By this $C^2\simeq \pi_{-2}\gr_C^2[-2]\oplus\pi_{-4}\gr_C^2[-4]$, and we have long exact homotopy sequences
	\begin{align*}
		0\longrightarrow (C^2)_{-2}\longrightarrow &(C^1)_{-2}\longrightarrow \pi_{-2}\gr^1_C\longrightarrow 0\longrightarrow (C^2)_{-3}\\ &\longrightarrow 0\longrightarrow 0\longrightarrow (C^2)_{-4}\longrightarrow (C^1)_{-4}\longrightarrow 0,
	\end{align*}
	\begin{align*}
		0\longrightarrow (C^{0})_0\longrightarrow &\pi_{0}\gr^0_C\longrightarrow 0\longrightarrow (C^{0})_{-1}\longrightarrow 0\longrightarrow (C^1)_{-2}\longrightarrow (C^{0})_{-2}\longrightarrow  \pi_{-2}\gr^0_C\\ &\longrightarrow 0\longrightarrow  (C^{0})_{-3}\longrightarrow 0\longrightarrow (C^1)_{-4}\longrightarrow (C^{0})_{-4}\longrightarrow 0. 
	\end{align*}
	Further, we have $D_0\simeq  \pi_0\gr^0_D[0]\oplus \pi_{-2}\gr^0_D[-2]$ and long exact homotopy sequences
	\begin{align*}
		0\longrightarrow (D_0)_0\longrightarrow (D_1)_0\longrightarrow 0 \longrightarrow 0\longrightarrow (D_1)_{-1}\longrightarrow 0 \longrightarrow (D_0)_{-2}\longrightarrow (D_1)_{-2}\longrightarrow \pi_{-2}\gr^1_D\longrightarrow 0,
	\end{align*}
	\begin{align*}
	  0\longrightarrow (D_1)_0\longrightarrow &(D_{2})_0\longrightarrow 0 \longrightarrow 0\longrightarrow (D_{2})_{-1}\longrightarrow 0 \longrightarrow (D_1)_{-2}\longrightarrow (D_{2})_{-2}\\ &\longrightarrow \pi_{-2}\gr^2_D\longrightarrow 0\longrightarrow (D_{2})_{-3}\longrightarrow 0 \longrightarrow 0\longrightarrow (D_{2})_{-4}\longrightarrow \pi_{-4}\gr^2_D\longrightarrow 0.
	\end{align*}
	In particular, we have that $\Fline$ is a strong derived $F$-zip as the homotopies are finite locally free, which allows us to construct sections.
\end{rem}

\begin{cor}
	Let $X/S$ be a K3-surface. Then the Hodge--de Rham spectral sequence associated to $X/S$ degenerates.
\end{cor}
\begin{proof}
	Combine Remark~\ref{strong K3} and Theorem~\ref{de Rham strong degen}.
\end{proof}

\begin{lem}
	\label{lem strong K3}
	Let $\sigma\colon\ZZ\rightarrow\NN_0^\ZZ$ be the function given by some $$
        \sigma(0)_0,\sigma(0)_{-2},\sigma(1)_{-2},\sigma(2)_{-2},\sigma(2)_{-4}\in\NN_0
        $$
        and otherwise zero. Then the inclusion
	$$
	\sFZip^{\sigma}\xhookrightarrow{\hphantom{aaa}}\FZip^{\sigma}
	$$
	is an equivalence. 
\end{lem}
\begin{proof}
	We only have to check that it is an effective epimorphism, which can be checked locally. Now the argumentation as in Remark~\ref{strong K3} concludes the proof.
\end{proof}

Again, as in the proper smooth curve case, we want to construct a derived $F$-zip of type $\sigma$ out of a classical $F$-zip and show that all derived $F$-zips of type $\sigma$ are given by those.

In particular, the derived $F$-zip associated to a K3-surface will carry no additional information besides the classical $F$-zip attached to its $H^{2}_{\dR}$.

\begin{construction}
	\label{constr K3}
	Let $\Mline=(M,C^\bullet,D_\bullet,\varphi_\bullet)$ be a classical $F$-zip over $A$ of type $\tau$, where $\tau(2)=\tau(0)=1$ and $\tau(1)=20$ and $\tau$ is otherwise zero. We set $M^+\coloneqq A[0]\oplus M[-2]\oplus  A[-4]$, $C^2_+\coloneqq C^{2}[-2]\oplus A[-4]$, $C_+^1=C^1_+[-2]\oplus A[-4]$, $D^+_0\coloneqq  D_0[-2]\oplus A[-4]$ and $D^+_1=A[0]\oplus D_1[-2]$. This defines a descending filtration $C^{\bullet}_+\colon C_+^2\rightarrow C_+^1\rightarrow M^+$ and an ascending filtration $D^+_\bullet\colon D^+_0\rightarrow D^+_1\rightarrow M^+$. We also get natural equivalences between the graded pieces of the filtrations up to Frobenius twist induced by $A^{(1)}\xrightarrow{\lowsim} A$ and $\varphi_\bullet$, denoted by $\varphi_\bullet^+$. We define a new derived $F$-zip over $A$ via $$\Mline^+\coloneqq (C_+^\bullet,D^+_\bullet,\varphi_\bullet^+ ).$$
\end{construction}

\begin{prop}
	let $\tau$ be as in Construction~\ref{constr K3}, and let $\sigma$ be as in Remark~\ref{strong K3}. Then for an $\FF_p$-algebra $A$, the map
	\begin{align*}
		\alpha\colon\cl\FZip_{R}^{\tau}(A)&\longrightarrow \FZip_{R}^{\sigma}(A)\\
		\Mline&\longmapsto\Mline^+
	\end{align*}
	induces an effective epimorphism $\cl\FZip^{\tau}_{R}\rightarrow t_{0}\FZip_{R}^{\sigma}$ of derived stacks. 
\end{prop}
\begin{proof}
	Using that a derived $F$-zip of type $\sigma$ is automatically strong (see Lemma~\ref{lem strong K3}), we see that the proof is analogous to the proof of Proposition~\ref{classical curve} with Construction~\ref{constr K3}.
\end{proof}

\section{Application to Enriques surfaces}
\label{sec:Enriques}

One of the main reasons behind the theory of derived $F$-zips is to extend the theory of $F$-zips so that we can use it on geometric objects that have non-degenerate Hodge--de Rham spectral sequence. One example of such geometric objects consists of  Enriques surfaces in characteristic $2$.  Here, we have three types of Enriques surfaces: $\ZZ/2\ZZ$, $\mu_{2}$ and $\alpha_{2}$. The Enriques surfaces of type $\alpha_{2}$ are of particular interest for us since they have non-degenerate Hodge--de Rham spectral sequences (for the other types, the spectral sequences degenerate).  One can show (see \cite{Lied}) that the moduli stack $\Mcal$ of Enriques surfaces has three substacks $\Mcal_{\alpha_{2}}$, $\Mcal_{\ZZ/2\ZZ}$ and $\Mcal_{\mu_{2}}$ that classify precisely these three types. Further, $\Mcal_{\alpha_{2}}$ and $\Mcal_{\ZZ/2\ZZ}$ are open and $\Mcal_{\mu_{2}}$ is closed in $\Mcal$. We will come to the same result using the theory of derived $F$-zips (see Proposition~\ref{prop Enriques}) since the substacks corresponding to the types of Enriques surfaces can be classified by the corresponding type of the derived $F$-zip associated to the de Rham hypercohomology.

\subsection{Overview}

We will briefly recall the definition of Enriques surfaces and some properties. We use the upcoming book of Cossec, Dolgachev and Liedtke as a reference (see \cite{Enriques}). For this, fix an algebraically closed field $k$ of characteristic $p>0$.
 
\begin{defi}
	An Enriques surface is a proper smooth surface over $k$ with Kodaira dimension $0$ and $b_2(X)\coloneqq \dim_{\QQ_\ell}H^2_{\et}(X,\QQ_\ell) = 10$, where $\ell\not = p$ is a prime.
\end{defi}

\begin{prop}
	Let $S$ be an Enriques surface over $k$. If the characteristic is $p=2$, then the group scheme of divisor classes which are numerically equivalent to $0$, denoted by $\Pic^{\tau}_{S/k}$, is either $\ZZ/2\ZZ,\mu_2$ or $\alpha_2$. In characteristic greater than $2$, we have $\Pic^{\tau}_{S/k}\cong \ZZ/2\ZZ$.
\end{prop}

\begin{defi}
	Let $S$ be an Enriques surface over $k$, and assume $p=2$. Then we call $S$ \textit{classical} (resp.\ \textit{singular} or \textit{supersingular}) or of \textit{type $\ZZ/2\ZZ$ $($resp.\ $\mu_{2}$ or $\alpha_{2})$} if $\Pic^{\tau}_{S/k}$ is isomorphic to $\ZZ/2\ZZ$ (resp.\ $\mu_2$ or $\alpha_2$).
\end{defi}

\begin{prop}
	Let $S$ be an Enriques surface over $k$. The associated Hodge--de Rham spectral sequence degenerates if and only if $S$ is not supersingular.
\end{prop}
\begin{proof}
	This is \cite[Corollary 1.4.15]{Enriques}, but let us recall the arguments (note that there is a typo in the reference, as the authors compute the crystalline cohomology and conclude the de Rham cohomology by the universal coefficient formula, which implies the numbers in Table (\ref{table de rham})).

	In \cite[Section 1.4, Tables 1.2 and~1.3]{Enriques}, the authors give the exact Hodge numbers and dimensions of the de Rham cohomology, which in particular implies the result about degeneracy. 
 
	Let us be a bit more precise and recall the important numbers. Let $h^{i,j}$ denote the $k$-dimension of $H^{j}(S,\Omega^{i}_{S/k})$ and $h^{i}_{\dR}$ the dimension of $H^{i}(S,\Omega_{S/k}^{\bullet})$. Then we have the following table linking the type of $S$ with the Hodge numbers: 	
	\begin{equation}
	\label{table hodge}
		\begin{tabular}{|| c || c || c | c || c | c | c ||} 
 			\hline
 		\rule{0pt}{3ex}	$\Pic^{\tau}_{S/k}$ & $h^{0,0}$  & $h^{1,0}$ & $h^{0,1}$ & $h^{0,2}$ & $h^{1,1}$ & $h^{2,0}$ \\ [0.5ex] 
 			\hline\hline
 			$\mu_{2}$ & 1 & 0 & 1 & 1 & 10 & 1  \\ 
 			\hline
			 $\ZZ/2\ZZ$ & 1 & 1 & 0 & 0 & 12 & 0  \\
			 \hline
			 $\alpha_{2}$ & 1 & 1 & 1 & 1 & 12 & 1  \\
			 \hline
		\end{tabular}
	\end{equation}
	The dimension of the de Rham cohomology is given as follows (this does not depend on the type of the Enriques surfaces): 
	\begin{equation}
	\label{table de rham}
		\begin{tabular}{|| c | c | c |  c | c ||} 
 			\hline
 			\rule{0pt}{3ex}	 $h^{0}_{\dR}$ &$h^{1}_{\dR}$ & $h^{2}_{\dR}$ & $h^{3}_{\dR}$ & $h^{4}_{\dR}$\\ [0.5ex] 
 			\hline\hline
 			 1 & 1 & 12 & 0 & 1 \\ 
 			\hline
		\end{tabular}
	\end{equation}
	By Serre duality, this table is enough to conclude (non-)degeneracy.
\end{proof}

We denote by $(\Pic^{\tau}_{S/k})^{D}$ the Cartier dual of $\Pic^{\tau}_{S/k}$. Note that $\alpha_2^D=\alpha_2$, $\ZZ/2\ZZ^D=\mu_2$ and $\mu_2^D=\ZZ/2\ZZ$.

\begin{prop}
	Let $S$ be an Enriques surface over $k$. There exists a	non-trivial $(\Pic^{\tau}_{S/k})^{D}$-torsor
	$$
	\pi\colon X\longrightarrow S. 
	$$
	In particular, $\pi$ is finite flat of degree $2$. Note that if $p\not = 2$ or $S$ is of type $\mu_2$, then $\pi$ is \'etale.
\end{prop}
\begin{proof}
	See \cite[Theorem 1.3.1]{Enriques}.
\end{proof}

\begin{defi}
	A finite flat map $X\rightarrow S$ of degree $2$ is called a \textup{K3}\textit{-cover}.
\end{defi}

\begin{prop}
	Let $S$ be an Enriques surface over $k$. Let $\pi\colon X\rightarrow S$ be a \textup{K3}-cover. Then $X$ is integral Gorenstein, satisfying
	$$
	H^1(X,\Ocal_X)=0,\quad \omega_X\cong \Ocal_X.
	$$
	Further, we have the following: 
	\begin{enumerate}
		\item	If $p\not=2$ or $S$ is of type $\mu_2$, then $X$ is a smooth \textup{K3}-surface. 
		\item	If $p=2$ and $S$ is of type $\ZZ/2\ZZ$ or $\alpha_2$, then $X$ is not a smooth surface.
	\end{enumerate}
\end{prop}
\begin{proof}
	See \cite[Proposition 1.3.3]{Enriques}.
\end{proof}

\begin{defi}
	Let $S$ be a  $\FF_{p}$-scheme. An \textit{Enriques surface $X$ over $S$} is a proper smooth morphism of algebraic spaces $f\colon X\rightarrow S$ such that the geometric fibers of $f$ are Enriques surfaces.
\end{defi}

\subsection{Derived $\boldsymbol{F}$-zips associated to Enriques surfaces}
In the following, every scheme will be in characteristic $2$.

We let $\Mcal$ denote the stack classifying Enriques surfaces with ``nice'' polarization, \textit{i.e.}\ the functor that sends an $\FF_{2}$-scheme $S$ to the groupoid of pairs $(X/S,\Lcal)$ consisting of Enriques surfaces $X\rightarrow S$ with ``nice'' line bundle $\Lcal$ on $X$~-- the term ``nice'' means a polarization such that $\Mcal$ defines an Artin stack. Examples of such classifying stacks are given in \cite[Theorem 5.11.6]{Enriques} and \cite[Section 5]{Lied} (we only need that $\Mcal$ is an Artin stack and are not interested in the polarization itself and as there are many different such polarizations such that $\Mcal$ is an Artin stack, we omit the explicit description). By our previous constructions, we get a morphism 
$$
p\colon \Mcal\longrightarrow t_0\FZip_{S},\quad (X/S,\Lcal)\longmapsto \underline{Rf_{\ast}\Omega_{X/S}^{\bullet}}.
$$
The Hodge numbers and dimension of the de Rham cohomology for Enriques surfaces over algebraically closed fields define types for the underlying $F$-zip (see Tables~\eqref{table hodge} and~\eqref{table de rham}).
We denote those by $\tau_{\ZZ/2\ZZ},\tau_{\mu_2},\tau_{\alpha_2}$\footnote{Recall that the types are given by $\tau_{\ast}(i)_{j}\coloneqq h^{i,-j-i}$ (the Hodge numbers of the corresponding types). Table (\ref{table hodge}) shows that $\tau_{\mu_{2}}\leq \tau_{\alpha_{2}}$ and $\tau_{\ZZ/2\ZZ}\leq \tau_{\alpha_{2}}$ and that there is  no relation between $\tau_{\mu_{2}}$ and $\tau_{\ZZ/2\ZZ}$.} for the types defined by the Hodge numbers of $\ZZ/2\ZZ$, $\mu_2$, $\alpha_2$ Enriques surfaces, respectively.
 
We denote the corresponding loci by $\Mcal_{\ZZ/2\ZZ}$ and $\Mcal_{\mu_2}$; \textit{i.e.}\ these denote the substacks classifying Enriques surfaces of types $\ZZ/2\ZZ$ and $\mu_{2}$, respectively. We denote the substack of $\alpha_{2}$ Enrique surfaces $f\colon X\rightarrow S$ such that $R^{i}f_{*}\Omega_{X/S}^{j}$ is finite locally free for all $i,j\in\ZZ$ by $\Mcal_{\alpha_{2}}$.\footnote{Note that if $f\colon X\rightarrow S$ is a locally Noetherian reduced Enriques surface, we can use \cite[Proposition (7.8.4)]{EGA3.2}
to deduce that $R^{i}f_{*}\Omega_{X/S}^{j}$ is finite locally free for all $i,j\in\ZZ$.} With these definitions, we see that
$$
p^{-1}(t_0\FZip^{\leq \tau_{\alpha_2}})\coloneqq \Mcal\times_{t_{0}\FZip_{S}}t_0\FZip^{\leq \tau_{\alpha_2}}\simeq \Mcal,
$$
and we will see in the following that the substacks $\Mcal_{\ZZ/2\ZZ}$ and $\Mcal_{\mu_2}$ are open in $\Mcal$ and $\Mcal_{\alpha_{2}}$ is closed in~$\Mcal$.

\begin{prop}
\label{prop Enriques}
	The substacks $\Mcal_{\ZZ/2\ZZ}$ and $\Mcal_{\mu_2}$ are open algebraic substacks, and $\Mcal_{\alpha_2}$ is a closed algebraic substack of $\Mcal$ locally of finite presentation.
\end{prop}
\begin{proof}
	Let us look at $\Mcal_{\ZZ/2\ZZ}\simeq p^{-1}(t_0\FZip^{\leq \tau_{\ZZ/2\ZZ}})$. We claim that this is an open substack of $\Mcal$. Since $\Mcal$ is an Artin stack, we know that it is a $1$-geometric $1$-truncated derived stack in our sense. In particular, since the base changes of open immersions are open immersions (\textit{i.e.}\ flat, locally finitely presented monomorphisms), we know by Remark~\ref{truncation} and Proposition~\ref{prop open and closed of F-zip} that $\Mcal_{\ZZ/2\ZZ}\hookrightarrow\Mcal$ is a flat, locally finitely presented monomorphism and in particular $\Mcal_{\ZZ/2\ZZ}$ is $1$-geometric. Since $\Mcal_{\ZZ/2\ZZ}\hookrightarrow \Mcal$ is a monomorphism (\textit{i.e.}\ $(-1)$-truncated), we see that $\Mcal_{\ZZ/2\ZZ}$ is $1$-truncated (see \cite[Lemma 5.5.6.14]{HTT}). In fact, we claim that this shows that $\Mcal_{\ZZ/2\ZZ}$ is an algebraic stack.
 
	To see this, note that since $\Mcal_{\ZZ/2\ZZ}\hookrightarrow \Mcal$ is a monomorphism, the diagonal of $\Mcal_{\ZZ/2\ZZ}$ is representable by an algebraic space. Further, we claim that the $1$-geometricity of $\Mcal_{\ZZ/2\ZZ}$ implies that we have a smooth atlas by a coproduct of affine schemes, so a scheme.
 
	To see the last part, let us look at a smooth $0$-atlas $q\colon\coprod\Spec(A_i)\twoheadrightarrow \Mcal_{\ZZ/2\ZZ}$. We have to check that this is smooth in the classical sense. For that consider the base change with an affine scheme $\Spec(B)\rightarrow \Mcal_{\ZZ/2\ZZ}$, denoted by $X$. This is an algebraic space and by geometricity has a smooth cover $\coprod \Spec(B_i)$ by some smooth $B$-algebras $B_i$ such that each $g\colon\Spec(B_{i})\rightarrow X$ is affine. So, we have a diagram of the following form with Cartesian square: 
	$$
	\begin{tikzcd}
		\coprod \Spec(B_{i})\arrow[r,"g"]&X\arrow[r,"f"]\arrow[d,""]& \Spec(B)\arrow[d,""]\\
		&\coprod\Spec(A_i)\arrow[r,"q"]&\Mcal_{\ZZ/2\ZZ}\rlap{.}
	\end{tikzcd}
	$$
As $g$ is smooth and surjective and $f\circ g$ is smooth, we know by descent that $f$ is smooth (as the property ``smooth'' is smooth local on the source for algebraic spaces, see \cite[06F2]{stacks-project}). Certainly $f$ is also surjective, as it is the base change of an effective epimorphism. Therefore, by definition we see that $p$ is smooth and surjective.

	The same argumentation works if we replace $\Mcal_{\ZZ/2\ZZ}$ with $\Mcal_{\mu_2}$.
 
	For the supersingular locus, we note that the inclusion
        $$
        t_0\FZip^{\tau_{\alpha_2}}\xhookrightarrow{\hphantom{aaa}} t_0\FZip^{\tau_{\leq\alpha_2}}
        $$
        is a closed immersion locally of finite presentation (again by Proposition~\ref{prop open and closed of F-zip}). So, analogously to the above, we see that $\Mcal_{\alpha_2}$ is an algebraic substack $\Mcal$ such that $\Mcal_{\alpha_{2}}\hookrightarrow \Mcal$ is a closed immersion of algebraic stacks locally of finite presentation. 
\end{proof}

\section{Derived \texorpdfstring{$\boldsymbol{F}$}{F}-zips with cup product}
\label{sec:general}

Here we discuss two possible generalizations of derived $F$-zips. Firstly, we could try to extend the theory of derived $F$-zips in such a way  that we can attach a derived $F$-zip to an lci morphism. Secondly, we could extend the theory of derived $F$-zips to the theory of derived $G$-zips as in \cite{PWZ}, for a reductive group $G$, and hope that the extra structure on the de Rham hypercohomology given by the cup product (see Section~\ref{sec extra structure}) endows it with a $G$-zip structure.
 
We will discuss both cases and show that the naive way of extending a derived $F$-zip does not work in either case. But for completeness, we will look at derived $F$-zips with some extra structure that is given by a perfect pairing. Again, we cannot connect this to the theory of $G$-zips, but this is just a very naive approach we want to discuss.

\subsection{Problems}

Now let us discuss the problems that occurred when trying to generalize the theory.

\subsubsection{Derived $\boldsymbol{F}$-zips for lci morphisms}
\label{de rham lci}
We could have defined derived $F$-zips not over animated rings but over usual commutative rings in positive characteristic. One benefit of the animation process is that simplicially every commutative ring can be approximated by smooth rings. One often uses this to generalize theories that work in the smooth case to the non-smooth case. To define derived $F$-zips, we looked at the de Rham hypercohomology of a smooth proper scheme. So to define a theory of derived $F$-zips that works for non-smooth schemes, we would need a non-smooth analogue of the de Rham hypercohomology. The most natural generalization comes from looking at the de Rham complex as a functor from smooth $\FF_p$-algebras and looking at its left Kan extension to animated $\FF_p$-algebras. This is done for example in \cite{Bhatt1,Bhatt2,Ill1} and is called the derived de Rham complex, denoted by $\dR_{X/R}$\footnote{We define $\dR_{-/R}$ as the left Kan extension of the functor $P\mapsto \Omega^{\bullet}_{P/R}$ along the inclusion $\Poly_{R}\hookrightarrow \AniAlg{R}$. Then, we denote by $\RR\dR_{-/R}$ the right Kan extension of $\dR_{-/R}$ along the Yoneda embedding $\AniAlg{R}\hookrightarrow\Pcal(\AniAlg{R}^{\op})^{\op}$ and set $\dR_{X/R}\coloneqq\RR\dR_{-/R}(X)$.} for a scheme $X$ over some ring $R$ of positive characteristic. Let us state some facts about the derived de Rham complex that can also be found in the mentioned articles by Bhatt or in the book by Illusie.
 
In the smooth case, this gives the usual de Rham hypercohomology $R\Gamma_{\dR}(X/R)$. The derived de Rham complex  naturally comes with two filtrations;  one is the conjugate filtration, and one is the Hodge filtration. These come from, respectively, the conjugate and Hodge filtration on the de Rham complex by extending via left Kan extension. The graded pieces are given by
$$
\gr^i_{\textup{conj}}\dR_{X/R}\simeq \wedge^iL_{X^{(1)}/R},\quad \gr^i_{\textup{HDG}}\dR_{X/R}\simeq \wedge^iL_{X/R}; 
$$ 
in particular, they are isomorphic up to Frobenius twist. 
 
Even though it seems natural, it is not clear that the Hodge filtration is complete, \textit{i.e.}\ $\limproj_{i}\textup{HDG}(i)\simeq 0$. Further, the filtrations may not be finite in any way. This holds more or less for any variety with isolated lci singularities. One very generic example is $A\coloneqq k[\varepsilon]/(\varepsilon^{p})$ for some field $k$ of characteristic $p>0$. One can show that $\wedge^nL_{A/k}$ is not quasi-isomorphic to $0$ for any $n\in \NN_0$ (see \cite[Remark~2.2]{Bhatt3}). This obstruction comes from the fact that in the lci case, $L_{A/k}$ is a complex concentrated in two degrees, and after base change to $k$, one can see that it is given by the direct sum of exterior powers and shifts of $k$ (see the proof of \cite[Lemma 2.1]{Bhatt3}). Now the exterior power of the shift of a module can be computed by its free divided power, which will not vanish even for higher powers (see \cite[Proposition 25.2.4.2]{SAG}). To avoid problems, we could define derived $F$-zips using non-complete filtrations, but the problem here is actually the unboundedness of the filtrations. There is no need for an $n$-atlas if we allow infinite filtrations since we would need to cover finitely many data at once, which renders this approach \textit{a priori} useless.

\subsubsection{Derived $\boldsymbol{G}$-zips}
\label{derived G-zip}
The theory of $G$-zips, for a connected reductive group $G$ over a field of characteristic $p>0$, endows the theory of $F$-zips with extra structure related to the group. The motivation behind this is that we have a cup product on the de Rham cohomologies of smooth proper maps, where the relative Hodge--de Rham spectral sequence degenerates. In even degrees, this endows the $F$-zip associated to the de Rham cohomology with a twisted symmetric structure, and in odd degrees with a twisted symplectic structure. All of this can be found in \cite{PWZ}.
 
There are three equivalent approaches to the theory of $G$-zips. The first one is to first identify the stack of $F$-zips over a scheme $S$ with the stack of vector bundles on some quotient stack $\Xfr$, where the quotient stack is defined via the following recipe. We take $\PP^{1}$ and pinch the point at $\infty$ and the point at $0$  together up to Frobenius twist. Now we let $\GG_m$ act on the affine line around $0$ in degree $1$ and on the affine line around $\infty$ in degree $-1$. Let us make precise what happens here. Vector bundles on $[\AA^1/\GG_m]$ are finitely filtered vector bundles, where depending on the action of $\GG_m$, in our case multiplication with an element in $\GG_m$ resp.\ with the inverse, we get an increasing, resp.\ decreasing, filtration (see the appendix for further details). Thus, a vector bundle on $[\PP^{1}/\Gm_{}]$ gives a vector bundle with an ascending and a descending filtration. The pullback to  $0$, resp.\ $\infty$, gives us the graded pieces. Gluing $0$ and $\infty$ together along the Frobenius, we see that a vector bundle on $\Xfr$ gives us an $F$-zip. Now a $G$-zip is just a $G$-torsor over $\Xfr$ (see Theorem~\ref{comp G-zip}).
 
Secondly, one can realize $G$-zips as exact fiber functors from finite $G$-representations to $F$-zips. Using that $F$-zips are the same as finite-dimensional vector bundles over the above quotient stack $\Xfr$ and using Tannaka duality, one sees that this description and the first one agree.
 
Lastly, there is a description of $G$-zips as a quotient stack $[G/E]$. We spare the details for the reader and refer to \cite{PWZ}, where the equivalence with the second description can also be found.

\medskip
In our context, the first and second approaches seem to be the natural ones. The first approach seems to be a bit tricky since we would need to show that there is a quotient stack such that  perfect complexes over this stack give us the derived $F$-zips. Naturally, one could take $\Xfr$ as the desired stack, and as explained in Section~\ref{app perf on pinched}, the perfect complexes on $\Xfr$ recover derived $F$-zips. But we still lack a good notion of derived groups and torsors attaching extra structure to perfect complexes.
 
For the second approach, we would need a replacement for finite $G$-representations. Looking at the works of Iwanari and Bhatt on derived Tannaka duality, it seems natural to replace $\Rep(G_{\FF_p})$ with $\PPerf(\textup{B}G_{\FF_p})$. But even though natural, it will turn out not to be  the right approach. The problem here is $\textup{B}G$. It is the classifying stack associated to a classical group scheme. Thus looking at exact fiber functors and Tannaka duality (which we do not have), one could argue that it should give us the classical theory of $G$-zips embedded into the derived setting. This is not the same as a derived analogue. For example, $\GG_m$-zips are the same as $F$-zips of rank $1$. We would expect derived $\GG_m$-zips to be derived $F$-zips of Euler characteristic $\pm 1$. But we will see that this is not completely true for exact fiber functors from $\PPerf(\BGm_{,\FF_p})$ to $\FZip(A)$ for some $\FF_{p}$-algebra $A$. Instead of derived $F$-zips of Euler characteristic $\pm 1$, they give us derived $F$-zips where the cohomologies are finite locally free, meaning that they give us (more or less) the classical theory.
 
To make everything we wrote precise, let $F$ be an exact monoidal functor from $\PPerf(\BGm_{,\FF_p})$ to $\FZip(A)$ for an $\FF_{p}$-algebra $A$. Since over a field, any complex is quasi-isomorphic to a complex with zero differentials, we see that the descent condition (induced by the Barr resolution of $\BGm_{,\FF_p}$) for a complex in $\PPerf(\BGm_{,\FF_p})$ is a condition on the cohomologies of the complex. With this, we see that $E\in \PPerf(\BGm_{,\FF_p})$ is equivalent to a finite direct sum of $E_i[i]$, where $i\in\ZZ$ and the $E_i$ are finite projective graded modules. Since $F$ is exact, we see that $F$ is already determined, up to equivalence, by its image on vector bundles on $\BGm$, \textit{i.e.}\ finite $\Gm$-representations, seen as complexes concentrated in degree 0. But finite $\Gm$-representations are generated under the tensor product by the standard representation, hence $F$ is already determined, up to equivalence, by the image of $\FF_p$, seen as a graded vector space in degree 1. This is certainly an invertible element in $\PPerf(\BGm_{,\FF_p})$, and thus $F(\FF_p)$ must also be invertible. Since the monoidal structure on $\FZip(A)$ is given componentwise, we see that the underlying module of the $F$-zip has to be invertible. But invertible perfect complexes over $A$ are locally shifts of line bundles and are globally given by the direct sum of shifts of finite projective modules (see \cite[0FNT]{stacks-project}). This is too much and would give us the classical theory of $F$-zips after passing to the cohomology.

\subsection{Extra structure coming from geometry}
\label{sec extra structure}

In this section, we naively put extra structure on derived $F$-zips by looking at the extra structure on the de Rham hypercohomology coming from the cup product, namely a perfect pairing on the underlying module of a derived $F$-zip.

In the following, we fix a ring $R$ of characteristic $p>0$.

\begin{defi}
	Let $A$ be an animated ring. Let $M$ and $N$ be perfect $A$-modules. A \textit{perfect pairing of $M$ and $N$} is a morphism $M\otimes N\rightarrow A$ such that the induced morphism $M\rightarrow N^{\vee}$ is an equivalence.
\end{defi}

\begin{rem}
Let $A$ be an animated ring. Note that any equivalence between perfect $A$-modules $M$ and $N$ of the form $M\rightarrow N^{\vee}$ induces a perfect pairing $M\otimes N\rightarrow A$ by adjunction. So giving a perfect pairing $M\otimes N\rightarrow A$ is equivalent to giving an equivalence $M\rightarrow N^{\vee}$.
\end{rem}

\begin{defi}
	Let $A$ be an animated ring. We define the $\infty$-category of perfect pairings $\textup{PP}_{A}$ over $A$ as the full subcategory of $X$, where $X$ is given by the pullback diagram
	$$
	\begin{tikzcd}
		X\arrow[r,""]\arrow[d,""]& \MMod^{\perf}_{A}\times\MMod^{\perf}_{A}\arrow[d,"{(M,N)\mapsto (M,N^{\vee})}"]\\
		\Fun(\Delta^{1},\MMod^{\perf}_{A})\arrow[r,"r"]&\Fun(\del\Delta^{1},\MMod^{\perf}_{A})\rlap{,}
	\end{tikzcd}
	$$
	of those morphisms $M\rightarrow N^{\vee}$ that are equivalences. (Note that $\textup{r}$ is given by restriction and by \cite[01F3]{kerodon} is an isofibration of simplicial sets. Thus $X$ is equivalent in $\ICat$ to the usual pullback of simplicial sets,\footnote{To check that $X$ is equivalent to the usual pullback of simplicial sets, we may use the Yoneda lemma and check that for any $\infty$-category $\Ccal$, the $\infty$-groupoid $\Fun(\Ccal,X)^{\simeq}$ is given by the usual pullback of simplicial sets, but this follows from Remark~\ref{kan fib pull}} so indeed $X$ classifies morphisms of $A$-modules $M\rightarrow N^{\vee}$).
\end{defi}

\begin{defi}
	Let $A$ be an animated $\FF_p$-algebra, $a\leq b\in \ZZ$ and $S\subset \ZZ$ be a finite subset. Let $\dRZip^{[a,b],S}_{\infty}(A)$ denote the $\infty$-category 
	$$
		\FZip^{[a,b],S}_{\infty}(A)\times_{(\colim[b-a],(\colim)^{\vee}),\MMod^{\perf}_{A}\times\MMod^{\perf}_{A}}\textup{PP}_{A},
	$$
	 \textit{i.e.}\ the $\infty$-category consisting of tuples $(\Fline,\psi)$, where $\Fline\coloneqq (C^{\bullet},D_{\bullet},\phi,\varphi_{\bullet})$ is a derived $F$-zip with $M\coloneqq\colim_{\ZZ^{\op}}C$ and  $\psi\colon M\otimes M\rightarrow \onebb[a-b]$ is a perfect pairing.\footnote{Again, as explained in the definition of $\textup{PP}_{A}$, we have that the pullback diagram defining $\dRZip^{[a,b],S}_{\infty}(A)$ is equivalent in $\ICat$ to the ordinary pullback of simplicial sets, as the projection from $\textup{PP}_{A}$ to $\MMod_{A}\times\MMod_{A}$ is an isofibration (since isofibrations are stable under pullbacks of simplicial sets by \cite[01H4]{kerodon}).} We set $$\dRZip_{\infty}(A)\coloneqq\colim_{\substack{a\leq b,\\ S\subseteq \ZZ \textup{ finite}}} \dRZip_{\infty}^{[a,b],S}(A)$$ and call its elements \textit{\textup{dR}-zips over $A$}.
\end{defi}

Next we want to show that for any proper smooth morphism $f\colon X\rightarrow S$ of schemes, we can attach a $\dR$-zip structure to the derived $F$-zip $\underline{Rf_{*}\Omega_{X/S}^{\bullet}}$. This structure comes naturally from the cup product.

\begin{lem}
\label{Tor of de rham}
	Let $A$ be a ring and $X$ be a proper smooth scheme over $A$ of relative dimension $n$. The de Rham hypercohomology $R\Gamma(X,\Omega^{\bullet}_{X/A})$ has Tor-amplitude in $[-2n,0]$.
\end{lem}
\begin{proof}
	Indeed, first of all, we claim that $\dim_{\kappa(a)} \pi_{i}(X_{\kappa(a)},\Omega^{\bullet}_{X_{a}/\kappa(a)})$ is zero for all $a\in\Spec(A)$ if $i\notin [-2n,0]$.
	This follows from Grothendieck vanishing (see \cite[Theorem 2.7]{HS}) in the following way.
	By Grothendieck vanishing, the perfect complex $R\Gamma(X_{\kappa(a)},\Omega^{k}_{X_{\kappa(a)}/\kappa(a)})[-k]$ has  non-zero homotopies in degrees $-n-k,\dots,-k$.
	Now, the distinguished triangle associated to the stupid truncation
	$$
		R\Gamma(X_{\kappa(a)},\sigma_{\geq k+1}\Omega^{\bullet}_{X_{\kappa(a)}/\kappa(a)})\longrightarrow R\Gamma(X_{\kappa(a)},\sigma_{\geq k}\Omega^{\bullet}_{X_{\kappa(a)}/\kappa(a)})\longrightarrow R\Gamma(X_{\kappa(a)},\Omega^{k}_{X_{\kappa(a)}/\kappa(a)})[-k]
	$$
	shows by induction that $\dim_{\kappa(a)}	\pi_{i}(X_{\kappa(a)},\Omega^{\bullet}_{X_{a}/\kappa(a)})$ is non-zero if and only if $i\in[-2n,0]$.

	Further, it suffices to check Zariski locally on $\Spec(A)$ that $R\Gamma(X,\Omega^{\bullet}_{X/A})$ has Tor-amplitude in $[-2n,0]$.  Any point $a\in\Spec(A)$ has an affine open neighbourhood $U=\Spec(A_{f})$, where $f\in A$, such that $R\Gamma(X,\Omega^{\bullet}_{X/A})_{|U}$ is, by \cite[0BCD]{stacks-project}, equivalent to a complex of the form 
	$$
	\dots \longrightarrow0\longrightarrow A_{f}^{d_{0}}\longrightarrow A_{f}^{d_{-1}}\longrightarrow \dots \longrightarrow A_{f}^{d_{-2n}}\longrightarrow 0\longrightarrow \cdots, 
	$$
	where $A_{f}^{d_{i}}$ sits in homological degree $-i$ and
        $$
        d_{i}\coloneqq \dim_{\kappa(a)}	\pi_{i}(R\Gamma(X,\Omega^{\bullet}_{X/A})\otimes_{A}^{L}\kappa(a))=\dim_{\kappa(a)}	\pi_{i}(X_{\kappa(a)},\Omega^{\bullet}_{X_{a}/\kappa(a)})
        $$
        (the last equality follows as the formation of the de Rham hypercohomology commutes with arbitrary base change; see \cite[0FM0]{stacks-project}).
\end{proof}

\begin{example}
  Let $A$ be a ring and $X$ be a proper smooth scheme over $A$ with non-empty fibers of equidimension~$n$. By Lemma~\ref{Tor of de rham}, $R\Gamma(X,\Omega_{X/A}^\bullet)$ has Tor-amplitude in $[-2n,0]$. Further, the de Rham hypercohomology admits a perfect pairing
  $$
  R\Gamma(X,\Omega_{X/A}^\bullet)\otimes_A^L R\Gamma(X,\Omega_{X/A}^\bullet)[2n]\longrightarrow A
  $$
  (see \cite[0G8K]{stacks-project}). Hence, this induces a dR-zip structure on $\underline{R\Gamma_{\dR}(X/A)}$.
\end{example}

\begin{prop}
	\label{dR-zips}
	The functor
	\begin{align*}
		\dRZip_{\infty,R}\colon\AniAlg{R}&\longrightarrow \ICat\\
		A&\longmapsto \dRZip_{\infty}(A)
	\end{align*}
	defines a hypercomplete sheaf for the fpqc topology. We denote the associated derived stack by $\dRZip_{R}$.
 Let $S\subseteq\ZZ$ be a finite subset, $a\leq b\in \ZZ$, and set $n\coloneqq b-a$. Then the induced morphism 
	$$
		p^{[a,b],S}\colon\dRZip^{[a,b],S}_{R}\longrightarrow\FZip^{[a,b],S}_{R}
	$$
	is $2n$-geometric and smooth. Further, $\dRZip^{[a,b],S}_{R}$ is $2n$-geometric if $n\geq 1$ and $1$-geometric if $n=0$ and is locally of finite presentation.
\end{prop} 
\begin{proof}
	Fix a finite subset $S\subseteq\ZZ$, $a\leq b\in \ZZ$, and set $n\coloneqq b-a$. For a derived $F$-zip $\Fline=(C^{\bullet},D_{\bullet},\phi,\varphi_{\bullet})$, let us set $M_{\Fline}\coloneqq\colim_{\ZZ^{\op}}C^{\bullet}$. Let us look at the following pullback square: 
	$$
	\begin{tikzcd}
		X\arrow[rrrr,""]\arrow[d,""]&&&&\Fun(\Delta^1,\PPerf)\arrow[d,""]\\
		\FZip^{[a,b],S}_{\infty,R}\arrow[rrrr,"{\Fline\mapsto (M_{\Fline},M_{\Fline}^{\vee}[-n])}"]&&&&\Fun(\del\Delta^1,\PPerf)\rlap{.}
	\end{tikzcd}
	$$
	Noting that a perfect pairing of $M_{\Fline}$ and $M_{\Fline}^{\vee}[-n]$ is the same as an equivalence $M_{\Fline}\xrightarrow{\lowsim} M_{\Fline}^\vee[-n]$, we see that $\dRZip^{[a,b],S}_\infty \simeq X$. In particular, $\dRZip^{[a,b],S}_{\infty,R}$ satisfies fpqc hyperdescent.

	Finally, let $p^{[a,b],S}\colon\dRZip^{[a,b],S}_{R}\rightarrow\FZip^{[a,b],S}_{R}$ denote the induced morphism of derived stacks. For a derived $F$-zip $\Fline$ over some animated $\FF_p$-algebra $A$, we have that $(p^{[a,b],S})^{-1}(\Fline)\simeq \Equiv(M_{\Fline},M_{\Fline}^{\vee}[-n])$, which is $2n$-geometric and smooth as $M_{\Fline}\otimes_{A} M_{\Fline}^{\vee}[-n]$ has Tor-amplitude in $[-2n,0]$ (see Lemmas~\ref{equiv zariski open} and~\ref{spec sym geometric}). Now the assertions on $p$ follow immediately by definition, and by Theorem~\ref{F-Zips locally geometric}, we get the results on $\dRZip^{[a,b],S}_{R}$.
\end{proof}

\begin{cor}
\label{last cor}
	The derived stack $\dRZip_{R}$ is locally geometric and locally of finite presentation.
\end{cor}
\begin{proof}
	Let $\tau\colon\ZZ\rightarrow \NN_{0}^{\ZZ}$ be a function with finite support. We know that 
	the inclusion $\FZip^{\leq \tau}_{R}\hookrightarrow \FZip_{R}$ is a quasi-compact open immersion and factors as a geometric morphism through $\FZip^{[a,b],S}_{R}$ for some finite subset $S\subseteq \ZZ$ and $a\leq b\in\ZZ$ (see Remark~\ref{F-zip as colim of open im}). In particular, we see that the pullback of $\FZip^{\leq \tau}_{R}$ along $p^{[a,b],S}$, denoted by $\dRZip^{\leq \tau}$, is again geometric by Proposition~\ref{dR-zips} and Theorem~\ref{F-Zips locally geometric} and  is quasi-compact open in $\dRZip_{R}$. Since $\FZip_{R}\simeq\colim_{\tau}\FZip_{R}^{\leq}$, we see that $\dRZip_{R}\simeq\colim_{\tau}\dRZip_{R}^{\leq \tau}$, and so  $\dRZip_{R}$ is locally geometric.
 That $\dRZip_{R}$ is locally of finite presentation follows analogously from the fact that $\FZip^{\leq\tau}_{R}$ is locally of finite presentation.
\end{proof}

\renewcommand\thesection{\Alph{section}}
\setcounter{section}{0}

\section*{Appendix. Comparison of $\boldsymbol{G}$-zips and $\boldsymbol{G}$-torsors}\label{app comp G-zip}

\addcontentsline{toc}{section}{Appendix. Comparison of \texorpdfstring{$\boldsymbol{G}$}{$G$}-zips and \texorpdfstring{$\boldsymbol{G}$}{$G$}-torsors}
\refstepcounter{section}

\setcounter{subsection}{1}
Throughout, we let $k$ be a field of characteristic $p>0$.

Let us first construct the quotient stack that will classify $F$-zips. Let $S$ be a $k$-scheme. 
Consider the two closed subschemes $\lbrace0\rbrace $ and $\lbrace\infty\rbrace$ inside $\PP^{1}$ that are naturally isomorphic. Let $\varphi$ denote the isomorphism of $\lbrace \infty\rbrace$ and $\lbrace 0\rbrace$ composed with the Frobenius. We know by \cite[Theorem 7.1]{Ferrand} that the following pushout exists in the category of $\FF_{p}$-schemes: 
$$
\begin{tikzcd}
\lbrace \infty\rbrace \amalg \lbrace0\rbrace \arrow[r,"{(\id,\varphi)}"]\arrow[d,"\iota",hook] & \lbrace \infty\rbrace\arrow[d]\\
\PP^{1}_{S}\arrow[r]&X_{S}
\end{tikzcd}
$$
(note that the Frobenius is integral, and thus the morphism from the coproduct is integral, and obviously the inclusion of the two points is a closed immersion). The new space $X_{S}$ is the $\PP^{1}_{S}$ where we pinch the points $0$ and $\infty$ together. We have a $\Gm_{,S}$-action on $\PP^{1}_{S}$, where we act on the affine line around $0$ via multiplication and on the affine line around $\infty$ via multiplication with the inverse.  
Since $\Gm_{,S}\times_{S}\PP^{1}_{S}\cong \Gm_{,\FF_{p}}\times_{\FF_{p}}\PP^{1}_{S}$, we see that the $\Gm_{}$-action on $\PP^{1}_{S}$ as an $\FF_{p}$-scheme is the same as the action as an $S$-scheme.
Therefore, we  get an induced $\Gm_{,\FF_{p}}$-action on $X$.

We have the following pushout diagram: 
\begin{equation}
\label{eq.pushout}
\begin{tikzcd}
\Gm_{,\FF_{p}}\times_{\FF_{p}}\lbrace \infty\rbrace \amalg \lbrace 0 \rbrace \arrow[rr,"\id\times{(\id,\varphi)}"]\arrow[d,"\id\times\iota",hook] &&\Gm_{,\FF_{p}}\times_{\FF_{p}} \lbrace \infty\rbrace\arrow[d]\\
\Gm_{,\FF_{p}}\times_{\FF_{p}}\PP^{1}_{S}\arrow[rr]&&\Gm_{,\FF_{p}}\times_{\FF_{p}}X_{S}\rlap{,}
\end{tikzcd}
\end{equation}
which is induced by the $\Gm_{}$-actions. The following diagram with commutative squares shows that, indeed, $\Gm_{,\FF_{p}}\times_{\FF_{p}}X_{S}$ fulfils the universal property of the pushout above: 
$$
	\begin{tikzcd}
	&\lbrace \infty\rbrace \amalg \lbrace0\rbrace\arrow[rr,"{(\id,\varphi)}"]\arrow[d,"s"] && \lbrace \infty\rbrace\arrow[d,"s"]\\
	\lbrace \infty\rbrace \amalg \lbrace0\rbrace\arrow[d,"\iota"]\arrow[r,"s"]&\Gm_{,\FF_{p}}\times_{\FF_{p}} \lbrace \infty\rbrace \amalg \lbrace0\rbrace\arrow[rr,"{\id\times (\id,\varphi)}"]\arrow[d,"{\id\times \iota}"]&& \Gm_{,\FF_{p}}\times_{\FF_{p}} \lbrace \infty\rbrace\\
	\PP^{1}_{S}\arrow[r,"s"]&\Gm_{,\FF_{p}}\times_{\FF_{p}} \PP^1_S\rlap{,} &&
	\end{tikzcd}
$$
where the morphism $s$ denotes the natural $0$-section.
 
Now let us define $\Xfr_{S}$ as the quotient stack $[X_{S}/\Gm_{,\FF_{p}}]$. We claim that $\Xfr_{S}$ is an Artin stack, which follows from the following lemma and \cite[04TK]{stacks-project}.
\begin{lem}
\label{lem.action.smooth}
	The group action of $\Gm_{,\FF_{p}}$ on $X_{S}$ is smooth.
\end{lem}
\begin{proof}
	Since the action of $\Gm_{,\FF_{p}}$ on $X_{S}$ is induced by the pushout diagram~\ref{eq.pushout}, we see by\footnote{Locally, the pushout of schemes is given by the fiber product of the rings corresponding to affine opens, as one sees in the construction (see for example the proof of the existence given in \cite[0E25]{stacks-project}). So, in particular, we can apply the reference.} \cite[08KQ]{stacks-project} that the $\Gm_{,\FF_{p}}$-action is smooth if and only if the $\Gm_{,\FF_{p}}$-action on $\lbrace\infty\rbrace$ and $\PP^{1}_{S}$ is smooth. The smoothness of the former action is clear. For the latter, it is enough to see that the $\Gm_{,S}$-action on $\AA^{1}_{S}$ by multiplication of degree $1$ is smooth (the degree $-1$ case is completely similar).

	Indeed, the question is local, so we may assume that $S=\Spec(R)$ is affine. Then the ring map corresponding to the $\Gm_{,S}$-action is equivalently given by
        $$
        R[X]\longrightarrow R[X][Y_{1},Y_{2},Z]/(Y_{1}Y_{2}-1,ZY_{1}-X).
        $$
	The Jacobi matrix corresponding to this map is given by
	$$
	\begin{pmatrix}
	Y_{2}&Y_{1} & 0\\
	Z & 0 & Y_{1}
	\end{pmatrix}.
	$$
	For any point in $\pfr\in\Spec(R[X][Y_{1},Y_{2},Z]/(Y_{1}Y_{2}-1,ZY_{1}-X))$,  this matrix has full rank as $Y_{1}Y_{2}-1\in \pfr$. Therefore, we see that indeed the $\Gm_{,S}$-action $\AA^{1}_{S}$ via multiplication of degree $1$ is smooth of relative dimension $1$, concluding the proof.
\end{proof}

We will start by showing that vector bundles over $\Xfr$ are the same as $F$-zips. Then the comparison of $G$-zips and $G$-torsors on $\Xfr$ is a Tannaka duality-like statement.

\begin{thm}
\label{F-zips as VB}
Let $S$ be a $k$-scheme, and let $n\in\NN$. There is an equivalence of categories
$$
\textup{VB}_n(\Xfr_{S})\simeq \cl\FZip^n(S).
$$ 
\end{thm}
\begin{proof}
	By descent, a vector bundle on $\Xfr_{S}$ is a $\Gm_{,\FF_{p}}$-equivariant vector bundle on $X_{S}$. Since $X_{S}$ is the coequalizer of $\lbrace 0\rbrace$, where one of the morphisms is given by the identity and the other sends $\lbrace 0\rbrace$ to $\lbrace \infty\rbrace$ and twists by Frobenius, we see that a vector bundle on $X$ is given by a finite locally free $\PP^1_S$-module with an isomorphism after pullback to $\lbrace \infty\rbrace$, respectively $\lbrace0\rbrace$, with Frobenius twist (see \cite[Corollary 6.5]{pushout}). A vector bundle of rank $n$ on $\PP^{1}_{S}$ is given by a pair $(V,W)$ of vector bundles of rank $n$ on $\AA^{1}_{S}$ such that $V_{|D(0)}\cong W_{|D(\infty)}$. Let $(V,W)$ be such a pair. Further assume that we have an isomorphism $W_{|\lbrace \infty\rbrace}^{(1)}\rightarrow V_{|\lbrace 0\rbrace}$; this defines a vector bundle on $X_{S}$. The $\Gm_{,\FF_{p}}$-equivariance induces  gradings on $V$ and $W$ (this is for example explained in \cite[03LE]{stacks-project}). Since $\Gm_{,\FF_{p}}$ acts on the affine line around $\lbrace 0\rbrace$ by multiplying in degree $1$ and on the affine line around $\lbrace \infty\rbrace$ by multiplying in degree $-1$, we see that by definition of the grading, the corresponding endomorphisms of $V$ and $W$ seen as $k$-vector spaces are morphisms of graded vector spaces in degree $1$ and $-1$, respectively. This construction gives an ascending chain 
	$$
	\cdots\longrightarrow V_{i-1}\longrightarrow V_{i}\longrightarrow V_{i+1}\longrightarrow \cdots
	$$ 
	and a descending chain
	$$
	\cdots\longrightarrow W^{i+1}\longrightarrow W^{i}\longrightarrow W^{i-1}\longrightarrow \cdots.
	$$
	The pullback to $\lbrace 0\rbrace$ gives the direct sum of graded pieces, and considering the Frobenius twist, we see that this isomorphism is the same as the datum of a family of isomorphisms $\varphi\colon (\gr^i_W)^{(1)}\rightarrow \gr^i_V$ for each $i$. Moreover, since the pullback to $\lbrace 0\rbrace$ also has to be a vector bundle, we see that each graded piece has to be finite locally free (since the direct sum of all has to be finite locally free of rank $n$). This in particular shows that the filtrations stabilize; \textit{i.e.}\ there are  only finitely many non-zero graded pieces. The pullbacks of these filtrations to $D(0)$ and $D(\infty)$, respectively, give the underlying modules of the filtrations, which have to be isomorphic;  \textit{i.e.}\ $V_\bullet$ and $W^\bullet$ define, respectively, an ascending and a descending filtration on the same module.
 
	Putting all these data together, we see that the category of vector bundles on $\Xfr_{S}$ of rank $n$ is equivalent to the category of $F$-zips over $S$ of rank $n$.
\end{proof}

\begin{rem}
  Theorem~\ref{F-zips as VB} shows. in particular, that the functor
  $$
  \textup{VB}(\Xfr)\colon S\longmapsto \textup{VB}(\Xfr_{S})
  $$
  is a sheaf for the fppf topology and even an Artin stack.
\end{rem}

\begin{cor}
\label{fzip is vb}
There is an equivalence of Artin stacks
$$
\textup{VB}(\Xfr)\simeq \cl\FZip_{k}.
$$
Further, for any scheme $S$, we have a monoidal equivalence of symmetric monoidal categories $\textup{VB}(\Xfr_{S})\simeq \cl\FZip_{k}(S)$. 
\end{cor}
\begin{proof}
	This is clear from Theorem~\ref{F-zips as VB}. Note that the proof of Theorem~\ref{F-zips as VB} shows that this equivalence respects the symmetric monoidal structures.
\end{proof}

\begin{cor}
	\label{comp G-zip}
	Let $G$ be a linear algebraic group over $k$ such that its identity component is reductive. 
	We have an isomorphism of Artin stacks
	$$
	G\textup{-Tors}(\Xfr)\cong G\textup{-zip}_k.
	$$	
\end{cor}
\begin{proof}
  Using Theorem~\ref{F-zips as VB}, we deduce from \cite[Theorem 7.13]{PWZ} that it is enough to show that $G\textup{-Tors}(\Xfr)$ is equivalent to the stack of fiber functors from the symmetric monoidal category of $G$ representations to the symmetric monoidal category $\textup{VB}(\Xfr)$, which we denote by $\Hom^{\otimes}(\Rep(G),\textup{VB}(\Xfr))$.\footnote{The associated pre-stack  is given by the assignment $S\mapsto \Hom(\Rep(G),\textup{VB}(\Xfr_{S}))$. That this is a stack follows from \cite[Section~III.3.2.1.2]{Rivano}.
  In general, for any symmetric monoidal categories $A$ and $B$, we denote by $\Hom^{\otimes}(A,B)$ the groupoid of fiber functors from $A$ to $B$.} 

  Let $S$ be a $k$-scheme and $\Xfr^{\bullet}_{S}$ denote the \v{C}ech nerve of $\Xfr_{S}$ (which by construction is termwise given by a scheme). Since $\BG$ is a sheaf for the fppf topology, we get by definition $G\textup{-Tors}(\Xfr_{S})=\Hom_{P(\Schrel{k})}(\Xfr_{S}, \textup{B}G)=\lim_{n\in\NN}\Hom_{P(\Schrel{k})}(\Xfr^{n}_{S},\textup{B}G)$. By Tannaka duality (see \cite[Theorem 2.3]{Ziegler}), we have
  $$
  \lim_{n\in\NN}\Hom_{P(\Schrel{k})}(\Xfr^{n}_{S},\textup{B}G) = \lim_{n\in\NN}\Hom^{\otimes}\left(\Rep(G),\textup{VB}\left(\Xfr_{S}^{n}\right)\right). 
  $$
  Here we embed the categories involved, which are $(2,1)$-categories, naturally into the world of $\infty$-categories (for example via the Duskin nerve, 
  which is explained in \cite[009P]{kerodon}); 
  then the limit above is the usual limit in the $\infty$-categorical sense. Since the pre-stack of fiber functors $\Hom^{\otimes}(\Rep(G),\textup{VB}(\Xfr))$ satisfies fpqc decent (see \cite[Section~III.3.2.1.2]{Rivano};
  in fact, one uses  this reference together with \cite[Theorem 2.1]{Ziegler} to see that the stack of fiber functors defines a gerbe, which is needed in \cite[Theorem 2.3]{Ziegler}), we have
  $$ \lim_{n\in\NN}\Hom^{\otimes}(\Rep(G),\textup{VB}(\Xfr_{S}^{n})) = \Hom^{\otimes}(\Rep(G),\textup{VB}(\Xfr_{S})),$$
  concluding the proof. 
\end{proof}


\end{document}